\documentclass[12pt]{article}

\usepackage{authblk}
\usepackage[english]{babel}
\usepackage{csquotes}
\usepackage{amsfonts,amsmath,amsthm}
\usepackage{empheq}
\usepackage{cancel}
\usepackage{bbm}
\usepackage[titletoc,title]{appendix}
\usepackage[backend=bibtex8,doi=false,eprint=false,giveninits=true,isbn=false,style=numeric-comp,url=false,maxnames=99]{biblatex}
\makeatletter
\def\blx@maxline{77}
\makeatother
\DeclareRedundantLanguages{english,german,french}{english,german,ngerman,french}
\usepackage{cases}
\usepackage{mathabx}
\usepackage{xfrac}
\usepackage{fancyhdr}
\usepackage{color}
\usepackage[colorlinks]{hyperref}
\definecolor{blue75}{rgb}{0,0,.75}
\definecolor{green75}{rgb}{0,.75,0}

\usepackage[a4paper, left=2.3cm, right=2.2cm, top=2.5cm,bottom=2.5cm]{geometry}
\numberwithin{equation}{section}
\newtheorem{theorem}{Theorem}[section]
\newtheorem{lemma}{Lemma}[section]

\newtheorem{remark}{Remark}[section]

\newtheorem{proposition}{Proposition}

\newtheorem{rule-def}[theorem]{Rule}
\theoremstyle{definition}

\newcommand{\cb}{\color{blue}}

\newcommand{\R}{{\mathbb R}}

\newcommand{\eps}{\varepsilon}
\newcommand{\Om}{\Omega }
\newcommand{\Sbb}{\mathbb S^{n-1}}

\usepackage{constants}
\newconstantfamily{C}{
	symbol=C,
	format=\parenthezises,
	reset={section}
}
\usepackage{enumerate}

\usepackage{graphicx}
\graphicspath{ {images/} }
\usepackage{wrapfig}
\usepackage{figbib}
\usepackage{caption}
\usepackage{subcaption}
\usepackage{changes}
\allowdisplaybreaks

\makeatletter
\date{\today} 
\begin{document}
	\title{On a mathematical model for tissue regeneration}
\author{Shimi Chettiparambil Mohanan, Nishith Mohan, and Christina Surulescu\thanks{\href{mailto:surulescu@mathematik.uni-kl.de}{surulescu@mathematik.uni-kl.de}}\\
RPTU Kaiserslautern-Landau, Department of Mathematics,\\ Postfach 3049, 67653 Kaiserslautern, 
Germany	}
	\maketitle{}
	
	\begin{abstract}
		\noindent
	We propose a PDE-ODE model for tissue regeneration, obtained by parabolic upscaling from kinetic transport equations written for the mesoscopic densities of mesenchymal stem cells and chondrocytes which evolve in an artificial scaffold impregnated with hyaluron. Due to the simple chosen turning kernels, the effective equations obtained on the macroscopic level are of the usual reaction-diffusion-taxis type. We prove global existence of solutions to the coupled macroscopic system and perform a stability and bifurcation analysis, which shows that the observed patterns  are driven by taxis. Numerical simulations illustrate the model behavior for various tactic sensitivities and initial conditions.
	\end{abstract}
	
	\section{Introduction}\label{sec:intro}
	
	With the advent of treatment paradigms speaking in favor of repairing and regenerating rather than solely surgical resection of tissues \cite{Athanasiou2009,Guo2015,Orlando2012,Pillai2018,Ude2014,Yu2015,Zreiqat2015}, mathematical modeling of regeneration processes in the context of e.g., wound healing, bone and meniscus regeneration has received increasing attention. Here we are primarily interested in the latter application and refer e.g. to \cite{Borgiani2017,WSE-H} and references therein for comprehensive reviews concerning such or closely related models; see also the model classification in \cite{grosjean_et_all}.\\[-2ex]
	
	\noindent
	Mesenchymal stem cells have been recognized as a source of cells which can be induced to differentiate into tissue repair cells like (fibro)chondrocytes \cite{Yu2015,Zellner2010}. Among these, adipose derived stem cells (ADSCs) are obtained from perivascular white adipose tissue and are relatively easy to isolate; moreover, they produce a higher yield of cells when compared to other adult stem cell sources, see e.g. \cite{Johnstone2013}. We will develop here a mathematical model to characterize the dynamics of ADSC and chondrocyte densities under the influence of extracellular matrix (ECM) including newly formed tissue and hyaluron impregnating a non-resorbable scaffold in which the cells migrate and proliferate. Unlike \cite{grosjean_et_all,Grosjean2023}, our model does not take into account fluid mechanics and tissue deformation:  we are rather interested in the patterns generated by the dynamics of involved cell populations under the said influences.   The scaffold's anisotropic structure is not addressed here - but see again \cite{grosjean_et_all}, which provides a detailed account of it. Our model here can be seen as a simplified module of that complex, more realistic setting. We also address here  the well-posedness of this simplification and prove the global existence of  solutions therefor.\\[-2ex]
	
	\noindent
	The paper is organized as follows: Section \ref{sec:model-deduction} is concerned with obtaining (in an informal manner) a macroscopic formulation of cell and hyaluron \& ECM dynamics from decriptions on a lower (mesoscopic) scale, thus enabling us to provide some effective equations of reaction-diffusion-transport type, for which we perform in Section \ref{sec:analysis} a stability analysis of patterns and bifurcations, along with a global existence proof. Section \ref{sec:numerics} shows 1D and 2D numerical simulations of the PDE-ODE system obtained in Section \ref{sec:model-deduction} under various scenarios.
	
	\section{Model deduction}\label{sec:model-deduction}
	
\subsection{Mesoscopic level}
We use the kinetic theory of active particles developed by Bellomo et al. \cite{Bell-KTAP} and describe the dynamics of cell density distributions for ADSCs and chondrocytes by way of kinetic transport equations (KTEs). Let $p_1(t,x,v)$ denote the density of ADSCs sharing at time $t>0$ and position $x\in \R^n$ the velocity regime $v\in V_1=s_1\Sbb$. Likewise, $p_2(t,x,v)$ represents the density of chondrocytes with velocity $v\in V_2 =s_2 \Sbb$. The positive constants $s_1$ and $s_2$ are the speeds of ADSCs and chondrocytes, respectively. The KTEs for $p_1$ and $p_2$ then write
\begin{align}
&\partial_tp_1+\nabla _x\cdot (vp_1)=\mathcal L_1[\lambda_1 (v,h)]p_1+\text{(de)differentiation \& proliferation}\label{KTE:ADSCs}          \\
&\partial_tp_2+\nabla _x\cdot (v p_2)=\mathcal L_2[\lambda_2]p_2+\text{(de)differentiation}.\label{KTE-chondro}
\end{align}
The first terms on the right hand sides of \eqref{KTE:ADSCs} and \eqref{KTE-chondro} characterize the reorientation of ADSCs and chondrocytes, respectively. The turning operators depend on the turning rates $\lambda _1$ and $\lambda_2$, whereby the former also depends on the pathwise gradient of some ligand concentration $h$, while the latter is constant:
\begin{align}
\mathcal{L} _1[\lambda_1(v,h)]p_1(t,x,v)&:=- \int_{V_1}\lambda_1(v,h)K_1(v',v)p_1(t,x,v)dv'+\int _{V_1}\lambda_1(v',h)  K_1(v,v')p(t,x,v') dv'\notag \\
&=-\lambda (v,h)p_1(t,x,v)+\int _{V_1}\lambda_1(v',h)  K_1(v,v')p(t,x,v') dv'\label{turnop-1}\\
\mathcal{L} _2[\lambda _2]p_2(t,x,v)&:=-\lambda_2p_2(t,x,v)+\lambda_2\int_{V_2}K_2(v,v')p_2(t,x,v') dv'.\label{turnop-2}
\end{align}	
In our model $h$ represents the volume fraction of an attractant ligand (hyaluron) which impregnates the fibers of an artifical scaffold on which the cells are supposed to migrate and spread. As such, it cannot diffuse or be transported, but it can be uptaken by chondrocytes and also produced by them in a limited manner. As our model does not specifically involve dynamics of newly produced extracellular matrix (ECM), nor resorption of the artificial scaffold, we lump all tissue and hyaluron volume fractions in the macroscopic variable $h$. These assumptions lead for $h$ to the degenerate partial differential equation 
\begin{align}\label{eq:h}
\partial _th=-\gamma_1hc_2+\gamma _2\frac{c_2}{K_{c_2}+c_2},
\end{align}
with $\gamma_1,\gamma_2>0$ constants.\\[-2ex]

\noindent
For the turning kernels $K_1, K_2$ we require as usual $\int _{V_i}K_i(v,v')dv=1$ ($i=1,2$). They can be used like e.g., in \cite{Conte2021,conte2023mathematical,corbin2021modeling,dietrich2022multiscale,Engwer2,engwer2015glioma,engwer2016effective,Hunt2016,kumar2021multiscale,grosjean_et_all} to account for a heterogeneous and even anisotropic environment, which is of particular relevance when a fibrous scaffold is used to support cell migration and proliferation. Here we adopt a simplified description and consider uniform turning kernels, i.e. $K_i(v,v')=1/|V_i|$ ($i=1,2$). For the turning rate of ADSCs we choose $\lambda _1(v,h)=\lambda_0\exp(-\phi D_th)$, where $D_th=h_t+v\cdot \nabla h$ represents the pathwise gradient of hyaluron concentration $h$ and $\lambda_0,\phi >0$ are constants. This choice is in line with previous works, e.g. \cite{Othmer2002,kumar2021multiscale} related to cell migration in response to gradients of diffusing or nondiffusing cues and it typically leads to chemotactic and haptotactic behavior, respectively. With these choices the turning operators from \eqref{turnop-1}, \eqref{turnop-2} become
\begin{align}
&\mathcal{L} _1[\lambda_1(v,h)]p_1(t,x,v)=\lambda_0e^{-\phi h_t}\left (-e^{-\phi v\cdot \nabla h}p_1(t,x,v)+\frac{1}{|V_1|}\int _{V_1}e^{-\phi v'\cdot \nabla h}p_1(t,x,v')dv'\right )\notag \\
&\mathcal{L} _1[\lambda_2]p_2(t,x,v)=-\lambda_2p_2+\lambda_2\frac{1}{|V_2|}c_2=\lambda_2(\frac{1}{s_2^{n-1}|\Sbb |}c_2(t,x)-p_2(t,x,v)),
\end{align}
where $c_i(t,x):=\int _{V_i}p_i(t,x,v)dv$ ($i=1,2$) denote the macroscopic densities of ADSCs and chondrocytes, respectively. For small enough values of $\phi$ we get for $\mathcal{L} _1$ the approximation
\begin{align}
\mathcal{L} _1[\lambda_1(v,h)]p_1(t,x,v)\simeq -\lambda_0\left (1-\phi (h_t+v\cdot \nabla h)\right )p_1+\frac{\lambda_0}{|V_1|}\left (c_1-\phi \left (h_tc_1+M_1\cdot \nabla h\right )\right ),
\end{align}
where $M_1(t,x):=\int _{V_1}v p_1(t,x,v)dv$ denotes the first moment of the ADSC orientation distribution. \\[-2ex]

\noindent
Similarly to \cite{Conte2021,corbin2021modeling,engwer2016effective,grosjean_et_all} we consider the following source terms:
\begin{align*}
\mathcal S_1p_1&:=-\alpha p_1+\delta p_2+\beta p_1(1-\frac{c_1}{K_{c_1}}-\frac{c_2}{K_{c_2}}), \\
\mathcal S_2p_2&:=\alpha p_1-\delta p_2,
\end{align*}
with $\alpha >0$ denoting the differentiation rate of ADSCs to chondrocytes and $\delta >0$ the dedifferentiation rate of chondrocytes to ADSCs. The constant $\beta >0$ represents the proliferation rate of ADSCs (chondrocytes are assumed not to proliferate during the time considered here) and $K_{c_1}, K_{c_2}>0$ represent the carrying capacities of the two cell populations. 

\subsection{Parabolic upscaling}

\noindent
We perform a parabolic scaling of the KTEs \eqref{KTE:ADSCs}, \eqref{KTE-chondro}, i.e. we rescale the time and space variables by $\tilde t:=\eps^2t$, $\tilde x:=\eps x$. Since proliferation is much slower than migration, we also rescale as in \cite{Conte2021,corbin2021modeling,engwer2016effective,grosjean_et_all} with $\eps ^2$ the corresponding terms $\mathcal S_1p_1$, $\mathcal S_2p_2$. For notation simplification we will drop in the following the $\sim $ symbol from the scaled variables $t$ and $x$. Thus, we get:
\begin{align*}
\eps^2\partial_tp_1+\eps \nabla \cdot (vp_1)&=-\lambda_0\left (1-\phi (\eps^2h_t+\eps v\cdot \nabla h)\right )p_1+\frac{\lambda_0}{|V_1|}\left (c_1-\phi \left (\eps^2h_tc_1+\eps M_1\cdot \nabla h\right )\right )\notag \\
+&\eps ^2\left (-\alpha p_1+\delta p_2+\beta p_1(1-\frac{c_1}{K_{c_1}}-\frac{c_2}{K_{c_2}})\right )\\
\eps^2\partial_tp_2+\eps \nabla \cdot (vp_2)&=\lambda_2(\frac{1}{|V_2|}c_2-p_2)+\eps ^2\left (\alpha p_1-\delta p_2\right ).
\end{align*}
In the sequel we consider $p_i$ ($i=1,2$) to be compactly supported in the phase space $\R^n\times V_i$. Using Hilbert expansions of $p_i$ and identifying equal powers of $\eps$ we obtain:\\[-2ex]

\noindent
$\eps ^0$: 
\begin{align}
&p_{1,0}=\frac{1}{s_1^{n-1}|\Sbb |}c_{1,0}\label{ADSC:p1-c1}\\
&p_{2,0}=\frac{1}{s_2^{n-1}|\Sbb |}c_{2,0};\label{chondro:p2-c2}
\end{align}

\noindent
$\eps ^1$: 
\begin{align}
v\cdot \nabla p_{1,0}&=-\lambda_0(p_{1,1}-\phi v\cdot \nabla hp_{1,0})+\frac{\lambda_0}{|V_1|}\left (c_{1,1}-\phi M_{1,0}\cdot \nabla h\right )\notag \\
&=-\lambda_0(p_{1,1}-\phi v\cdot \nabla hp_{1,0})+\frac{\lambda_0}{|V_1|}c_{1,1}\quad \text{(by \eqref{ADSC:p1-c1})}\label{eq:eps1-p1}\\ 
v\cdot \nabla p_{2,0}&=\lambda_2(\frac{1}{|V_2|}c_{2,1}-p_{2,1}).\label{eq:eps1-p2}
\end{align}
With the notation $\mathcal L[\lambda_0]p:=\lambda_0\left (\frac{1}{|V_1|}\int _{V_1}p(v)dv-p(v)\right)$ we can rewrite \eqref{eq:eps1-p1} as
\begin{align}
\mathcal L[\lambda_0]p_{1,1}=-\lambda_0\phi \frac{1}{|V_1|}c_{1,0}v\cdot \nabla h+\nabla \cdot (v\frac{c_{1,0}}{|V_1|}).
\end{align}
Since the integral with respect to $v$ of the right hand side in the above equation vanishes, we can invert $\mathcal L[\lambda_0]$ to obtain
\begin{align}
p_{1,1}=-\frac{1}{\lambda_0}\nabla \cdot (v\frac{c_{1,0}}{|V_1|})+\phi \frac{1}{|V_1|}c_{1,0}v\cdot \nabla h.\label{eq:p11}		
\end{align}
Analogously, from \eqref{eq:eps1-p2} we obtain
\begin{align}
p_{2,1}=-\frac{1}{\lambda_2}\nabla \cdot \left (v \frac{c_{2,0}}{|V_2|}\right ).\label{eq:p21}
\end{align}

\noindent
$\eps ^2$: 
\begin{align}
\partial_tp_{1,0}+\nabla \cdot (vp_{1,1})&=-\lambda_0\left (p_{1,2}-\frac{c_{1,2}}{|V_1|}\right)\notag \\
+&\lambda _0\phi \left (h_t(p_{1,0}-\frac{c_{1,0}}{|V_1|})+v\cdot \nabla hp_{1,1}-\frac{1}{|V_1|}M_{1,1}\cdot \nabla h\right )\notag \\
-&\alpha p_{1,0}+\delta p_{2,0}+\beta p_{1,0}\left (1-\frac{c_{1,0}}{K_{c_1}}-\frac{c_{2,0}}{K_{c_2}}\right )\label{eq:eps2-p1}\\
\partial_tp_{2,0}+\nabla \cdot (vp_{2,1})&=\lambda_2(\frac{1}{|V_2|}c_{2,2}-p_{2,2})+\alpha p_{1,0}-\delta p_{2,0}.\label{eq:eps2-p2}
\end{align}
Integrating \eqref{eq:eps2-p1} with respect to $v$ and using \eqref{eq:p11}, we obtain
\begin{align}\label{eq:c10}
\partial_tc_{1,0}-\nabla \cdot \left (\frac{s_1^2}{n\lambda _0}\nabla c_{1,0}\right )+\nabla \cdot \left (\frac{s_1^2\phi }{n}c_{1,0}\nabla h\right )=-\alpha c_{1,0}+\delta c_{2,0}+\beta c_{1,0}\left (1-\frac{c_{1,0}}{K_{c_1}}-\frac{c_{2,0}}{K_{c_2}}\right )
\end{align}
\noindent
Likewise, integrating \eqref{eq:eps2-p2} with respect to $v$ and using \eqref{eq:p21} leads to 
\begin{align}\label{eq:c20}
\partial_tc_{2,0}-\nabla \cdot \left (\frac{s_2^2}{n\lambda _2}\nabla c_{2,0}\right )=\alpha c_{1,0}-\delta c_{2,0}.
\end{align}
The obtained equations \eqref{eq:c10}, \eqref{eq:c20} describe the macroscopic dynamics of ADSC and chondrocyte densities, respectively, at leading order w.r.t. $\eps\ll 1$. They are supplemented with the macroscopic equation \eqref{eq:h} for the dynamics of hyaluron/ECM density $h$.\\[-2ex]

\noindent
The above macroscopic system for ADSC and chondrocyte densities has been deduced upon considering $x\in \R^n$ and, together with \eqref{eq:h}, it has to be supplemented with appropriate initial conditions. Subsequently we will consider the dynamics to take place in a bounded, sufficiently regular domain $\Om \subset \R^n$ and endow it with no-flux boundary conditions, which can be obtained in a similar way to that presented in \cite{plaza,corbin2021modeling,dietrich2022multiscale}. We emphasize the fact that the obtained motility coefficients - although occurring in a macroscopic setting- depend on parameters originating on a lower (mesoscopic) scale: cell speeds $s_1, s_2$, turning rates $\lambda_0, \lambda _2$, and orientation bias $\phi $ towards the gradient of $h$.

	\section{Analysis of the macroscopic model}\label{sec:analysis}
Choosing $a_1:=\frac{s_1^2}{n\lambda _0}$, $a_2:=\frac{s_2^2}{n\lambda _2}$, and $b:=\frac{s_1^2\phi }{n}$ we consider the following macroscopic model for the dynamics of ADSCs, chondrocytes, and hyaluron/ECM:
	\begin{equation}
		\begin{aligned}
			\label{model}
			\partial_t c_1 &= a_1 \Delta c_1 - \nabla \cdot (b c_1 \nabla h ) - \alpha c_1 + \delta c_2+ \beta c_1 ( 1 - \frac{c_1}{K_{c_1}} - \frac{c_2}{K_{c_1}}),       ~~~~~~ & \text{in} ~~ (0,T) \times \Omega,                  \\[5pt]
			\partial_t c_2 &= a_2 \Delta c_2 + \alpha c_1 - \delta c_2,                                                        ~~~~~~  &\text{in} ~~ (0,T) \times\Omega,                  \\[5pt]
			\partial_t h &=-\gamma_1 hc_2+\gamma_2 \dfrac{c_2}{K_{c_2}+c_2},                                                       ~~~~~~  &\text{in} ~~ (0,T) \times\Omega,                  \\[5pt]
			\partial_\nu c_1 &= \partial_\nu c_2 = \partial_\nu h = 0,                                                         ~~~~~~  &\text{on} ~~ (0,T) \times\partial \Omega,         \\[5pt]
			c_1(0) &= c_1^0, \ c_2(0) = c_2^0, \ h(0) = h_0& \text{in} ~~ \Omega, 
		\end{aligned}
	\end{equation}
	\\
	where $a_1,\ a_2, \ \alpha, \ \beta,\ \gamma,\ \delta, \ b > 0 $ are all positive constants and $\Om\subset \R^n$ ($n=2,3$) is a bounded domain with sufficiently regular boundary, while $\nu $ represents the outer unit normal on $\partial \Om$.
	

	
	\subsection{Global existence}
	We consider model \eqref{model} with $x\in \Omega $ and $t>0$. 
 We also assume that the initial data $(c_1^0, c_2^0, h_0)$ satisfy for some $\omega \in (0, 1)$ 
	\begin{equation}
		\begin{cases}
			\label{assu1}
			c_1^0(x) \geq 0, \ c_2^0(x) \geq 0, \  h_0(x) \geq 0,                                                                                        \\[5pt]
			c_1^0 \in C^{2 + \omega}(\bar{\Omega}),  \ c_2^0(x) \in C^{2 + \omega}(\bar{\Omega}), \ h_0(x) \in C^{2 + \omega}(\bar{\Omega}),             \\[5pt]
			\dfrac{\partial c_1^0}{\partial \nu} = \dfrac{\partial c_2^0}{\partial \nu} = \dfrac{\partial h_0}{\partial \nu} = 0 ~\text{on}~ \partial \Omega.
		\end{cases}
	\end{equation}
	The main result regarding global solvability in a 2D spatial domain is given as follows:
	\begin{theorem}
		\label{mth}
		Let $\Omega \subset \mathbb{R}^2$ be a bounded domain with smooth boundary, assume that $a_1, a_2, b, \alpha, \delta, \beta, \gamma_1$ and $\gamma_2$ are positive. Then for any $(c_1^0, c_2^0, h_0)$ satisfying \eqref{assu1} with some $\omega \in (0, 1)$, the problem \eqref{model}admits a unique global classical solution $(c_1, c_2, h) \in (C^{2, 1}(\bar{\Omega}) \times [0, \infty))^3$ for which $c_1, c_2$ and $h$ are nonnegative.
	\end{theorem}

	\subsubsection{Change of variable}
	
	Employing the strategy outlined in \cite{tao2009global,marciniak2010boundedness,tao2011global,tao2014energy,ke2018note,giesselmann2018existence,surulescu2021does}, we change the variables to transform the first equation of \eqref{model} in the divergence form. To this end, substitute
	\begin{equation}
		\label{cov1}
		z = c_1 e^{-\tfrac{b}{a_1} h}.
	\end{equation}
	Using \eqref{cov1} we can rewrite system \eqref{model} in the following form
	\begin{equation}
		\begin{cases}
			\label{sys1}
			\partial_t z = a_1e^{-\tfrac{b}{a_1} h} \nabla \cdot (e^{\tfrac{b}{a_1} h} \nabla z) - \alpha z + \delta c_2 e^{-\frac{b}{a_1} h}                                                                     \\[7pt]
			\hspace{3cm} + z \left(\beta\big(1 - \frac{1}{K_{c_1}}z e^{\tfrac{b}{a_1} h} - \frac{c_2}{K_{c_1}}\big) + \frac{b \gamma_1}{a_1} h c_2 - \frac{b\gamma_2 }{a_1}\frac{c_2}{K_{c_1} + c_2}\right), & x \in \Omega,\ t > 0, \\[7pt]
			\partial_t c_2 = a_2 \Delta c_2  - \delta c_2 + \alpha   z e^{\frac{b}{a_1} h}                                                          & x \in \Omega,\ t > 0,                                        \\[7pt]
			\partial_t h = -\gamma_1 h c_2 + \gamma_2\dfrac{c_2}{K_{c_1} + c_2},                                                                    & x \in \Omega,\ t > 0,                                        \\[7pt]
			\partial_\nu z = \partial_\nu c_2 = \partial_\nu h = 0,                                                                                 & x \in \partial \Omega,\ t > 0                                \\[7pt]
			z(0, x) = z_0(x) = e^{-\tfrac{b}{a_1} h_0} c_1^0, \ c_2(0, x) = c_2^0, \ h(0, x) = h_0,                                                 & x \in \Omega.
		\end{cases}
	\end{equation}
	It should be pointed out that \eqref{model} and \eqref{sys1} are equivalent within the concept of classical solutions \cite{surulescu2021does}.

	\subsubsection{Local existence and extensibility criterion}
	
In the sequel we will use the following notations and conventions:
\begin{itemize}
	\item $Q_T = \Omega \times (0, T)$;
	\item we abbreviate the integrals $\int_\Omega f(x) dx$ as $\int_\Omega f(x)$;
	\item the sequentiality of the constants $ C_i, i = 1, 2, 3, \ldots$ holds only within the lemma/theorem and its proof in which the constants are used. The sequence restarts once the proof is over;
	\item $W^{2,1}_p(Q_T):=\{u :Q_T\to \R\ :\ u,\ \nabla u,\ \nabla ^2 u,\  \partial_tu\in L^p(Q_T)\}$.
\end{itemize}
	\begin{lemma}
		\label{local_existence}
		Assume the initial data $(c_1^0, c_2^0, h_0)$ satisfy \eqref{assu1} with some $\omega \in (0, 1)$. Then, the problem \eqref{sys1} admits a unique classical solution
		\begin{equation}
			\begin{cases}
				\label{leloc1}
				z \in C^0(\bar{\Omega} \times [0, T_{\max})) \cap C^{2, 1}(\bar{\Omega} \times (0, T_{\max}))                  \\[5pt]
				c_2 \in C^0(\bar{\Omega} \times [0, T_{\max})) \cap C^{2, 1}(\bar{\Omega} \times (0, T_{\max}))                \\[5pt]
				h \in C^{2, 1}(\bar{\Omega} \times [0, T_{\max}))
			\end{cases}
		\end{equation}
		such that
		\begin{equation}
			\label{leloc2}
			\text{either}~ T_{\max} = \infty, \quad \text{or} \quad \limsup_{t \nearrow T_{\max}} \left\{\|z(\cdot, t)\|_{L^\infty(\Omega)} + \|h(\cdot, t)\|_{W^{1, 5}(\Omega)}\right\} = \infty.
		\end{equation}
		Moreover, we have $z > 0, c_2 > 0$ and $h > 0$ in $\bar{\Omega} \times (0, T_{\max})$.
	\end{lemma}
	
	\begin{proof}
		
		Local existence for problem \eqref{sys1} is established by a fixed-point argument. Consider the Banach space $X$ of functions $(z, h)$ with norm
		\begin{equation*}
			\|(z, h)\|_X = \|z\|_{C^{1, 0}(\bar{\cb \Om} \times [0,T])} + \|h\|_{C(0, T; W^{1, 5}(\Omega))} \quad (0 < T < 1)
		\end{equation*}
		and a closed subset $S$ given by
		\begin{equation*}
			S = \left\{(z, h) \in X\ :\ z, h \geq 0,\ z(x, 0) = z_0(x),\ h(x, 0) = h_0(x),\  \frac{\partial z}{\partial \nu} = 0,\ \|(z, h)\|_X \leq M\right\}
		\end{equation*}
		with
		\begin{equation}
			\label{M}
			M = 2 \|z_0\|_{C^1(\bar{\Omega})} + 4 \|h_0\|_{W^{1, 5}(\Omega)} + 2.
		\end{equation}
		For any $(z, h) \in S$, we define a corresponding function $(\bar{z}, \bar{h}) \equiv F(z, h)$ where $(\bar{z}, \bar{h})$, along with $c_2$, satisfies the following decoupled problems
		\begin{equation}
			\label{dcop1}
			\begin{cases}
				\partial_t c_2 =  a_2 \Delta c_2 - \delta c_2 + \alpha  z e^{\frac{b}{a_1} h},         & \quad (x, t) \in Q_T,                             \\[5pt]
				\frac{\partial c_2}{\partial \nu} = 0,                                                 & \quad x \in \partial \Omega, 0 < t < T,           \\[5pt]
				c_2(x, 0) = c_2^0(x),                                                                  & \quad x \in \Omega;
			\end{cases}
		\end{equation}
		\begin{equation}
			\label{dcop2}
			\begin{cases}
				\partial_t \bar{h} = - \gamma_1 c_2 \bar{h} + \gamma_2 \frac{c_2}{K_{c_1} + c_2},                & \quad (x, t) \in Q_T,                             \\[5pt]
				\bar{h}(x, 0) = h_0(x),                                                     & \quad x \in \Omega;
			\end{cases}
		\end{equation}
		and
		\begin{equation}
			\label{dcop3}
			\begin{cases}
				\partial_t \bar{z} - a_1 \Delta \bar{z} - b \nabla h \cdot \nabla \bar{z} + g(x, t) \bar{z} = \frac{b \gamma_1}{a_1} z h c_2 + \delta c_2 e^{- \frac{b}{a_1} h} ,       & \quad (x, t) \in Q_T,                             \\[5pt]
				\frac{\partial \bar{z}}{\partial \nu} = 0,                                                                                                                       & \quad x \in \partial \Omega, 0 < t < T,           \\[5pt]
				\bar{z}(x, 0) = z_0(x),                                                                                                                                          & \quad x \in \Omega,
			\end{cases}
		\end{equation}
		where $g(x, t) = \alpha + \frac{b \gamma_2}{a_1} \frac{c_2}{K_{c_1} + c_2} - \beta\big(1 - \frac{1}{K_{c_1}}z e^{\tfrac{b}{a_1} h} - \frac{c_2}{K_{c_1}}\big)$.
		As $(z, h) \in S$, by the parabolic $L^p$-theory \cite[Theorem 2.1]{ladyzenskaja1968linear}, \eqref{dcop1} admits a unique solution $c_2$ satisfying
		\begin{equation}
			\label{c2s1}
			\|c_2\|_{W_5^{2, 1}(Q_T)} \leq C_1(M).
		\end{equation}
		By the Sobolev embedding $W_p^{2, 1}(Q_T) \hookrightarrow C^{1 + \omega, \frac{1 + \omega}{2}}(\bar{Q}_T)\ (p > 4,\ \omega = 1 - \frac{4}{p})$  \cite[Lemma II.3.3]{ladyzenskaja1968linear}, we can directly have from \eqref{c2s1} that
		\begin{equation}
			\label{c2s2}
			\|c_2\|_{C^{\frac{6}{5}, \frac{3}{5}}(\bar{Q}_T)} \leq C_2(M).
		\end{equation}
		Moreover, as $z \geq 0$ we can apply the parabolic comparison principle to \eqref{dcop1} and have
		\begin{equation}
			\label{c2s3}
			c_2 \geq 0.
		\end{equation}
		We will now turn our attention to the ODE (\ref{dcop2}), which is explicitly solvable in $Q_T$, with the unique solution being
		\begin{equation}
			\label{hsol}
			\bar{h}(x, t) = h_0(x) \exp\left(-\int_0^t \gamma_1c_2 ds\right) + \exp\left(-\int_0^t \gamma_1c_2 ds\right) \int_0^t \left\{\exp\left(\int_0^s \gamma_1 c_2 dw\right) \dfrac{\gamma_2 c_2}{K_{c_1} + c_2} \right\}ds \geq 0,
		\end{equation}
		\begin{align}
			\label{nabh}
			\nonumber \nabla  \bar{h}(x, t) = & \nabla h_0(x) \exp\left(- \int_0^t \gamma_1 c_2 ds\right) - \gamma_1 h_0(x) \exp\left(- \int_0^t \gamma_1 c_2 ds\right) \int_0^t \nabla c_2 ~ ds                                   \\[7pt]
			\nonumber & - \gamma_1 \exp\left(- \int_0^t \gamma_1 c_2 ds\right) \int_0^t \nabla c_2 ~ ds \int_0^t \left\{\exp\left(\int_0^s \gamma_1 c_2 dw\right) \dfrac{\gamma_2 c_2}{K_{c_1} + c_2} \right\}ds    \\[7pt]
			\nonumber & + \bigg\{ \int_0^t \bigg[\gamma_1 \exp\left(\int_0^s \gamma_1 c_2 dw\right) \left(\int_0^s \nabla c_2  ~ ds \right) \frac{\gamma_2 c_2}{K_{c_1} + c_2}  \\[7pt] & + \exp\left(\int_0^s \gamma_1 c_2 dw\right) \frac{\gamma_2 K_{c_1}\nabla c_2}{(K_{c_1} + c_2)^2}\bigg]ds \bigg\}
			\times  \exp\left(- \int_0^t \gamma_1 c_2 ds\right).
		\end{align}
		Note that $W^{1, 5}(\Omega) \hookrightarrow C(\bar{\Omega})$, hence $\|c_2\|_{C(\bar{Q}_T)} \leq C_3(M)$ and $\|\nabla c_2(\cdot, t)\|_{L^5(\Omega)} \leq C_4(M)$. Hence, from \eqref{hsol} and \eqref{nabh} we can have
		\begin{align}
			\label{hbound}
			\|\bar{h}(\cdot, t)\|_{L^5(\Omega)} & \leq
			\|h_0\|_{L^5(\Omega)}+\frac{\gamma _2|\Om |^{1/5}}{\gamma _1}C_3(M)\left (e^{\gamma_1 C_3(M)T} - 1\right )\notag \\
			&\le \|h_0\|_{L^5(\Omega)}+\frac{\gamma _2|\Om |^{1/5}}{\gamma _1}C_3(M)e^{\gamma_1 C_3(M)T},
		\end{align}
		\begin{align}
			\label{nhbound}
			 \|\nabla \bar{h}(\cdot, t)\|_{L^5(\Omega)}  \leq &\|\nabla h_0\|_{L^5(\Omega)} 
			+ \gamma_1 \|h_0\|_{L^\infty(\Omega)}|\Om |^{1/5}C_4(M)T  \notag \\                                      
			&+\gamma_2|\Om |^{1/5}C_4(M)Te^{\gamma _1C_3(M)T}\notag \\
			&+\frac{C_4(M)}{C_3(M)}\frac{\gamma_2}{\gamma_1}|\Om |^{1/5}\left (1+T\gamma_1C_3(M)e^{\gamma _1C_3(M)T}\right )\notag \\
			&+ \frac{\gamma_2}{\gamma _1}|\Om |^{1/5}\frac{C_4(M)}{C_3(M)}\frac{K_{c_1}}{(K_{c_1}+C_3(M))^2}e^{\gamma_1C_3(M)T}.
		\end{align}
		From \eqref{hbound} and \eqref{nhbound} we can have
		\begin{align}
			\label{hWbound}
			\nonumber \|\bar{h}(\cdot, t)\|_{W^{1, 5}(\Omega)} \leq & ~\|h_0\|_{W^{1, 5}(\Omega)} + \gamma_1 \|h_0\|_{L^\infty(\Omega)}C_4(M) T +\frac{\gamma_2}{\gamma_1}|\Om |^{1/5}C_3(M)e^{\gamma_1C_3T}                                 \\[5pt]
			\nonumber  
			&+\frac{C_4(M)}{C_3(M)}\frac{\gamma_2}{\gamma_1}|\Om |^{1/5}\left [1+\left (\gamma_1C_3(M)T+\frac{K_{c_1}}{(K_{c_1}+C_3(M))^2}\right )e^{\gamma_1C_3(M)T}\right ]\notag \\
			\leq & ~ \frac{M}{2}
		\end{align}
		if we take $T$ sufficiently small. 
		
\noindent
		We will now consider the parabolic problem \eqref{dcop3}. From \eqref{c2s1}, \eqref{c2s2}, \eqref{c2s3}, and $(z, h) \in S$, we can conclude that
		\begin{equation}
			\label{zeq1}
			\|b \nabla h(\cdot, t)\|_{L^5(\Omega)} \leq C_5(M), \quad \|g(x, t)\|_{L^5(\Omega)} \leq C_6(M), \quad \| \tfrac{b \gamma_1}{a_1} z h c_2 + \delta c_2 e^{- \frac{b}{a_1} h} \|_{L^5(\Omega)}  \leq C_7(M).
		\end{equation}
		Applying the parabolic $L^p$-theory \cite[Theorem 2.1]{ladyzenskaja1968linear} to \eqref{dcop3} while taking into account the inequalities in \eqref{zeq1}, we obtain
		\begin{equation}
			\label{zeq2}
			\|\bar{z}\|_{W^{2, 1}_5(Q_T)} \leq C_8(M).
		\end{equation}
		This, in conjunction with the Sobolev embedding \cite[Lemma II.3.3]{ladyzenskaja1968linear}, results in
		\begin{equation}
			\label{zeq3}
			\|\bar{z}\|_{C^{\frac{6}{5}, \frac{3}{5}}(\bar{Q}_T)} \leq C_9(M).
		\end{equation}
		
	\noindent
		By following the approach in \cite[(2.21)]{pang2017global}, we can likewise arrive at
		\begin{align}
			 \|\bar{z}\|_{C^{1, 0}(\bar{Q}_T)} & \leq \|\bar{z} - z_0\|_{C^{1, 0}(\bar{Q}_T)} + \|z_0\|_{C^1(\bar{\Omega})}   \notag                \\
			                           & \leq T^\frac{3}{5}\|\bar{z}\|_{C^{1, \frac{3}{5}}(\bar{Q}_T)} + \|z_0\|_{C^1(\bar{\Omega})}    \notag       \\
			                                 & \leq T^\frac{3}{5} C_{10}(M) + \|z_0\|_{C^1(\bar{\Omega})}    	\leq \frac{M}{2}, \label{zeq4}
		\end{align}
		provided $T \in (0, 1)$ and sufficiently small, such that $T^\frac{3}{5} C_{10}(M) \leq 1$. Moreover, by $z_0(x) \geq 0$, $(z, h) \in S$, and the parabolic comparison principle, we have
		\begin{equation}
			\label{nc1}
			\bar{z} \geq 0.
		\end{equation}
		From \eqref{hsol}, \eqref{hWbound} and \eqref{zeq4}, we conclude that $(\bar{z}, \bar{h}) \in S$, provided $T$ is sufficiently small. Hence, it is established that $F$ maps $S$ into itself.
		Next, we will show that $F$ is a contraction mapping on the set $S$. 
		
		\noindent
		Let $(z_1, h_1), (z_2, h_2) \in S$ and set $(\bar{z}_1, \bar{h}_1):= F(z_1, h_1), (\bar{z}_2, \bar{h}_2):=  F(z_2, h_2)$. From \eqref{dcop1}
		\begin{equation*}
			\begin{cases}
				\partial_t (c_{21} - c_{22}) = a_2\Delta (c_{21} - c_{22}) - \delta (c_{21} - c_{22}) + \alpha (e^{\frac{b}{a_1} h_1} z_1 - e^{\frac{b}{a_1} h_2} z_2)       & \quad (x, t) \in Q_T,                             \\
				\partial_\nu(c_{21} - c_{22}) = 0,                                                                                                                 & \quad x \in \partial \Omega,\ 0 < t < T,           \\
				(c_{21} - c_{22})(x, 0) = 0,                                                                                                                       & \quad x \in \Omega,
			\end{cases}
		\end{equation*}
		where we can also have
		\begin{equation}
			\label{contrac1}
			\alpha (e^{\frac{b}{a_1} h_1} z_1 - e^{\frac{b}{a_1} h_2} z_2) \leq C_{11}(M) \left(\|z_1 - z_2\|_{C^{1, 0}(\bar{Q}_T)} + \|h_1 - h_2\|_{C(0, T; W^{1, 5}(\Omega))}\right);
		\end{equation}
		for a detailed calculation of \eqref{contrac1},  please refer to \cite{pang2017global}. For $T \in (0, 1)$, the parabolic $L^p$ theory yields the following two estimates:
		\begin{equation}
			\label{cont1}
			\|c_{21} - c_{22}\|_{W_5^{2, 1}(Q_T)} \leq C_{12}(M) \left(\|z_1 - z_2\|_{C^{1, 0}(\bar{Q}_T)} + \|h_1 - h_2\|_{C(0, T; W^{1, 5}(\Omega))}\right)
		\end{equation}
		and
		\begin{equation}
			\label{cont2}
			\|c_{21} - c_{22}\|_{C^{1, 0}(\bar{Q}_T)} \leq C_{12}(M) \left(\|z_1 - z_2\|_{C^{1, 0}(\bar{Q}_T)} + \|h_1 - h_2\|_{C(0, T; W^{1, 5}(\Omega))}\right).
		\end{equation}
		The Sobolev embedding theorem in conjunction with \eqref{lem_app} in the Appendix, then gives
		\begin{equation}
			\label{cont3}
			\|\bar{h}_1 - \bar{h}_2\|_{C(0, T; W^{1, 5(\Omega)})} \leq T C_{13}(M) (\|z_1 - z_2\|_{C^{1, 0}(\bar{Q}_T)} + \|h_1 - h_2\|_{C(0, T; W^{1, 5}(\Omega))} ).
		\end{equation}
		Consider the equation for $\bar{z}_1 - \bar{z}_2$. From \eqref{dcop3} we can have
		\begin{equation}
			\label{cont4}
			\begin{cases}
				\partial_t(\bar{z}_1 - \bar{z}_2) - a_1 \Delta (\bar{z}_1 - \bar{z}_2) - b \nabla h_1 (\bar{z}_1 - \bar{z}_2) + \tilde{f}_1(x, t)  (\bar{z}_1 - \bar{z}_2) = \tilde{f}_2(x, t)   & (x, t) \in Q_T,\\
				\partial_\nu(\bar{z}_1 - \bar{z}_2) = 0,                                                                                                                 \quad & x \in \partial \Omega, 0 < t < T,           \\
				(\bar{z}_1 - \bar{z}_2)(x, 0) = 0,                                                                                                                        \quad & x \in \Omega,
			\end{cases}
		\end{equation}
		where
		\begin{align*}
			\tilde{f}_1(x, t) = & ~ \alpha + \tfrac{b\gamma_2}{a_1} \tfrac{c_{21}}{K_{c_1} + c_{21}} - \beta (1 - \tfrac{1}{K_{c_1}} z_1 e^{\frac{b}{a_1} h_1} - \tfrac{c_{21}}{K_{c_1}})   \\[5pt]
			\tilde{f}_2(x, t) = & ~ \tfrac{b \gamma_1}{a_1} z_1 h_1 c_{21} - \tfrac{b \gamma_1}{a_1} z_2 h_2 c_{22} + \delta  e^{-\frac{b}{a_1} h_1} c_{21} - \delta  e^{-\frac{b}{a_1} h_2} c_{22} - \tfrac{b \gamma_2}{a_1} \bar{z}_2 \left(\tfrac{c_{21}}{K_{c_1} + c_{21}} - \tfrac{c_{22}}{K_{c_1} + c_{22}} \right) \\
			& - \beta \bar{z}_2 \left\{(1 -  \tfrac{1}{K_{c_1}} z_2 e^{\frac{b}{a_1} h_2} - \tfrac{c_{22}}{K_{c_1}}) - (1 -  \tfrac{1}{K_{c_1}} z_1e^{\frac{b}{a_1}h_1} - \tfrac{c_{21}}{K_{c_1}})\right\} + b \nabla h_1 \cdot \nabla \bar{z}_2 - b \nabla h_2 \cdot \nabla \bar{z}_2.
		\end{align*}
		As $(z_j, h_j) \in S (j = 1, 2)$, by \eqref{c2s2}, \eqref{zeq3}, \eqref{cont2} and \eqref{cont3}, we can have
		\begin{align*}
			& \|\tilde{f}_1\|_{L^5(Q_T)} \leq C_{14}(M),     \\[5pt]
			& \|\tilde{f}_1\|_{L^5(Q_T)} \leq C_{15}(M)(\|z_1 - z_2\|_{C^{1, 0}(\bar{Q}_T)} + \|h_1 - h_2\|_{C(0, T; W^{1, 5}(\Omega))} ).
		\end{align*}
		Again, by the parabolic $L^p$-theory in conjunction with $0 < T < 1$, we can have
		\begin{equation}
			\label{cont5}
			\|\bar{z}_1 - \bar{z}_2\|_{W_5^{2, 1}(Q_T)} \leq  C_{16}(M)(\|z_1 - z_2\|_{C^{1, 0}(\bar{Q}_T)} + \|h_1 - h_2\|_{C(0, T; W^{1, 5}(\Omega))} ),
		\end{equation}
		which together with Sobolev embedding results in
		\begin{equation}
			\label{cont6}
			\|\bar{z}_1 - \bar{z}_2\|_{C^{\frac{6}{5}, \frac{3}{5}}(\bar{Q}_T)} \leq  C_{17}(M)(\|z_1 - z_2\|_{C^{1, 0}(\bar{Q}_T)} + \|h_1 - h_2\|_{C(0, T; W^{1, 5}(\Omega))}).
		\end{equation}
		Following the approach adopted in \cite[(2.27)]{pang2017global}, we arrive at
		\begin{equation}
			\label{cont7}
			\|\bar{z}_1 - \bar{z}_2\|_{C^{1, 0}(\bar{Q}_T)} \leq T^{\frac{3}{5}}C_{17}(M) (\|z_1 - z_2\|_{C^{1, 0}(\bar{Q}_T)} + \|h_1 - h_2\|_{C(0, T; W^{1, 5}(\Omega))}).
		\end{equation}
		From \eqref{cont2} and \eqref{cont7}, we conclude that $F$ is a contraction mapping on $S$ if $T$ is sufficiently small such that $T C_{13}(M) + T^{\frac{3}{5}}C_{17}(M) \leq \tfrac{1}{2}$. Hence, by Banach's fixed point theorem, $F$ possesses a unique fixed-point $(z, h) \in S$, which in conjunction with \eqref{dcop1} implies the local existence of a solution $(z, c_2, h)$ to \eqref{sys1}.
		
	\noindent
		The maximum existence time $T_{\max}$ of the solution $(z, c_2, h)$, as given in \eqref{leloc2}, can de derived by following the same reasoning as in \cite[Lemma 2.2]{pang2017global}.
	\end{proof}

	\subsubsection{A priori estimates}
	
	\begin{lemma}
		\label{lemmhbound}
		There are some positive constants $\hat C$ and $\tilde C$ such that any solution of \eqref{sys1} satisfies
		\begin{equation}
			\label{hbound}
			h(x, t) \leq \left\{\|h_0\|_{L^\infty(\Omega)} + \frac{\widehat{C}}{\widetilde{C}}\right\} \cdot e^{\widetilde{C} t} \quad \text{for all}~ x \in \Omega~\text{and}~t \in (0, T_{\max}).
		\end{equation}
		\end{lemma}
	\begin{proof}
		Consider the third equation of \eqref{sys1}. Since the positivity of solution component $c_2$ and $h$ has been established in Lemma \ref{local_existence}, we can find two positive constants $\widetilde{C}$ and $\widehat{C}$ such that $\tfrac{c_2}{K_{c_1} + c_2} \leq \hat{C}$, resulting in
		\begin{equation*}
			h_t \leq \widetilde{C} h + \widehat{C} \quad ~\text{in}~ \Omega \times (0, T_{\max}),
		\end{equation*}
		as in \cite[Lemma 2.2]{surulescu2021does}. Since for
		\begin{equation*}
			\tilde{h}(x, t) := \|h_0\|_{L^\infty(\Omega)} e^{\widetilde{C} t} + \widehat{C} \int_0^t e^{\widetilde{C} (t - s)}ds, \quad x \in \Omega, t \geq 0,
		\end{equation*}
		we have $\tilde{h}_t{\cb \le} \widetilde{C}{\tilde h} + \widehat{C}$ in $\Omega \times (0, \infty)$ with $\tilde{h}(\cdot, 0) = \|h_0\|_{L^\infty(\Omega)} \geq h_0$ in $\Omega$, by an ODE comparison argument we can conclude that $h \leq \tilde{h}$ in $\Omega \times (0, T_{\max})$. 
		Thus,
		\begin{equation*}
			\|h(\cdot, t)\|_{L^\infty(\Omega)} \leq \|h_0\|_{L^\infty(\Omega)} e^{\widetilde{C} t} + \widehat{C}  \int_0^t e^{\widetilde{C} (t - s)}ds, \quad \text{for all}~ t \in (0, T_{\max}).
		\end{equation*}
		With the estimate $\int_0^t e^{\widetilde{C}(t - s)} ds \leq \frac{1}{\widetilde{C}} e^{\widetilde{C} t}$ for $t \geq 0$ we then get \eqref{hbound}.
	\end{proof}
	\begin{lemma}
		\label{lemmal1bound}
		Let $T > 0$. Then there exists a $C(T) > 0$ such that
		\begin{equation}
			\label{l1b1}
			\int_\Omega c_1(\cdot, t) \leq C(T), \quad \text{for all}~ t \in (0, \widehat{T}),
		\end{equation}
		\begin{equation}
			\label{l1b2}
			\int_\Omega c_2(\cdot, t) \leq C(T) \quad \text{for all}~ t \in (0, \widehat{T}),
		\end{equation}
		and
		\begin{equation}
			\label{sqc1}
			\int_0^{\widehat{T}} \int_\Omega c_1^2 \leq C(T),
		\end{equation}
		where $\widehat{T} := \min\{T, T_{\max}\}$.
	\end{lemma}

	\begin{proof}
		By integrating the first and second equations of \eqref{model} with respect to space and adding the respective results we obtain
		\begin{align*}
			\dfrac{d}{dt} \left\{\int_\Omega c_1 + \int_\Omega c_2\right\} + \frac{\beta}{K_{c_1}} \int_\Omega c_1^2 = \beta \int_\Omega c_1 (1 - \tfrac{c_2}{K_{c_2}}) \quad \text{for all}~ t \in (0, \widehat{T}).
		\end{align*}
		With the positivity of the solution components $c_1$ and $c_2$ established in Lemma \ref{local_existence} we get
		\begin{align}
			\label{l1b21}
			\dfrac{d}{dt} \left\{\int_\Omega c_1 + \int_\Omega c_2\right\} + \frac{\beta}{K_{c_1}} \int_\Omega c_1^2 \leq \beta \left\{\int_\Omega c_1 + \int_\Omega c_2\right\} \quad \text{for all}~ t \in (0, \widehat{T}),
		\end{align}
		or
		\begin{align}
			\label{l1b23}
			\dfrac{d}{dt} \left\{\int_\Omega c_1 + \int_\Omega c_2\right\} \leq \beta \left\{\int_\Omega c_1 + \int_\Omega c_2\right\} \quad \text{for all}~ t \in (0, \widehat{T}).
		\end{align}
		Applying Gronwall's lemma to \eqref{l1b2} in the interval $t \in (0, \widehat{T})$ results in
		\begin{equation*}
			\int_\Omega c_1(\cdot, t) + \int_\Omega c_2(\cdot, t) \leq \left\{\int_\Omega c_1^0 + \int_\Omega c_2^0\right\} e^{\beta {\cb \hat T}}  \quad \text{for all}~ t \in (0, \widehat{T}),
		\end{equation*}
		directly yielding \eqref{l1b1} and \eqref{l1b2}. Estimate \eqref{sqc1} can be directly obtained by integrating \eqref{l1b21} on the interval $t \in (0, \widehat{T})$ and considering both \eqref{l1b1} and \eqref{l1b2}.
	\end{proof}
	
	\begin{lemma}
		\label{lemnabc2}
		Let $T > 0$. Then there exists a $\hat{C}(T) > 0$ such that
		\begin{equation}
			\label{nabc21}
			\int_\Omega |\nabla c_2(\cdot, t)|^2 \leq \hat{C}(T) \quad \text{for all} ~ t \in (0, \widehat{T})
		\end{equation}
		as well as
		\begin{equation}
			\label{delc2}
			\int_0^{\widehat{T}}  \int_\Omega  |\Delta c_2|^2 \leq \hat{C}(T),
		\end{equation}
		where $\widehat{T} := \min\{T, T_{\max}\}$.
	\end{lemma}
	\begin{proof}
		We test the $c_2$-equation of \eqref{model} with $- \Delta c_2$ and use Young's inequality to have
		\begin{align}
			\nonumber \frac{1}{2} \frac{d}{dt} \int_\Omega |\nabla c_2|^2 + a_2 \int_\Omega |\Delta c_2|^2 + \delta \int_\Omega |\nabla c_2|^2 & = - \alpha \int_\Omega c_1 \Delta c_2               \\
			\label{nabc22} & \leq \frac{a_2}{2} \int_\Omega |\Delta c_2|^2 + \frac{\alpha^2}{2 a_2} \int_\Omega c_1^2
		\end{align}
		for all $ t \in (0, T_{\max})$. In particular, this shows that $v(t) := \int_\Omega |\nabla c_2(\cdot, t)|^2, t \in [0, \widehat{T}),$ satisfies
		\begin{equation*}
			v'(t) + 2 \delta v(t) \leq f(t) := \frac{\alpha^2}{a_2} \int_\Omega c_1^2
		\end{equation*}
		where from Lemma \ref{lemmal1bound} we know that
		\begin{equation}
			\label{nabc23}
			\int_{t}^{t + 1} f(s) ds \leq \frac{\alpha^2}{a_2}C(T) > 0 \quad \text{for all}~ t \in [0, \widehat{T} - 1).
		\end{equation}
		Hence, \cite[Lemma.3.4]{stinner2014global} ensures that
		\begin{equation}
			\label{nabc24}
			\int_\Omega |\nabla c_2(\cdot, t)|^2 \leq \max\left\{\int_\Omega |\nabla c_2^0|^2 + \tfrac{\alpha^2}{a_2}C(T), \tfrac{\alpha^2}{ a_2 \delta}(\tfrac{1 + 4 \delta}{2 \delta}) C(T)\right\} \quad \text{for all}~ t \in (0, \widehat{T}),
		\end{equation}
		thus yielding \eqref{nabc21}. Estimate \eqref{delc2} can be directly obtained by integrating  \eqref{nabc22} and as a consequence of \eqref{nabc24}.
	\end{proof}
	\noindent
	The following result is an immediate outcome of the estimates in Lemma \ref{lemnabc2}.
	\begin{lemma}
		\label{regc2}
		Let $p \geq 1$ and $T > 0$. Then there exists $C(p, T) > 0$ such that with $\widehat{T} := \min\{T, T_{\max}\}$,
		\begin{equation}
			\label{reguc2}
			\int_\Omega c_2^p(\cdot, t) \leq C(p, T)  \quad \text{for all} ~ t \in (0, \widehat{T}).
		\end{equation}
	\end{lemma}
	\begin{proof}
		As $W^{1, 2}(\Omega) \hookrightarrow L^p(\Omega)$, combining \eqref{l1b2} with \eqref{nabc21} results in \eqref{reguc2}.
	\end{proof}

	
	\begin{lemma}
		\label{z_inf_bound}
		Let $p \geq 4$ and $T > 0$. Then there exists $C(T) > 0$ such that
		\begin{equation}
			\label{zlb1}
			\|z(\cdot, t)\|_{L^\infty(\Omega)} \leq C(T) \quad \text{for all} ~ t \in (0, \widehat{T}),
		\end{equation}
		where $\widehat{T} := \min\{T, T_{\max}\}$.
	\end{lemma}
	\begin{proof}
		Using \eqref{sys1} we have
		\begin{align}
			\label{zlb2}
			\nonumber & \dfrac{d}{dt} \int_\Omega e^{\frac{b}{a_1} h} z^p = p \int_\Omega e^{\frac{b}{a_1} h} h z^{p - 1} z_t  + \tfrac{b}{a_1}\int_\Omega e^{\frac{b}{a_1} h} z^p h_t\\
			\nonumber = & p \int_\Omega e^{\frac{b}{a_1} h} z^{p - 1}\Big\{a_1e^{-\tfrac{b}{a_1} h} \nabla \cdot (e^{\tfrac{b}{a_1} h} \nabla z) - \alpha z + \delta c_2 e^{-\frac{b}{a_1} h} \\
			\nonumber & + z [\beta(1 - \tfrac{1}{K_{c_1}}z e^{\frac{b}{a_1} h} - \tfrac{c_2}{K_{c_1}}) + \tfrac{b \gamma_1}{a_1} h c_2 - \tfrac{b\gamma_2}{a_1} \tfrac{c_2}{K_{c_1} + c_2}]\Big\}      \\
			\nonumber  & + \tfrac{b}{a_1} \int_\Omega e^{\frac{b}{a_1} h} z^p \Big\{-\gamma_1 h c_2 + \gamma_2\tfrac{c_2}{K_{c_1} + c_2}\Big\}                       \\
			\nonumber & = - a_1 p(p - 1) \int_\Omega e^{\frac{b}{a_1} h} z^{p - 2} |\nabla z|^2 + (\beta - \alpha) p \int_\Omega e^{\frac{b}{a_1} h} z^p + p \int_\Omega z^{p - 1} c_2
			+ (p - 1)\tfrac{b \gamma_1}{a_1} \int_\Omega e^{\frac{b}{a_1} h} z^p c_2 h \\ & - p \beta \tfrac{1}{K_{c_1}}\int_\Omega e^{\frac{2 b}{a_1} h} z^{p + 1} - p \beta \tfrac{1}{K_{c_1}} \int_\Omega e^{\frac{b }{a_1} h} z^p c_2 - (p - 1)\tfrac{b\gamma_2}{a_1} \int_\Omega e^{\frac{b}{a_1} h} z^p \tfrac{c_2}{K_{c_1} + c_2}
		\end{align}
		for all $t \in (0, T_{\max})$.  
		
		\noindent
		The positivity of the solution components together with the condition $p \geq 4$ lead to 
				\begin{align}
			\label{zlb3}
			\dfrac{d}{dt} \int_\Omega e^{\frac{b}{a_1} h} z^p \leq & - a_1 p(p - 1) \int_\Omega e^{\frac{b}{a_1} h} z^{p - 2} |\nabla z|^2 + (\beta - \alpha) p \int_\Omega e^{\frac{b}{a_1} h} z^p \notag \\
			&+ p \int_\Omega z^{p - 1} c_2
			+ (p - 1)\tfrac{b \gamma_1}{a_1} \int_\Omega e^{\frac{b}{a_1} h} z^p c_2 h
		\end{align}
		for all $t \in (0, T_{\max})$. 
		
		\noindent
		By Lemma \ref{lemmhbound} we can have a finite $C_1(T)$ such that $\|h\|_{L^\infty(\Omega \times (0, \widehat{T}))}\le C_1(T)$ for all $t \in (0, \widehat{T})$, hence
		\begin{align}
			\label{zlb4}
			\nonumber \dfrac{d}{dt} \int_\Omega e^{\frac{b}{a_1} h} z^p \leq - \frac{4 a_1 (p - 1)}{p^2} \int_\Omega |\nabla z^{\frac{p}{2}}|^2 & + (\beta - \alpha) p C_{f_1}(C_1(T))\int_\Omega z^p + p \int_\Omega z^{p - 1} c_2 \\
			& + (p - 1) C_{f_2}(C_1(T))\tfrac{b \gamma_1}{a_1} \int_\Omega  z^p c_2
		\end{align}
		for all $t \in (0, \widehat{T})$, where  $C_{f_j}(C_1(T))>0$ are constants depending on $C_1(T)$, $j=1,2$. 
		
		\noindent
		From \eqref{zlb4} we can have $C_2(T) > 0$ and $C_3(T) > 0$, such that
		\begin{align}
			\label{zlb5}
			\dfrac{d}{dt} \int_\Omega e^{\frac{b}{a_1} h} z^p \leq - \frac{4 a_1 (p - 1)}{p^2} \int_\Omega |\nabla z^{\frac{p}{2}}|^2 + C_2(T) \int_\Omega z^p + C_3(T) \int_\Omega z^p c_2
		\end{align}
		for all $t \in (0, \widehat{T})$. By Young's inequality
		\begin{align}
			\label{zlb6}
			\dfrac{d}{dt} \int_\Omega e^{\frac{b}{a_1} h} z^p \leq - \frac{4 a_1 (p - 1)}{p^2} \int_\Omega |\nabla z^{\frac{p}{2}}|^2 + C_2(T) \int_\Omega z^p + C_3(T)\epsilon_1 \int_\Omega z^{p + 1} + C_3(T)C(\epsilon_1) \int_\Omega  c_2^{p + 1}
		\end{align}
		for all $t \in (0, \widehat{T})$. 
		
		\noindent
		Lemma \ref{regc2} provides a $C(p, T) > 0$ such that $\int_\Omega c_2^p \leq C(p, T)$ for all $t \in (0, \widehat{T})$, hence,
		\begin{align}
			\label{zlb6}
			\dfrac{d}{dt} \int_\Omega e^{\frac{b}{a_1} h} z^p \leq - \frac{4 a_1 (p - 1)}{p^2} \int_\Omega |\nabla z^{\frac{p}{2}}|^2 + C_2(T) \int_\Omega z^p + C_3(T)\epsilon_1 \int_\Omega z^{p + 1} + C(p, T)
		\end{align}
		for all $t \in (0, \widehat{T})$. We can rewrite \eqref{zlb6} in the following form:
		\begin{equation}
			\label{zlb7}
			\dfrac{d}{dt} \int_\Omega e^{\frac{b}{a_1} h} z^p + C_4 \int_\Omega |\nabla z^{\frac{p}{2}}|^2 +  C_5(T) \int_\Omega e^{\frac{b}{a_1} h} z^p \leq C_6(p, T) \int_\Omega z^{p + 1} + C_7(p, T)
		\end{equation}
		for all $t \in (0, \widehat{T})$. By means of the Gagliardo-Nirenberg inequality, we can estimate
		\begin{equation}
			\label{zlb8}
			C_6(p, T) \int_\Omega z^{p + 1} \leq C_4 \|\nabla z^{\frac{p}{2}}\|^2_{L^2(\Omega)} + C_8 p^{\frac{4p}{p - 2}}\|z^{\frac{p}{2}}\|_{L^1(\Omega)}^{\frac{2p}{p - 2}}
		\end{equation}
		for all $t \in (0, \widehat{T})$. For the details of this calculation refer to \cite[(3.31)]{ke2018note}. 
		
		\noindent
		Inserting \eqref{zlb8} into \eqref{zlb7}, we have
		\begin{align}
			\label{zlb7}
			\nonumber \dfrac{d}{dt} \int_\Omega e^{\frac{b}{a_1} h} z^p + C_5(T) \int_\Omega e^{\frac{b}{a_1} h} z^p & \leq  C_8 p^{\frac{4p}{p - 2}}\|z^{\frac{p}{2}}\|_{L^1(\Omega)}^{\frac{2p}{p - 2}} + C_7(p, T)  \\[5pt]
			& \leq C_9 p^{\frac{4p}{p - 2}}\left(\max\{1, \|c_1^{\frac{p}{2}}\|_{L^1(\Omega)}\}\right)^{\frac{2p}{p - 2}}
		\end{align}
		for all $t \in (0, \widehat{T})$. Define $\tau := \tfrac{1}{2} \widehat{T}$, set $p_i = 2^{i + 2}$ and let
		\begin{equation}
			\label{defMi}
			M_i := \max\left\{1, \sup_{t \in (\tau, \widehat{T})} \int_\Omega z^{\frac{p_i}{2}}(\cdot, t)\right\}, \quad i = 1, 2, 3, \ldots.
		\end{equation}
		Then \eqref{zlb7} directly entails
		\begin{equation}
			\label{zlb7}
			\dfrac{d}{dt} \int_\Omega e^{\frac{b}{a_1} h} z^p + C_5(T) \int_\Omega e^{\frac{b}{a_1} h} z^p  \leq C_{10} p_i^{\frac{2p_i}{p_i - 2}} M_{i - i}^{\frac{2p_i}{i - 1}}(T)
		\end{equation}
		for all $t \in (0, \widehat{T})$. 
		
		\noindent
		A comparison argument then gives us a $\rho > 1$ independent of $i$ such that
		\begin{equation}
			\label{defMi1}
			M_i(T) \leq \max\left\{\rho^i M_{i - 1}^{\frac{2p_i}{p_i - 2}}(T),C_{11}(T) \|z_0\|^{p_i}_{L^\infty(\Omega)}\right\},
		\end{equation}
		where we have used the fact that $\kappa_i := \tfrac{2 p_i}{p_i - 2} \leq 4$. 
		
		\noindent
		If $\rho^i M_{i - 1}^{\kappa_i}(T) \leq C_{12}(T) \|z_0\|^{p_i}_{L^\infty(\Omega)}$ for finitely many $i \geq 1$, then we have
		\begin{equation}
			\label{defMi2}
			\left(\sup_{t \in (\tau, \widehat{T})}\int_\Omega z^{p_i - 1}(\cdot, t)\right)^{\frac{1}{p_i - 1}} \leq \left(\frac{C_{12}(T)\|z_0\|_{L^\infty(\Omega)}^{p_i}}{\rho_i}\right)^{\frac{1}{p_{i - 1}\kappa_i}}
		\end{equation}
		for such $i$, which ensures that
		\begin{equation}
			\label{zlb8}
			\sup_{t \in (\tau, \widehat{T})} \|z(\cdot, t)\|_{L^\infty(\Omega)} \leq \|z_0\|_{L^\infty(\Omega)}.
		\end{equation}
		Else, if $\rho^i M_{i - 1}^{\kappa_i}(T) > C_{12}(T) |\Omega|  \|z_0\|^{p_i}_{L^\infty(\Omega)}$ for a sufficiently large $i$, then \eqref{defMi1} entails
		\begin{equation}
			\label{defMi3}
			M_i(T) \leq \rho^i M_{i - i}^{\kappa_i}(T) \quad \text{for such sufficiently large}~ i.
		\end{equation}
		Thus \eqref{defMi3} is valid for all $i \geq 1$ with enlarging $\rho$, if required. That is
		\begin{equation*}
			M_i(T) \leq \rho^i M_{i - i}^{\kappa_i}(T) \quad \text{for all}~ i \geq 1.
		\end{equation*}
		We are now in a position to apply a straightforward induction argument as in \cite[(3.36)]{ke2018note} and obtain
		\begin{equation*}
			\|z(\cdot, t)\|_{L^\infty(\Omega)} \leq C_{13}(T) \quad \text{for all}~ t \in (\tau, \widehat{T}),
		\end{equation*}
		from which \eqref{zlb1} directly follows.
	\end{proof}
	\begin{lemma}
		\label{c2_inf_bound}
		Let $T > 0$. Then there exists $C(T) > 0$ such that
		\begin{equation}
			\label{c2lb1}
			\|c_2(\cdot, t)\|_{L^\infty(\Omega)} \leq C(T) \quad \text{for all} ~ t \in (0, \widehat{T}),
		\end{equation}
		where $\widehat{T} := \min\{T, T_{\max}\}$.
	\end{lemma}
	
	\begin{proof}
		Applying the variation of constants formula to determine $c_2$ from the second equation of \eqref{sys1}, we get
		\begin{align}
			\label{c2lb2}
			\nonumber \|c_2(\cdot, t)\|_{L^\infty(\Omega)} & = \sup_\Omega \left(e^{t(a_2\Delta - \delta)} c_2^0 + \alpha \int_0^t e^{(t - s)(a_2\Delta - \delta)} e^{\frac{b}{a_1}h(\cdot, s)} z(\cdot, s) ds\right)  \\
			\nonumber                                      & \leq \sup_\Omega \left(e^{t(a_2\Delta - \delta)} c_2^0\right) + \sup_\Omega \left( \alpha \int_0^t e^{(t - s)(a_2\Delta - \delta)} e^{\frac{b}{a_1}h(\cdot, s)} z(\cdot, s) ds\right) \\
			& \leq \|e^{t(a_2\Delta - \delta)} c_2^0\|_{L^\infty(\Omega)} + \alpha  \int_0^t \|e^{(t - s)(a_2\Delta - \delta)} e^{\frac{b}{a_1}h(\cdot, s)} z(\cdot, s) \|_{L^\infty(\Omega)} ds\quad \text{for all} ~ t \in (0, \widehat{T}).
		\end{align}
		Clearly,
		\begin{equation}
			\label{c2lb3}
			\|e^{t(a_2\Delta - \delta)} c_2^0\|_{L^\infty(\Omega)} \leq C_1 \quad \text{for all} ~ t \in (0, \widehat{T}),
		\end{equation}
		where $C_1 = \|c_2^0\|_{C(\bar{\Omega})}$. 
		
		\noindent
		Lemma \ref{lemmhbound} ensures a finite $C_2(T) := \|h\|_{L^\infty(\Omega \times (0, \widehat{T}))}$ for all $t \in (0, \widehat{T})$, and Lemma \ref{z_inf_bound} guarantees $\|z(\cdot, t)\|_{L^\infty(\Omega)} \leq C_3(T)$ for all $t \in (0, \widehat{T})$, hence
		\begin{align}
			\alpha \int_0^t \|e^{(t - s)(a_2\Delta - \delta)}e^{\frac{b}{a_1}h(\cdot, s)} z(\cdot, s)\|_{L^\infty(\Omega)}  & \leq C_4(T) \int_0^t (1 + (t - s)^{-\frac{1}{2}})e^{-C_5(t - s)} ds   \notag \\                              
						 & \leq 2 C_4(T) C_6 \int_0^t (1 - s^{-\frac{1}{2}}) e^{-C_5 s}ds     \notag   \\
			& \leq C_7(T) \quad \text{for all} ~ t \in (0, \widehat{T}).\label{c2lb4}
		\end{align}
		Here and above we have used the Neumann heat semigroup estimates \cite[Lemma 1.3]{winkler2010aggregation} and the fact that $\int_0^\infty (1 + s^{-\frac{1}{2}})e^{-C_5 s} ds < \infty$. Collecting \eqref{c2lb3}-\eqref{c2lb4}, we get \eqref{c2lb1}.
		
	\end{proof}


	\begin{lemma}
		\label{nabhbound1}
		Let $T > 0$ and $q \geq 1$. Then there exists $C(q, T) > 0$ which satisfies
		\begin{equation}
			\label{nabhb1}
			\|\nabla h(\cdot, t)\|_{L^q(\Omega)} \leq C(q, T) \int_\tau^T \|\nabla c_2(\cdot, s)\|_{L^q(\Omega)}ds + C(q, T)
		\end{equation}
		for all $t \in (\tau, \widehat{T})$, with $\widehat{T} := \min\{T, T_{\max}\}$ and $\tau := \tfrac{1}{2} \widehat{T}$.
	\end{lemma}
	\begin{proof}
		Differentiating the third equation of \eqref{sys1}, we have
		\begin{equation}
			\label{nabhb2}
			\partial_t \nabla h = - \gamma_1 c_2 \nabla h + \nabla c_2 \left\{-\gamma_1 h + \frac{K_{c_1}}{(K_{c_1} + c_2)^2}\right\}.
		\end{equation}
		As consequences of Lemma \ref{lemmhbound} and Lemma \ref{c2_inf_bound}, the functions $- \gamma_1 c_2$ and $-\gamma_1 h + \tfrac{K_{c_1}}{(K_{c_1} + c_2)^2}$ are bounded in $\Omega \times (0, \widehat{T})$, so from \eqref{nabhb2} we can have
		\begin{align}
			\label{nabhb3}
			\nonumber \|\nabla h(\cdot, t)\|_{L^q(\Omega)} & = \left\|\nabla h(\cdot, \tau) + \int_\tau^t \partial_t \nabla h(\cdot, s) ds\right\|_{L^q(\Omega)}  \\
			& \leq C_1(T) + C_2(T) \int_\tau^t \left\{\|\nabla h(\cdot, s)\|_{L^q(\Omega)} + \|\nabla c_2(\cdot, s)\|_{L^q(\Omega)}\right\} ds
		\end{align}
		for all $t \in (\tau, \widehat{T})$. By applying Gronwall's lemma to \eqref{nabhb3}, we arrive at \eqref{nabhb1}.
	\end{proof}

	\begin{lemma}
		\label{nabhbound2}
		For all $q \geq 1$ and $T > 0$, there exists a $C(q, T) > 0$ such that with $\widehat{T} := \min\{T, T_{\max}\}$ and $\tau := \tfrac{1}{2} \widehat{T}$, the following holds true
		\begin{equation}
			\label{1nabh1}
			\|\nabla h(\cdot, t)\|_{L^q(\Omega)} \leq C(q, T), \quad \text{for all}~ t \in (\tau, \widehat{T}).
		\end{equation}
	\end{lemma}
	\begin{proof}
		By combining \eqref{nabc21} and \eqref{delc2} from Lemma \ref{lemnabc2} in conjunction with elliptic regularity we have a $C_1(T) > 0$ such that
		\begin{equation*}
			\int_\tau^{\widehat{T}} \|c_2(\cdot, t)\|^2_{W^{2,2}(\Omega)} dt \leq C_1(T),
		\end{equation*}
		which when combined with the continuous embedding $W^{2, 2}(\Omega) \hookrightarrow W^{1, q}(\Omega)$ and the Cauchy-Schwarz inequality allows us to have a $C_2(q) > 0$ such that
		\begin{align}
			\label{1nabh2}
			\nonumber \int_\tau^{\widehat{T}} \|\nabla c_2(\cdot, t)\|_{L^q(\Omega)}dt & \leq C_2(q) \int_\tau^{\widehat{T}} \|c_2(\cdot, t)\|_{W^{2,2}(\Omega)}dt    \\
			\nonumber & \leq C_2(q) \sqrt{T} \cdot \left\{\int_\tau^{\widehat{T}} \|c_2(\cdot, t)\|^2_{W^{2,2}(\Omega)} dt \right\}^\frac{1}{2}                       \\
			& \leq \sqrt{C_1(T)} C_2(q) \cdot \sqrt{T}.
		\end{align}
		The estimate \eqref{1nabh1} follows from the combination of \eqref{1nabh2} and Lemma \ref{nabhbound1}.
	\end{proof}
	\begin{proof}[Proof of Theorem \ref{mth}]
		With the aid of the extensibility criterion \eqref{leloc2} from Lemma \ref{local_existence}, we only need to combine the outcome of Lemma \ref{z_inf_bound} with an application of Lemma \ref{nabhbound2} for $q := 5$ to reach the desired claim.
	\end{proof}

	\subsection{Linear stability and bifurcations}\label{sec:stability}
	In this study we show that the taxis sensitivity parameter $b$ plays an important role in determining the stability of the steady state. Sufficiently high values of $b$ can induce Hopf bifurcations resulting in spatial and temporal patterns. 
	The approach to investigate the formation of patterns around the steady state is based on the linear stability analysis of the system (\ref{model}). We do a similar analysis to that in \cite{mishrainrepulsive}. In the absence of motility terms the model system admits a single non-trivial steady state:
	\begin{equation}
		\label{eq_sol1}
		\left(K_{c_1} \dfrac{ \delta}{\delta + \alpha},K_{c_1} \dfrac{\alpha}{\delta + \alpha}, \dfrac{\gamma_2(\delta + \alpha)}{\gamma_1 K_{c_1}(\delta + 2 \alpha )}\right)
	\end{equation}
	which can be rewritten as:
	\begin{equation}
		\label{eq_sol2}
		\left(c_1^*, \dfrac{\alpha}{\delta} c_1^*, \dfrac{\gamma_2}{\gamma_1(K_{c_1} +\dfrac{ \alpha}{\delta} c_1^*)}\right)
	\end{equation}
	where $c_1^* = K_{c_1} \dfrac{ \delta}{\delta + \alpha}$. We will now discuss the stability of the steady state (\ref{eq_sol2}). 
	\begin{theorem}
		\label{th1}
		The steady state (\ref{eq_sol2}) of the model system (\ref{model}) without motility terms is always locally asymptotically stable.
	\end{theorem}
	
	\begin{proof}
		
		At any equilibrium solution  $(c_1, c_2, h)$, the Jacobian matrix of system (\ref{model}) without motility terms is 
		\begin{equation}
		J(c_1,c_2,h)=	\begin{pmatrix}
				\label{jacob}
				- \alpha - \beta\dfrac{ c_1}{K_{c_1}} + \beta( 1 - \dfrac{c_1}{K_{c_1}} - \dfrac{c_2}{K_{c_1}} )    & \quad  -\beta \dfrac{ c_1}{K_{c_1}}+\delta                                  &  \quad       0                                \\
				\alpha                                         &  \quad    -\delta                                         &\quad         0                                \\
				0                                            &\quad  - \gamma_1h + \dfrac{\gamma_2 K_{c_1}}{(K_{c_1} + c_2)^2}           &\quad -c_2\gamma_1
			\end{pmatrix}.
		\end{equation}
		Therefore by using (\ref{jacob}) the Jacobian of (\ref{model}) at $\left(c_1^*, \dfrac{\alpha}{\delta} c_1^*, \dfrac{\gamma_2}{\gamma_1(K_{c_1} +\dfrac{ \alpha}{\delta} c_1^*)}\right)$ is given by
		\begin{equation}
		\renewcommand{\arraystretch}{1.8}
			J(c_1^*, \frac{\alpha}{\delta} c_1^*, \frac{\gamma_2}{\gamma_1(K_{c_1} +\frac{ \alpha}{\delta} c_1^*)}):=J^* = \begin{pmatrix}
				\label{jacob2}
				-\beta\dfrac{ c_1^*}{K_{c_1}}-\alpha    & \quad       -\beta\dfrac{ c_1^*}{K_{c_1}}+\delta                                 &  \quad        0                                            \\
				\alpha          &  \quad           -\delta                                       &  \quad        0                                            \\
				0               & \quad    -\dfrac{\gamma_2\frac{ \alpha}{\delta} c_1^*}{(K_{c_1} + \frac{\alpha }{\delta}c_1^*)^2}       & \quad     -\gamma_1 \dfrac{\alpha }{\delta} c_1^*
			\end{pmatrix}.
		\end{equation}
		The characteristic equation of $J^*$ is given by:
		\begin{equation}
			\label{ls1}
			\lambda^3 + A_1\lambda^2 + A_2 \lambda + A_3 = 0,
		\end{equation}
		where
		\begin{align*}
			A_1 & = \beta \dfrac{c_1^*}{K_{c_1}} + \gamma_1\dfrac{ \alpha }{\delta}c_1^* + \alpha+ \delta,                                                       \\
			A_2 & = \beta \delta \dfrac{c_1^*}{K_{c_1}} + \beta \alpha \dfrac{ c_1^{*}}{K_{c_1}} + \alpha \gamma_1 c_1^* + \dfrac{\alpha \beta \gamma_1}{\delta}\dfrac{ c_1^{*2}}{K_{c_1}}+\dfrac{ \gamma_1 \alpha^2 }{\delta}{c_1}^*,     \\
			A_3 & =\dfrac{ \alpha^2 \beta \gamma_1 }{\delta}\dfrac{c_1^{*2}}{K_{c_1}} + \alpha \beta \gamma_1 \dfrac{c_1^{*2}}{K_{c_1}}.
		\end{align*}
		Based on the positive values of the parameters, it can be concluded that $A_i > 0,\ i = 1, 2, 3$ and the following has been verified:
		\begin{equation*}
			\begin{aligned}
				A_1 A_2 - A_3 &= \big(\beta \dfrac{c_1^*}{K_{c_1}} + \gamma_1\dfrac{ \alpha }{\delta}c_1^* + \alpha+ \delta \big)\big( \beta \delta \dfrac{c_1^*}{K_{c_1}} + \beta \alpha \dfrac{ c_1^{*}}{K_{c_1}} + \alpha \gamma_1 c_1^* + \dfrac{\alpha \beta \gamma_1}{\delta}\dfrac{ c_1^{*2}}{K_{c_1}}+\dfrac{ \gamma_1 \alpha^2 }{\delta}{c_1}^*\big)\notag\\
				&\hspace{1cm} -\big(\dfrac{ \alpha^2 \beta \gamma_1 }{\delta}\dfrac{c_1^{*2}}{K_{c_1}} + \alpha \beta \gamma_1 \dfrac{c_1^{*2}}{K_{c_1}}\big)\\[5pt]
				&= \big(\gamma_1\dfrac{ \alpha }{\delta}c_1^* + \alpha+ \delta \big)\big( \beta \delta \dfrac{c_1^*}{K_{c_1}} +  \beta \alpha \dfrac{ c_1^{*}}{K_{c_1}} + \alpha \gamma_1 c_1^* + \dfrac{\alpha \beta \gamma_1}{\delta}\dfrac{ c_1^{*2}}{K_{c_1}}+\dfrac{ \gamma_1 \alpha^2 }{\delta}{c_1}^*\big)\notag\\  & \hspace{1cm} +\beta \dfrac{c_1^*}{K_{c_1}}  \big(\beta \delta \dfrac{c_1^*}{K_{c_1}} +  \beta \alpha \dfrac{ c_1^{*}}{K_{c_1}}  + \dfrac{\alpha \beta \gamma_1}{\delta}\dfrac{ c_1^{*2}}{K_{c_1}}\big)\\[5pt]
				&>0.
			\end{aligned}
		\end{equation*}Thus, we can apply the Routh-Hurwitz criterion to conclude that
		$Re(\lambda_i) < 0,\ i = 1, 2, 3$. Therefore, the steady state is always asymptotically stable in the absence of motility terms. 
	\end{proof}
	
	\noindent
	We will now turn our attention to the model system with motility terms. In order to perform stability analysis, 
		we linearise (\ref{model}) at the steady state (\ref{eq_sol2}) to obtain the following Jacobian matrix:

		\begin{equation} \label{jac}
			\renewcommand{\arraystretch}{1.8}
			J_j = \begin{pmatrix}
				-\beta\dfrac{ c_1^*}{K_{c_1}} - \alpha - a_1k_j & \quad 	-\beta\dfrac{ c_1^*}{K_{c_1}}+\delta  & \quad b c_1^*            k_j                                           \\
				\alpha                 &  \quad           -\delta -a_2 k_j                                             &   \quad       0                                            \\
				0                      & \quad  -\dfrac{\gamma_2\frac{ \alpha}{\delta} c_1^*}{(K_{c_1} + \frac{\alpha }{\delta}c_1^*)^2}       & \quad     -\gamma_1 \dfrac{\alpha }{\delta} c_1^*
			\end{pmatrix},
		\end{equation}
		where $\{{k_j}\}_{j=0}^{\infty}$ denotes the set of  eigenvalues of the Laplace operator $-\Delta$ with homogeneous Neumann boundary conditions in a smooth domain $\Omega$.
		The characteristic polynomial for (\ref{jac}) is 
		\begin{equation*}
			\lambda^3 + (A_j)^{(2)} \lambda^2 +( A_j)^{(1)}\lambda + (A_j)^{(0)}= 0,
		\end{equation*}
		where
		\begin{equation*}
			\begin{aligned}
				(A_j)^{(2)} & = \beta \dfrac{c_1^*}{K_{c_1}} + \gamma_1\dfrac{ \alpha }{\delta}c_1^* + \alpha+ \delta+ (a_1 + a_2)k_j                                                                                                                                                 \\
				( A_j)^{(1)} & = a_1 a_2 {k_j}^2 + (a_1 \delta +a_2 \alpha+ a_2 \beta \dfrac{c_1^*}{K_{c_1}} + \alpha \gamma_1 \dfrac{ a_1 +a_2  }{\delta}c_1^*)k_j \notag \\
				&\hspace{2cm}+ \beta \delta \dfrac{c_1^*}{K_{c_1}}+ \alpha \beta \dfrac{c_1^*}{K_{c_1}} + \alpha \gamma_1 c_1^* +\dfrac{\alpha^2 \gamma_1 }{\delta }c_1^*+ \dfrac{\alpha \beta \gamma_1}{\delta} \dfrac{(c_1^*)^{2}  }{K_{c_1}}                       \\
				(A_j)^{(0)} & = \dfrac{a_1 a_2 \alpha \gamma_1 }{\delta} c_1 ^* {k_j}^2 + \left(\dfrac{\alpha^2 b \gamma_2 }{\delta}\dfrac{(c_1^*)^2}{(K_{c_1}+\dfrac{\alpha}{\delta} c_1^*)^2} + a_1\alpha \gamma_1 c_1^* + \dfrac{a_2 \alpha\beta \gamma_1}{\delta}  \dfrac{(c_1^*)^2}{K_{c_1}}+ \dfrac{a_2 \gamma_1 \alpha^2}{\delta } c_1^*\right)k_j \notag \\
				&\hspace{2cm}+\dfrac{ \alpha^2 \beta \gamma_1 }{\delta}\dfrac{c_1^{*2}}{K_{c_1}} + \alpha \beta \gamma_1 \dfrac{c_1^{*2}}{K_{c_1}}.
			\end{aligned}
		\end{equation*}
		\noindent
		A stable steady state requires negative real parts for the eigenvalues of matrix $J_j$ for all $j\geq 0$. Using Routh Hurwitz's stability criterion, this corresponds to: 
		\begin{equation}\label{ruth1}
			(A_j)^{(2)} > 0, \quad  (A_j)^{(1)} > 0 \quad \text{and }
		\end{equation}
		\begin{equation}\label{ruth2}
			T(b,j):=(A_j)^{(2)}  (A_j)^{(1)}- (A_j)^{(0)}=B_1 {k_j}^3 + B_2 {k_j}^2 + B_3 k_j + B_4 -\dfrac{\alpha^2 b \gamma_2 }{\delta}\dfrac{(c_1^*)^2}{(K_{c_1}+\dfrac{\alpha}{\delta} c_1^*)^2} k_j>0, 
		\end{equation}
	for all $j \geq 0$, with $B_1, B_2, B_3, B_4 \geq 0$, due to the positivity of coefficients and non-negativity of eigenvalues. For $k_j=0$ we have proved in Theorem \ref{th1} that the steady state was stable.
		\begin{remark}
			\noindent It can be observed that for $b=0$ the steady state (\ref{eq_sol2}) satisfies the stability condition.
		\end{remark}
		\noindent It can be easily seen that ({\ref{ruth1}}) is always satisfied, and for $b=0$, (\ref{ruth2}) is  $$(A_j)^{(2)} (A_j)^{(1)} - (A_j)^{(0)}= B_1 {k_j}^3 + B_2 {k_j}^2 + B_3 k_j + B_4 >0. $$
		This indicates that diffusivity alone does not affect the local stability of the steady state, and only taxis towards the hyaluron/ECM gradient can cause instability.
		Now we can determine the $b$-dependent stability condition for the steady state (and for $k_j\neq 0$).
		Consider 
		\begin{equation*}
			\psi (k_j):= \frac{ B_1 {k_j}^3 + B_2 {k_j}^2 + B_3 k_j + B_4}{\dfrac{\alpha^2  \gamma_2 }{\delta}\dfrac{(c_1^*)^2}{(K_{c_1}+\dfrac{\alpha}{\delta} c_1^*)^2} k_j} = \frac{ \psi_1 (k_j)}{\psi_2k_j},
		\end{equation*}
	where we denoted
	\begin{align*}
	&\psi_1 (\zeta ):=B_1 {\zeta }^3 + B_2 {\zeta }^2 + B_3 \zeta + B_4\\
	&\psi_2:=\dfrac{\alpha^2  \gamma_2 }{\delta}\dfrac{(c_1^*)^2}{(K_{c_1}+\dfrac{\alpha}{\delta} c_1^*)^2}.
	\end{align*}	
		\noindent Since $\psi_2$ and the coefficients of $\psi_1¸$ are all positive for all $j \geq 0$ we have 
		\begin{equation}
			\lim_{x \to 0+}\psi(x)=\lim_{x \to \infty}\psi(x)=\infty
		\end{equation}
				\noindent and notice that $\psi $ is a strictly convex function. Hence there exist $b_c $ such that
		\begin{equation}\label{b_c}
			b_c:= \min_{j \in \mathbb{N}^+}  \frac{ \psi_1 (k_j)}{\psi_2k_j}.
		\end{equation}
		When $b<b_c$ , the steady state is stable. If 
		\begin{equation}\label{notequalpsi}
			\psi (k_j) \neq \psi(k_m) \quad \text{for } j \neq m,
		\end{equation}
		then the minimum is attained for a single $ j = j_0$. This proves the following result.
		\begin{proposition}\label{prop:1}
			The steady state(\ref{eq_sol2}) is locally asymptotically stable if $b<b_c$ defined in (\ref{b_c}).
		\end{proposition}
	\noindent
		Since condition (\ref{ruth1}) is always satisfied, the stability of the steady state is determined by condition (\ref{ruth2}). When $b<b_c$ condition (\ref{ruth2}) is satisfied, hence ensuring the stability of the steady state. 
		The following theorem shows the existence of a Hopf bifurcation.
		\begin{theorem}
			For $b=b_c$, if (\ref{notequalpsi}) holds, then a Hopf bifurcation occurs.
		\end{theorem}
	\noindent
		To show the occurence of a Hopf bifurcation we use \cite{mishrainrepulsive}, \cite{transhopf}, and \cite{amannhopf}. When (\ref{notequalpsi}) is satisfied we have
		\begin{equation}
			\begin{aligned}
				&(i) \ {A_j}^{(2)}= -trace (J_j) >0, \ {A_j}^{(1)}>0, \ {A_j}^{(0)}= -det (J_j) >0 \ \text{for all }j\geq0 \text{ and } b>0; \\
				&(ii) \ {A_j}^{(2)} {A_j}^{(1)} = {A_j}^{(0)} (b_c)  \ \text{for some } j=j_0.
			\end{aligned}
		\end{equation} 
		That means the characteristic polynomial of $J_j$ has a real negative root $\lambda_0$ and a pair of purely imaginary roots $\pm i \omega_0$. Let $\lambda_1 \in \mathbb{R} $ and $\lambda_{2,3} = \sigma(b) \pm i \omega(b) $ be the eigenvalues of the Jacobian matrix such that $\lambda_1(j_0)= \lambda_0$ and $\lambda_{2,3}(j_0) =\pm i \omega_0$ . Then we have
		
		\begin{equation}
			\begin{split}\label{3.9}
				-{A_j}^{(2)} &= 2 \sigma(b) + \lambda_1 (b) \\
				{A_j}^{(1)}&= \sigma(b)^2 + \omega(b)^2 + 2\sigma(b) \lambda_1(b)\\
				-{A_j}^{(0)} &=  ( \sigma(b)^2 + \omega(b)^2 ) \lambda_1(b).
			\end{split}
		\end{equation}
		
		\noindent From above we get $\lambda_1(b) < 0$. Also since $ \sigma(b_c) =0 $, from (\ref{3.9}) we get $ \omega(b_c)^2 = {{A_j}_0} ^{(1)} > 0$ and upon differentiating each equation in (\ref{3.9}) with respect to $b$ we obtain 
		
		\begin{equation}
			\begin{split}\label{3.10}
				& 2 \sigma'(b) + \lambda_1 '(b) =0 \\
				&2 \sigma(b) \sigma'(b) + 2 \omega(b)\omega'(b) + 2\sigma'(b) \lambda_1(b) +2\sigma(b) \lambda_1'(b)= 0\\
				& (\sigma(b)^2 + \omega(b)^2 ) \lambda_1'(b)+ 2 \sigma(b) \sigma'(b) \lambda_1(b) + 2 \omega(b)\omega'(b) \lambda_1(b) = -\psi_2.
			\end{split}
		\end{equation}
		\noindent
		Therefore  we have $ \sigma'(b) = \frac{-  \lambda_1 '(b)}{2}$ and substituting the results we obtained earlier in the second and third equations of (\ref{3.10}) we get 
		
		\begin{equation*}
			2 \omega(b_c)\omega'(b_c) +2\sigma'(b_c) \lambda_1(b_c) = 0 \quad\text{ and} \quad \omega(b_c)^2  \lambda_1 '(b_c)+ 2 \omega(b_c)\omega'(b_c) \lambda_1(b_c) = -\psi_2.
		\end{equation*}
		
		\noindent
		Hence by solving the system we obtain 
		
		\begin{equation}
			\lambda_1'(b_c)  = - \frac{\psi_2}{{\omega'(b_c)}^2 + {\lambda_1(b_c)}^2} <0 \quad\text{hence}\quad \sigma' (b_c) > 0.
		\end{equation}
		\noindent
		Therefore the transversality condition for a Hopf bifurcation is satisfied at $b=b_c$.
		\nocite{*}
		
		\section{Numerical simulations and patterns}\label{sec:numerics}
We perform numerical simulations in order to illustrate the behavior of the model in terms of patterns.  

\subsection{1D simulations}
\noindent
We use Matlab's \texttt{pdepe} to solve system \eqref{model} in one dimension. We fix a set of positive parameters and calculate the critical value $b_c$. According to our stability analysis, bifurcation occurs when $b>b_c$. We set 
		\begin{align*}
		&  a_1 =0.015; \quad a_2 = 0.007 ; \quad \delta= 0.6; \quad \beta=0.05; \quad\alpha=0.15; \quad \gamma_1 =0.1; \quad \gamma_2 =0.3 . 
		\end{align*}
		For the above parameters we obtain $b_c =3.34$ and we consider both situations,  where $b$ exceeds $b_c$ and where it is less. As initial conditions we take various perturbations of the steady state \eqref{eq_sol2}, summarized in four simulation scenarios:\\
		
		\noindent
		\textbf{Scenario 1:} $x\in [0,1]$, perturbed initial amount of chondrocytes:
		\begin{align*}
		c_1 (x,0)= c_1^*, \ c_2 (x,0)= c_2^*+0.1\exp(-(x-.5)^2/0.2), \ h (x,0)= h^*,\ b=3.7\ (b=1.8).
		\end{align*}
		
	\begin{figure}[!htbp]
	\begin{subfigure}{0.32\textwidth}
		\centering
		\includegraphics[width=\textwidth]{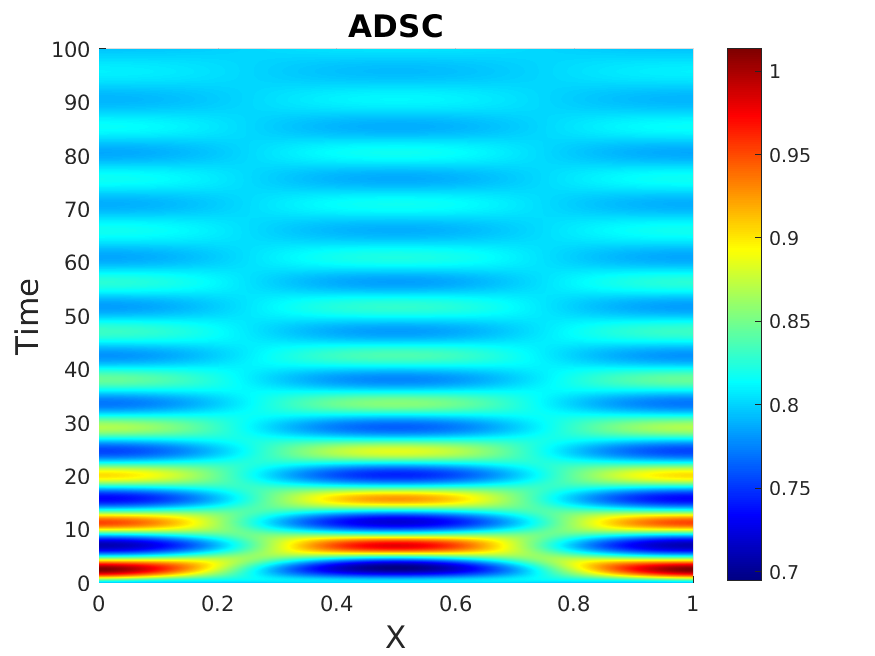}
		\caption{ADSCs for $b=3.7$}
		\label{}
	\end{subfigure}
	\begin{subfigure}{0.32\textwidth}
		\centering
		\includegraphics[width=\textwidth]{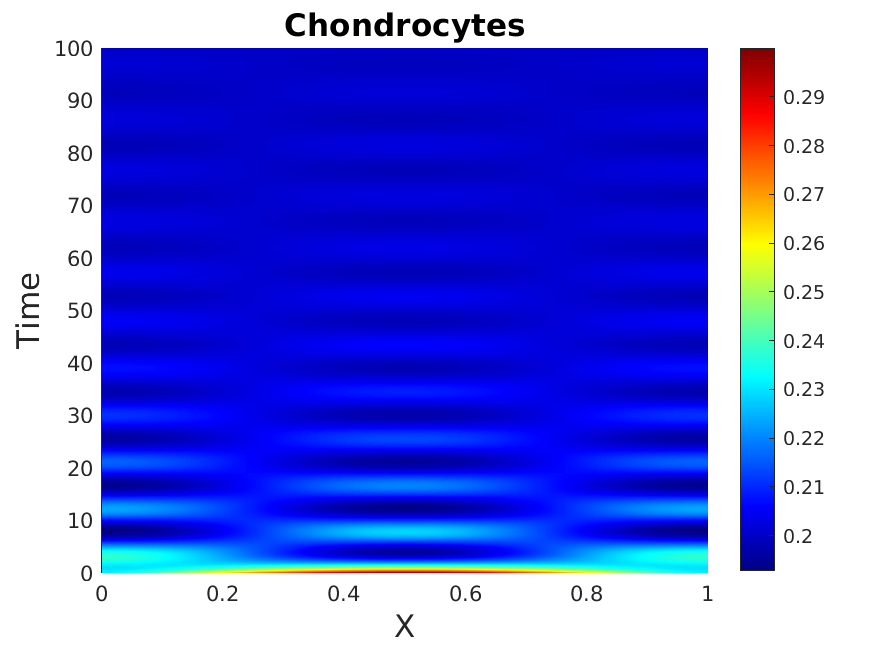}
		\caption{Chondrocytes for $b=3.7$}
	\end{subfigure}	
\begin{subfigure}{0.32\textwidth}
	\centering
	\includegraphics[width=\textwidth]{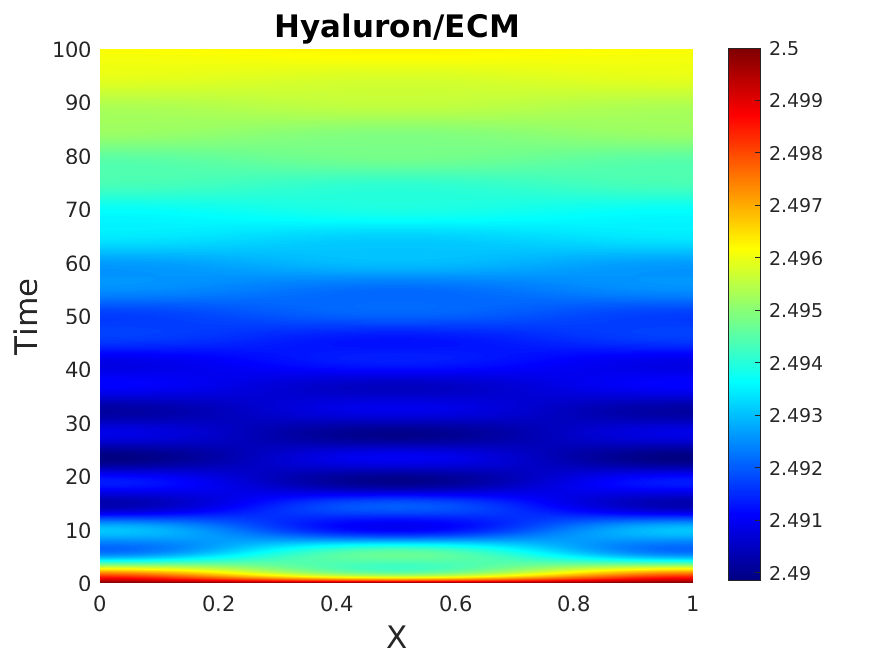}
	\caption{Hyaluron \& ECM for $b=3.7$}
\end{subfigure}	\\
\begin{subfigure}{0.32\textwidth}
	\centering
	\includegraphics[width=\textwidth]{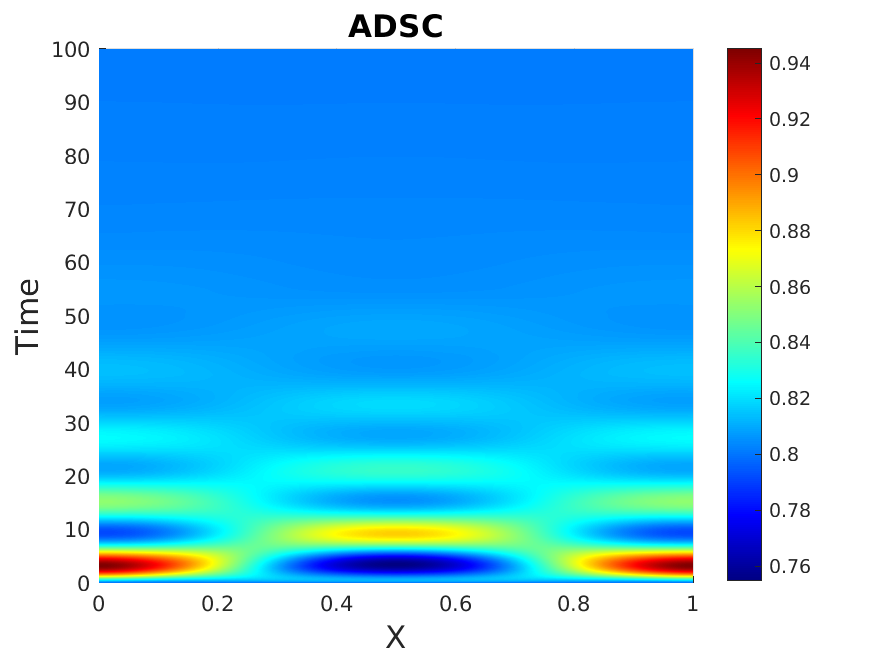}
	\caption{ADSCs for $b=1.8$}
	\label{}
\end{subfigure}
\begin{subfigure}{0.32\textwidth}
	\centering
	\includegraphics[width=\textwidth]{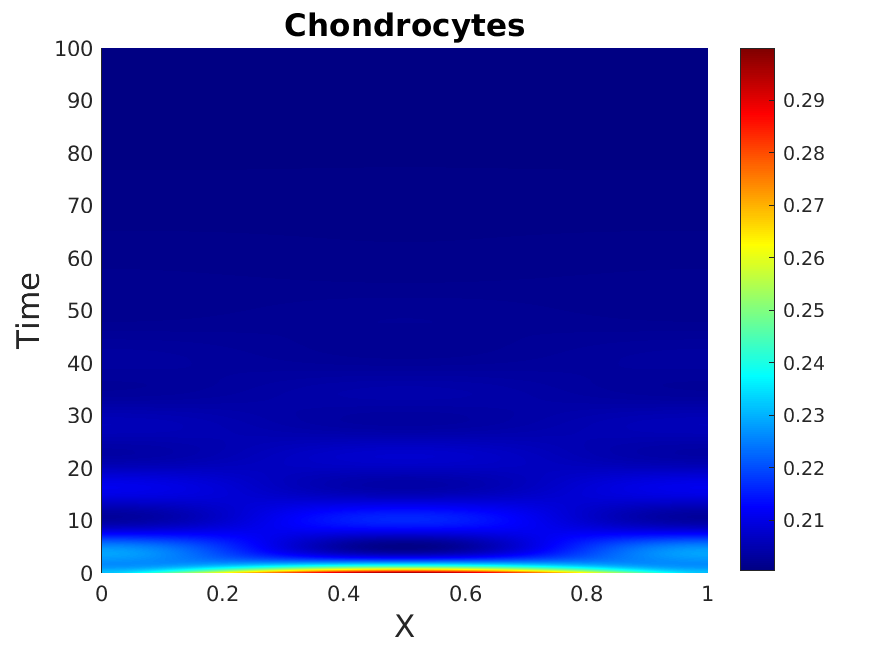}
	\caption{Chondrocytes for $b=1.8$}
\end{subfigure}	
\begin{subfigure}{0.32\textwidth}
	\centering
	\includegraphics[width=\textwidth]{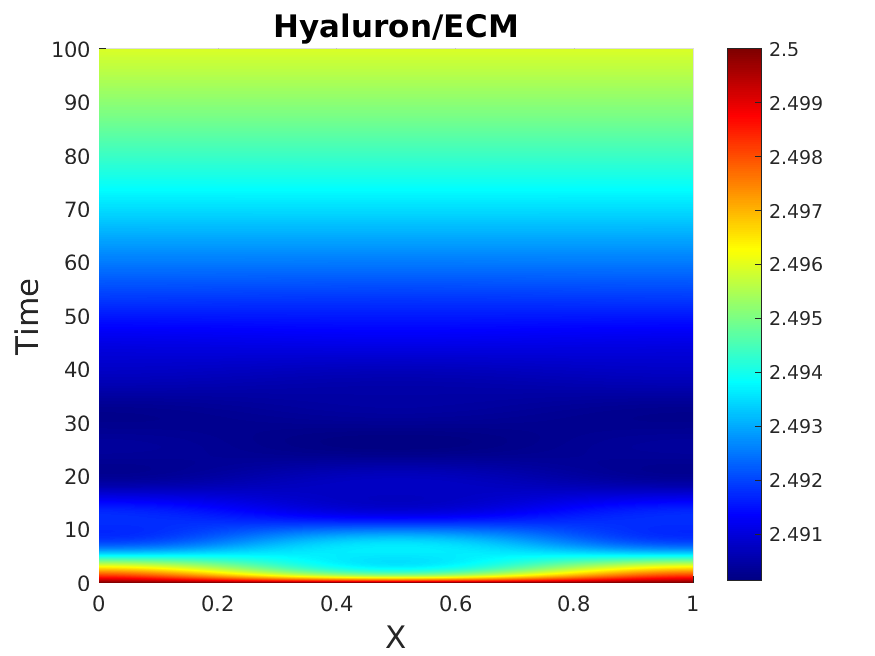}
	\caption{Hyaluron \& ECM for $b=1.8$}
\end{subfigure}
\caption[]{ADSC, chondrocyte, and hyaluron/ECM density for Scenario 1. Upper row: $b>b_c$, lower row $b<b_c$.}\label{fig:1}
	\end{figure}

\noindent
Figure \ref{fig:1} shows the simulated patterns for Scenario 1 when the tactic sensitivity coefficient $b$ exceeds the critical bifurcation value $b_c$ (upper row of plots) and when it is below it (lower row). The former situation leads as expected to oscillatory behavior of all three solution components, accentuated for ADSC and chondrocytes, which infer motility. The oscillations are very slowly damped, mainly due to the cells diffusing and spreading within the whole simulated domain. When $b>b_c$ the incipient perturbation of the system's steady state is quickly losing its effect, which is in line with the result in Proposition \ref{prop:1}.\\

\noindent
Using as initial condition the same perturbation of steady-states, but applying it to different components of the solution leads to changes in the overall behavior of the system, as Figures \ref{fig:2} and \ref{fig:3} produced by Scenarios 2 and 3 are showing, respectively. The former includes a perturbation of ADSC density and Figure \ref{fig:2} exhibits a behavior similar to that in Figure \ref{fig:1}, with the difference that the solution is stabilizing faster and at lower densities. Scenario 3 features the same perturbation, but of hyaluron and ECM, i.e. of the tactic signal. The effect is substantial: the oscillations are far weaker, even for $b>b_c$ and damped rapidly, while higher densities of ADSCs, chondrocytes, and hylauron/ECM are obtained. This seems to be the most favorable case for the envisaged tissue regeneration.\\

	\noindent
\textbf{Scenario 2:} $x\in [0,1]$, perturbed initial amount of ADSCs:	
\begin{align*}
c_1 (x,0)= c_1^*+0.1\exp(-(x-.5)^2/0.2), \ c_2 (x,0)= c_2^*, \ h (x,0)= h^*,\ b=3.7\ (b=1.8).
\end{align*}

	\begin{figure}[!htbp]
	\begin{subfigure}{0.32\textwidth}
		\centering
		\includegraphics[width=\textwidth]{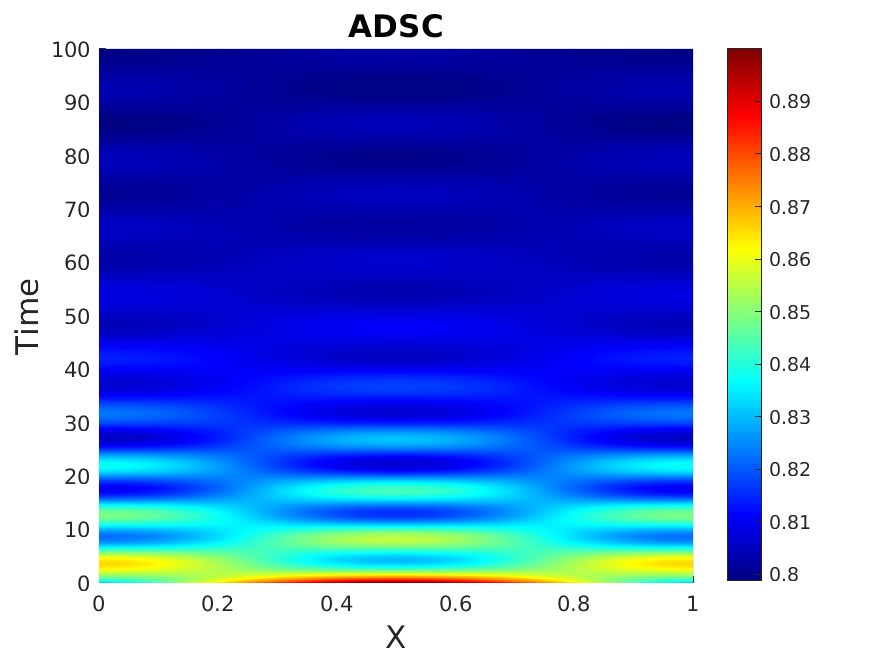}
		\caption{ADSCs for $b=3.7$}
		\label{}
	\end{subfigure}
	\begin{subfigure}{0.32\textwidth}
		\centering
		\includegraphics[width=\textwidth]{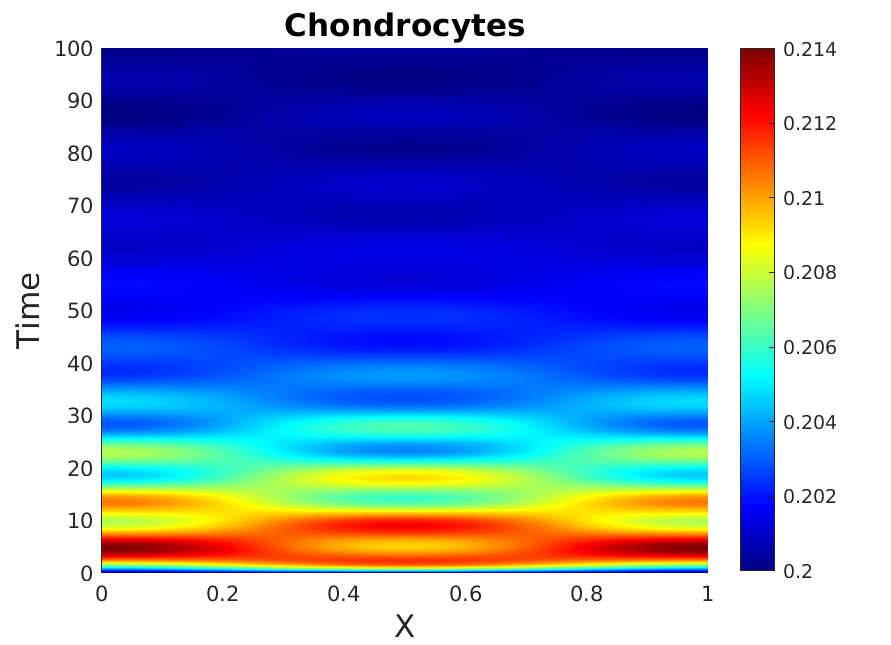}
		\caption{Chondrocytes for $b=3.7$}
	\end{subfigure}	
	\begin{subfigure}{0.32\textwidth}
		\centering
		\includegraphics[width=\textwidth]{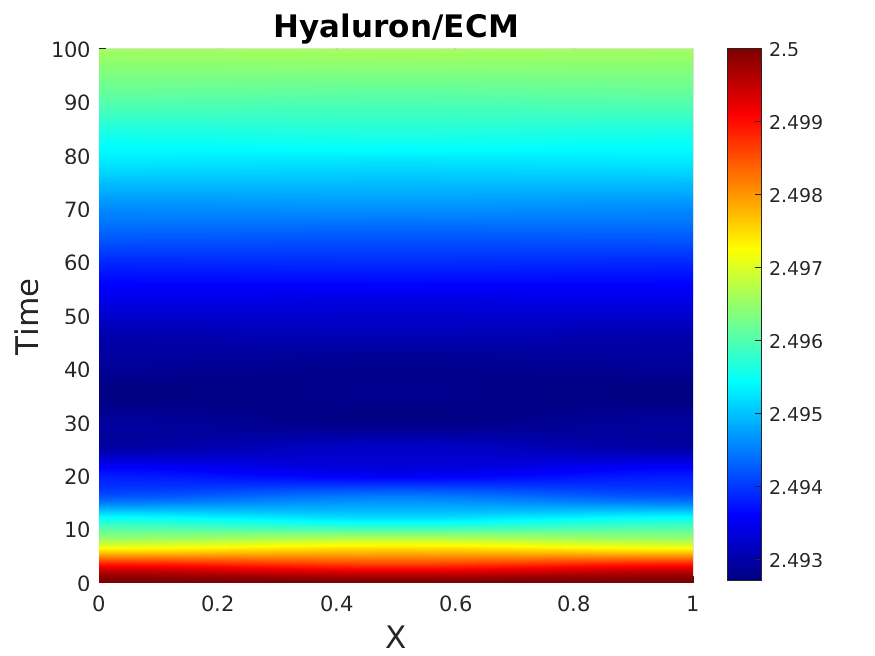}
		\caption{Hyaluron \& ECM for $b=3.7$}
	\end{subfigure}	\\
	\begin{subfigure}{0.32\textwidth}
		\centering
		\includegraphics[width=\textwidth]{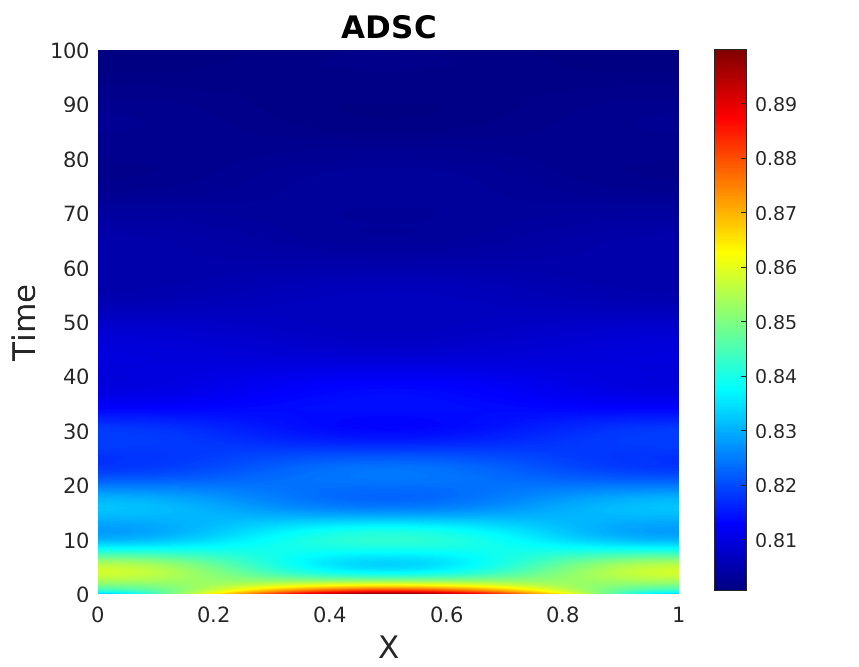}
		\caption{ADSCs for $b=1.8$}
		\label{}
	\end{subfigure}
	\begin{subfigure}{0.32\textwidth}
		\centering
		\includegraphics[width=\textwidth]{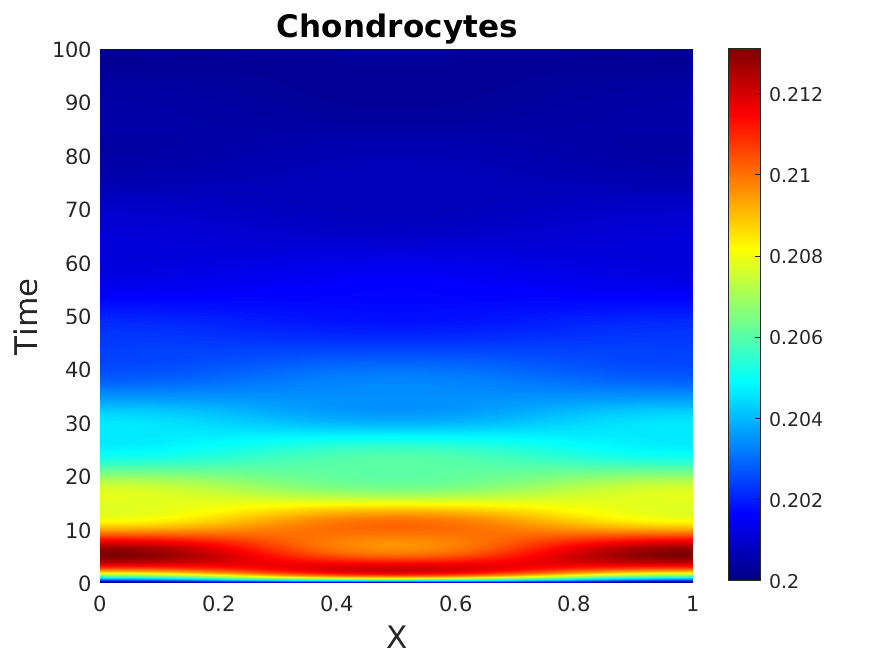}
		\caption{Chondrocytes for $b=1.8$}
	\end{subfigure}	
	\begin{subfigure}{0.32\textwidth}
		\centering
		\includegraphics[width=\textwidth]{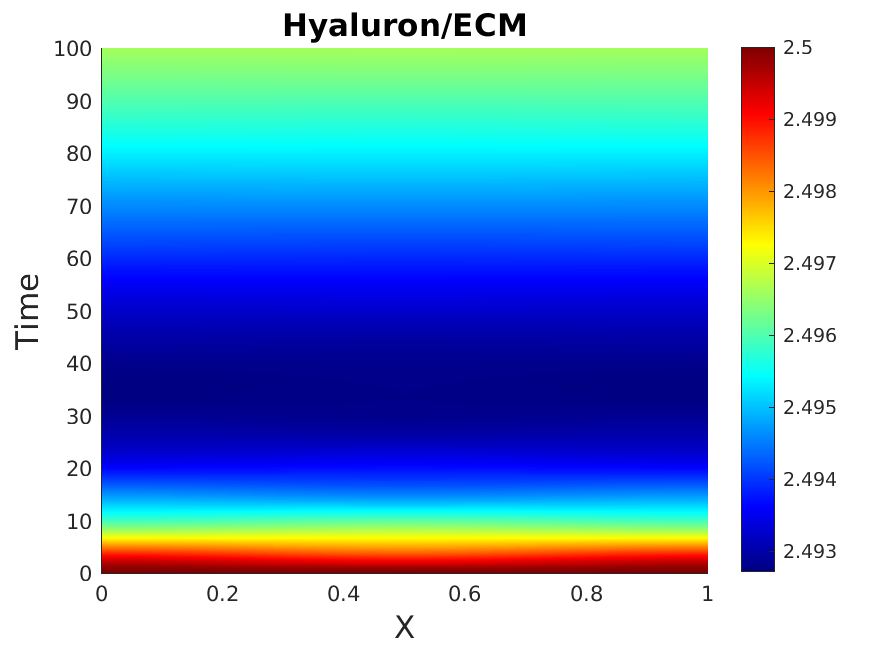}
		\caption{Hyaluron \& ECM for $b=1.8$}
	\end{subfigure}
	\caption[]{ADSC, chondrocyte, and hyaluron/ECM density for Scenario 2. Upper row: $b>b_c$, lower row $b<b_c$.}\label{fig:2}
\end{figure}

\noindent
\textbf{Scenario 3:} $x\in [0,1]$, perturbed initial amount of hyaluron \& ECM:	
\begin{align*}
c_1 (x,0)= c_1^*, \ c_2 (x,0)= c_2^*, \ h (x,0)= h^*+0.1\exp(-(x-.5)^2/0.2),\ b=3.7.
\end{align*}

	\begin{figure}[!htbp]
	\begin{subfigure}{0.32\textwidth}
		\centering
		\includegraphics[width=\textwidth]{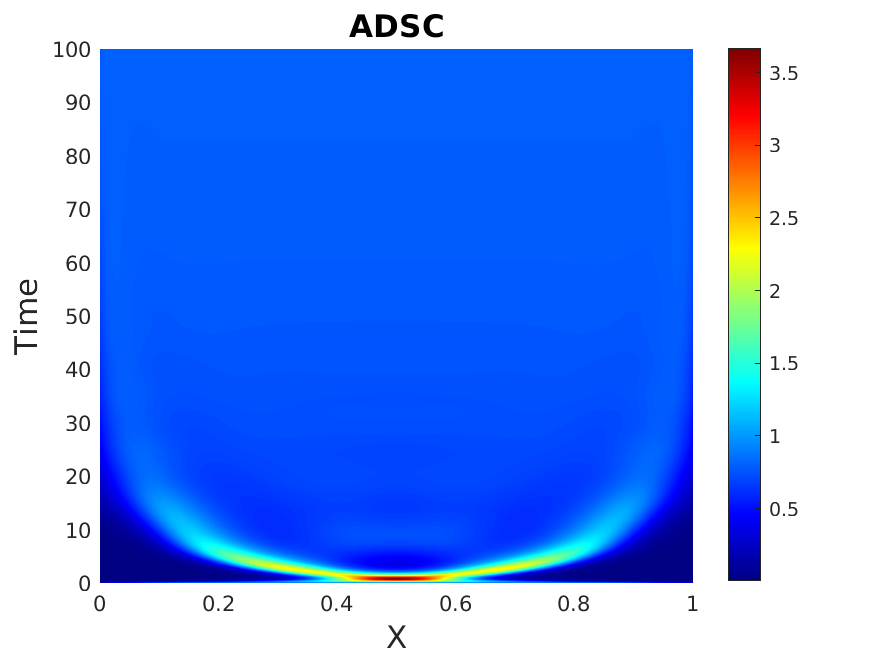}
		\caption{ADSCs for $b=3.7$}
		\label{}
	\end{subfigure}
	\begin{subfigure}{0.32\textwidth}
		\centering
		\includegraphics[width=\textwidth]{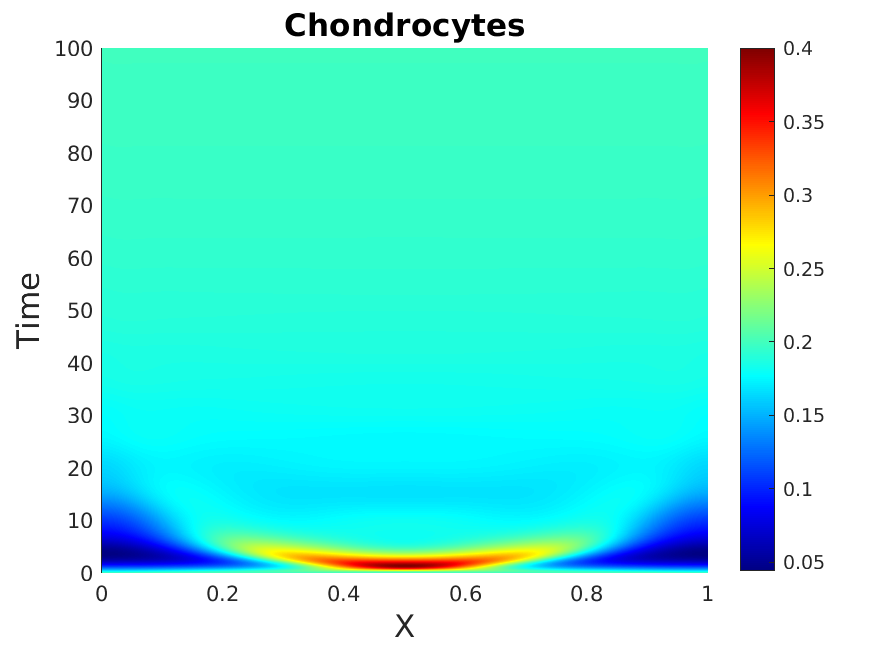}
		\caption{Chondrocytes for $b=3.7$}
	\end{subfigure}	
	\begin{subfigure}{0.32\textwidth}
		\centering
		\includegraphics[width=\textwidth]{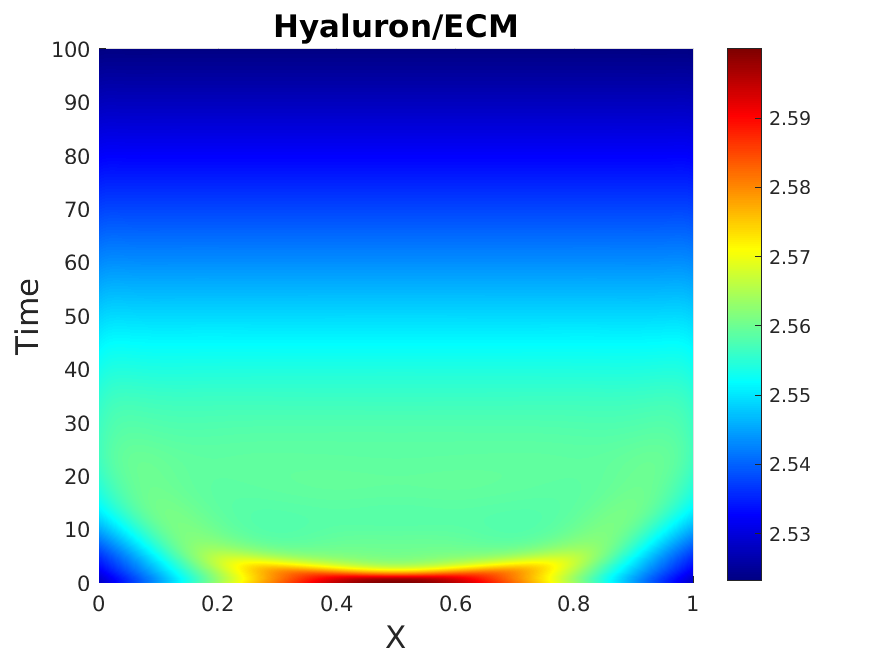}
		\caption{Hyaluron \& ECM for $b=3.7$}
	\end{subfigure}
	\caption[]{ADSC, chondrocyte, and hyaluron/ECM density for Scenario 3, $b>b_c$.}\label{fig:3}
\end{figure}

\noindent
In the following Scenario 4 we consider initial periodic perturbations of the ADSC and hyaluron/ECM steady-states and enlarge the simulation domain ten times. The simulation results are shown in Figure \ref{fig:4}. The periodic patterns are somewhat reminiscent to the used perturbation with cosines, but only occur in this pregnance when the domain is large enough. The alternation of large and lower densities in the solution components is due to ADSCs performing taxis towards gradients of $h$, whose dynamics is almost entirely controled by chondrocytes. As in the previous figures, a small enough tactic sensitivity $b<b_c$ leads to quicker stabilization to lower densities of cells (of both phenotypes) and tissue.\\

\noindent
\textbf{Scenario 4:} $x\in [0,10]$, perturbed initial amounts of ADSC and hyaluron \& ECM:	
\begin{align*}
&c_1 (x,0)= c_1^*+.01\cos(4\pi (x-5)/10), \ c_2 (x,0)= c_2^*, \ h (x,0)= h^*+.01\cos(4\pi (x-5)/10),\\
&b=3.7\ (b=1.8).
\end{align*}

	\begin{figure}[!htbp]
	\begin{subfigure}{0.32\textwidth}
		\centering
		\includegraphics[width=\textwidth]{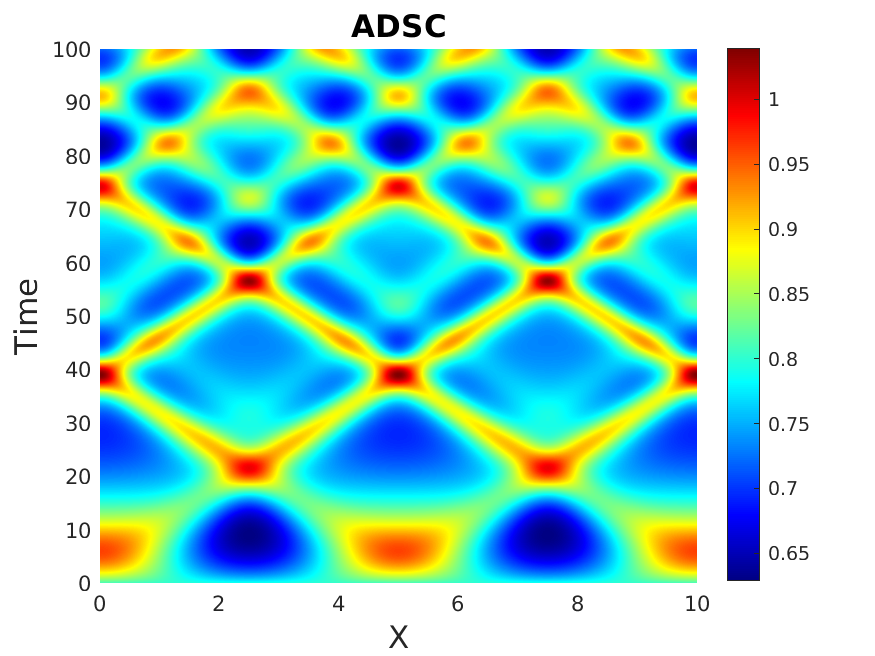}
		\caption{ADSCs for $b=3.7$}
		\label{}
	\end{subfigure}
	\begin{subfigure}{0.32\textwidth}
		\centering
		\includegraphics[width=\textwidth]{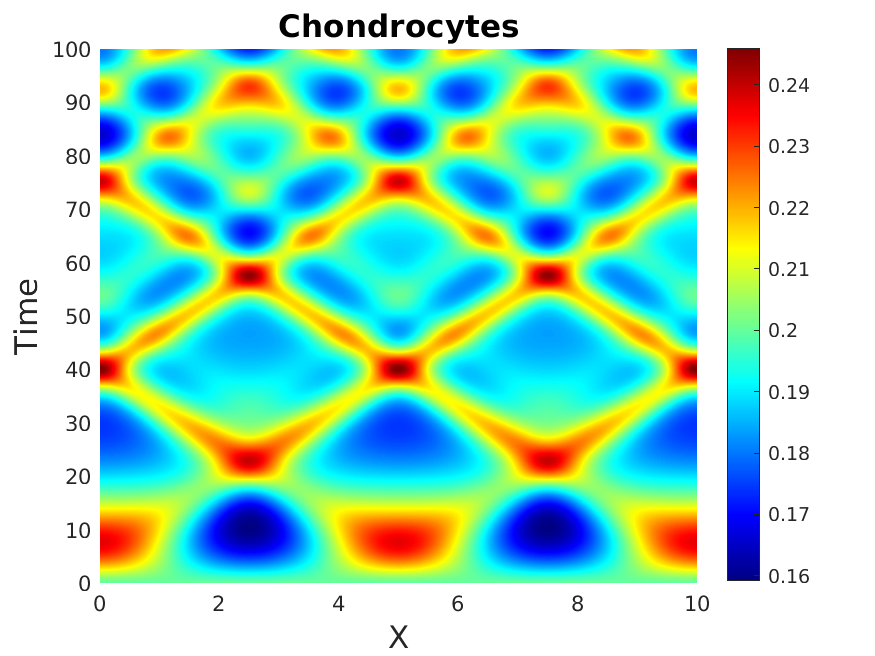}
		\caption{Chondrocytes for $b=3.7$}
	\end{subfigure}	
	\begin{subfigure}{0.32\textwidth}
		\centering
		\includegraphics[width=\textwidth]{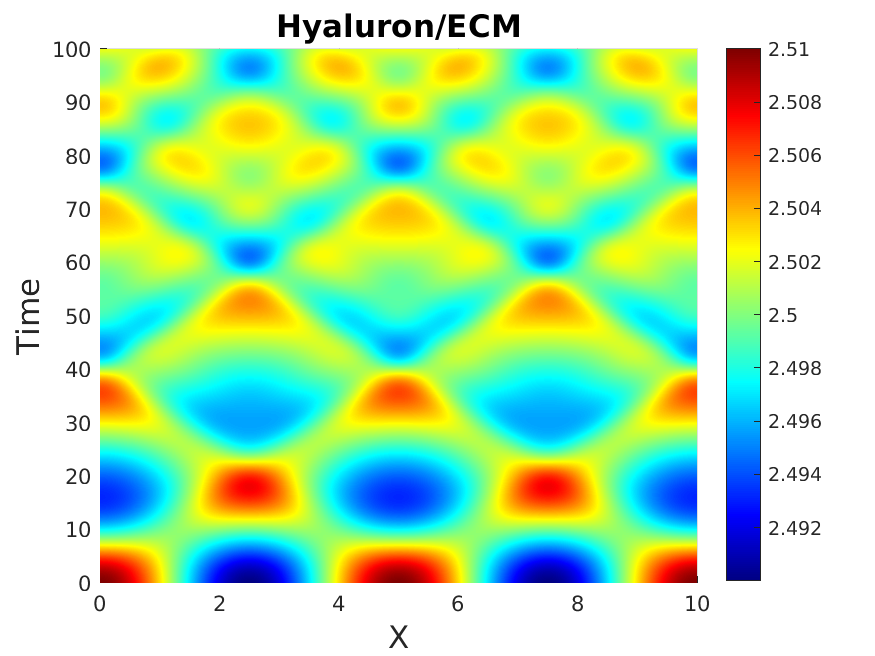}
		\caption{Hyaluron \& ECM for $b=3.7$}
	\end{subfigure}	\\
	\begin{subfigure}{0.32\textwidth}
		\centering
		\includegraphics[width=\textwidth]{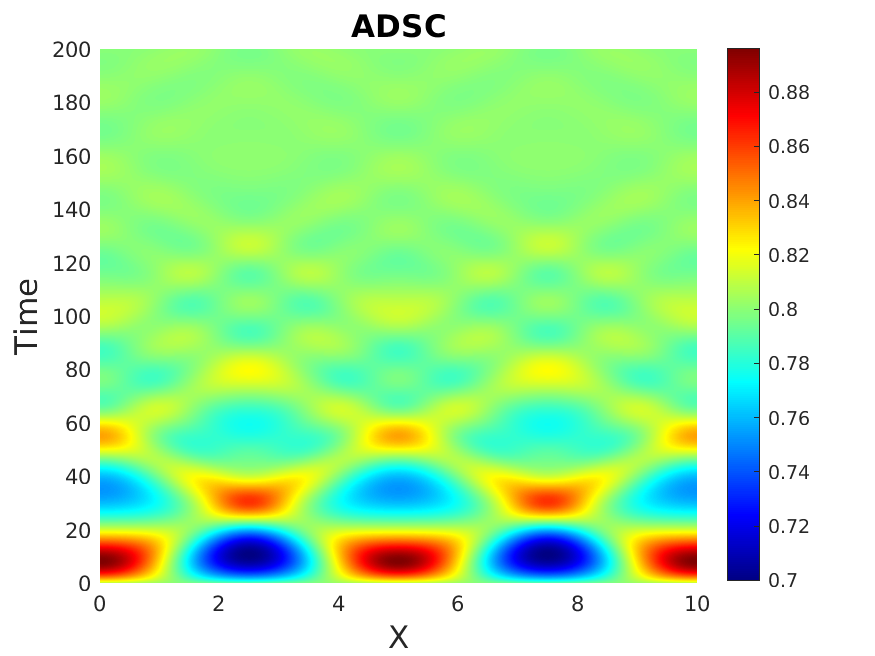}
		\caption{ADSCs for $b=1.8$}
		\label{}
	\end{subfigure}
	\begin{subfigure}{0.32\textwidth}
		\centering
		\includegraphics[width=\textwidth]{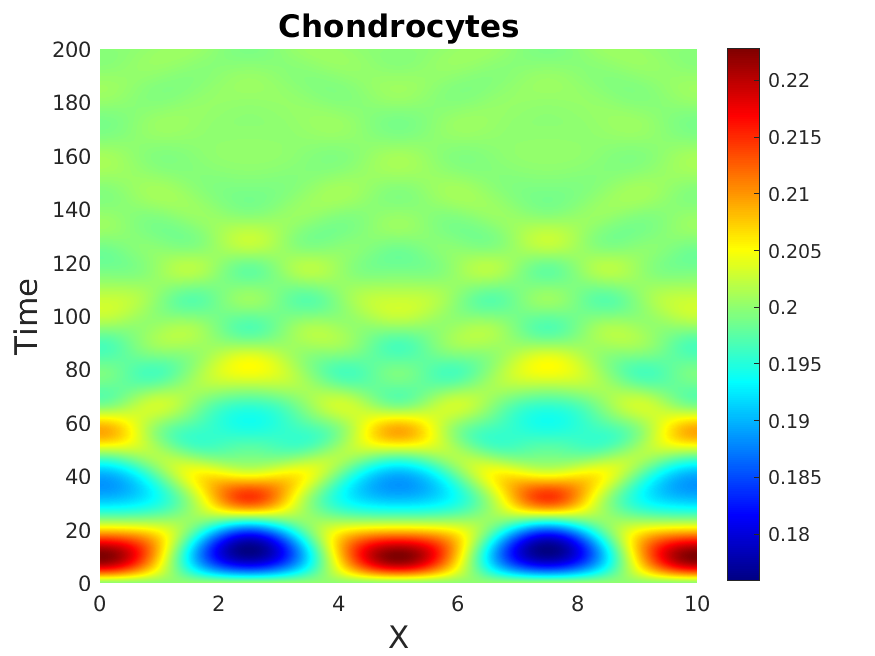}
		\caption{Chondrocytes for $b=1.8$}
	\end{subfigure}	
	\begin{subfigure}{0.32\textwidth}
		\centering
		\includegraphics[width=\textwidth]{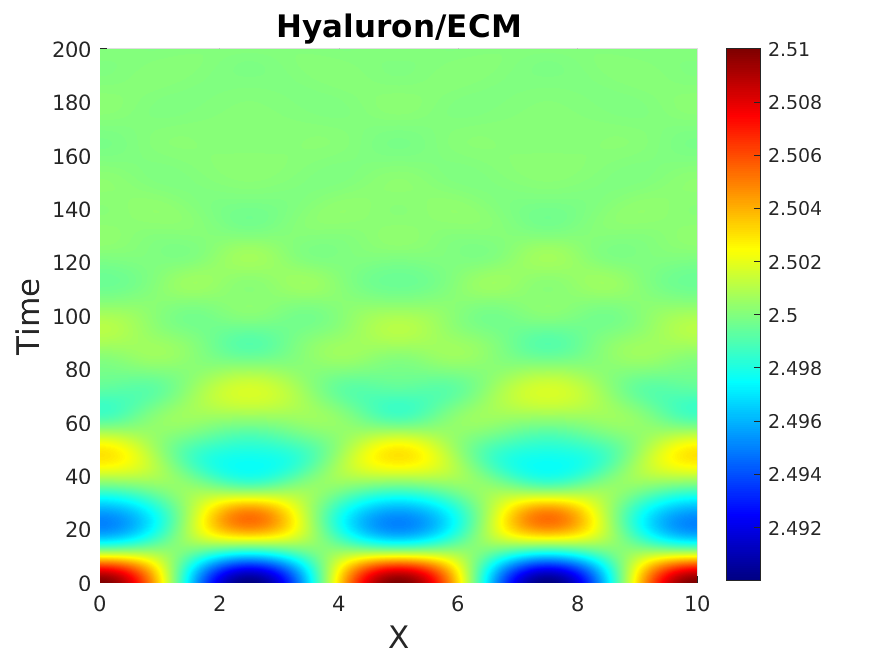}
		\caption{Hyaluron \& ECM for $b=1.8$}
	\end{subfigure}
	\caption[]{ADSC, chondrocyte, and hyaluron/ECM density for Scenario 4. Upper row: $b>b_c$, lower row $b<b_c$.}\label{fig:4}
\end{figure}

\subsection{2D simulations}
		
\noindent		
The 2D discretization of the model is done using the finite difference method (FDM). A standard central difference scheme is used to discretize the diffusion parts of the equations for ADSCs and chondrocytes, while the taxis term in the ADSC equation was handled by a first order upwind scheme. For the time derivatives an implicit-explicit (IMEX) scheme is used, thereby treating the diffusion parts implicitly and discretizing the taxis and source terms with an explicit Euler method. \\[-2ex]

\noindent
The initial conditions considered in this case are
\begin{align}\label{eq:ICs-2D}
&c_1 (x,0)= c_1^*+10^{-6}\exp(-((x-5)^2+(y-5)^2)/.2), \notag \\ 
&c_2 (x,0)= c_2^*+10^{-9}\exp(-((x-5)^2+(y-5)^2)/.2), \ h (x,0)= 1+10^{-6} U,
\end{align}
where $U$ represents a uniform distribution within $(0,1)$. Figure \ref{fig:IC-2d} shows their plots.\\

	\begin{figure}[!htbp]
	\begin{subfigure}{0.32\textwidth}
		\centering
		\includegraphics[width=\textwidth]{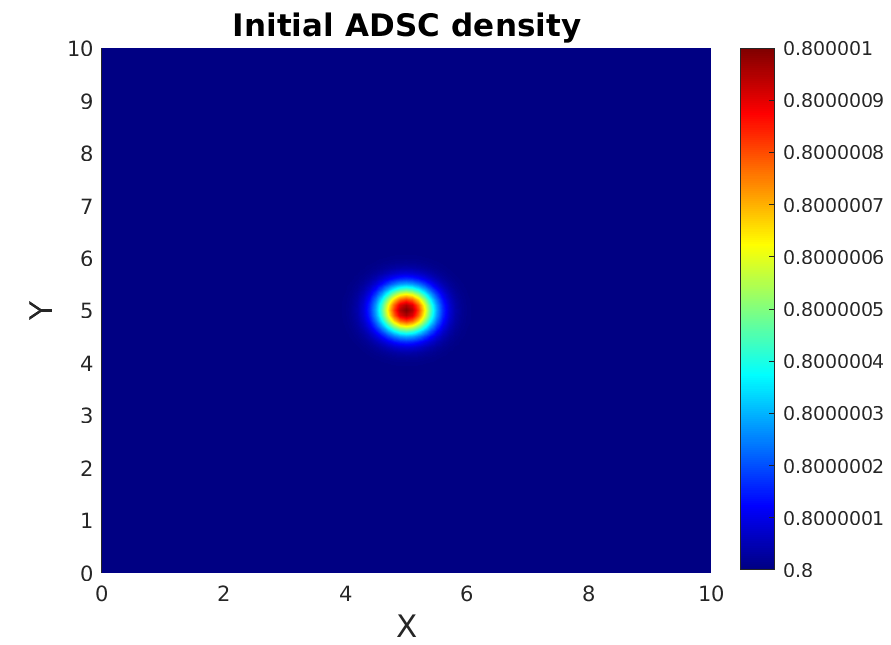}
	\end{subfigure}
	\begin{subfigure}{0.32\textwidth}
		\centering
		\includegraphics[width=\textwidth]{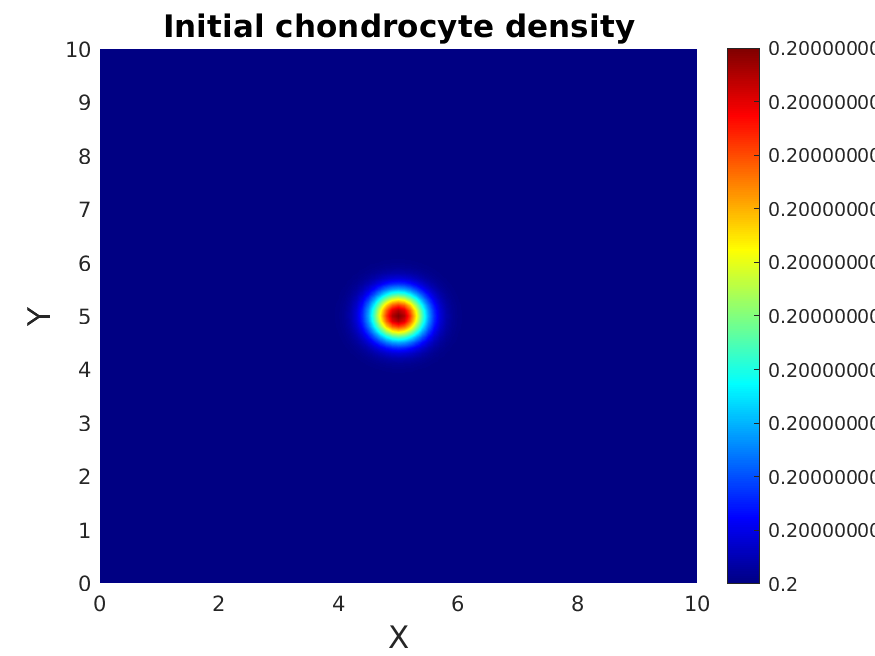}
	\end{subfigure}	
	\begin{subfigure}{0.32\textwidth}
		\centering
		\includegraphics[width=\textwidth]{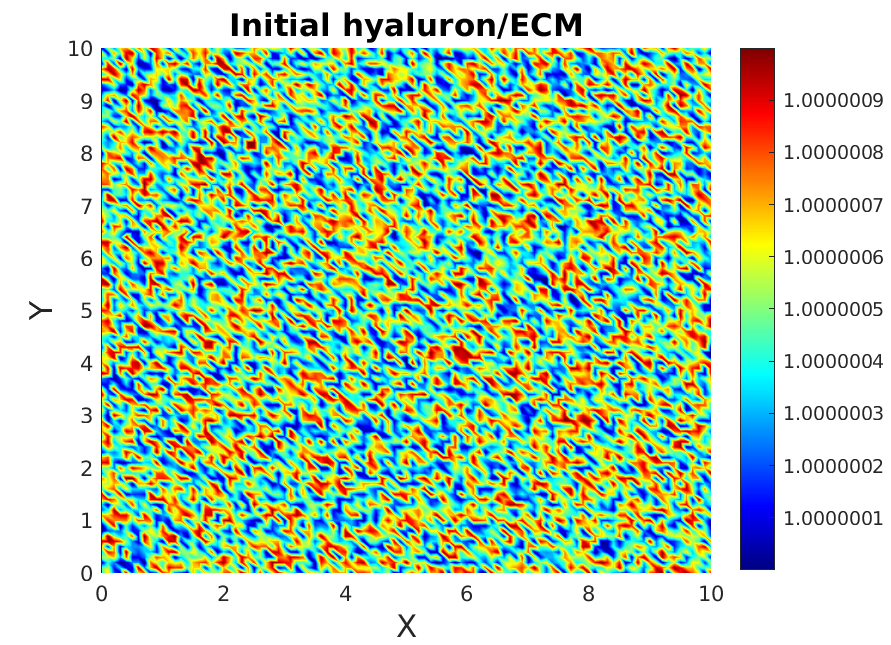}
	\end{subfigure}
	\caption[]{Initial conditions \eqref{eq:ICs-2D} for ADSC, chondrocyte, and hyaluron/ECM density.}\label{fig:IC-2d}
\end{figure}

\begin{figure}[!htbp]
	\begin{subfigure}{0.32\textwidth}
		\centering
		\includegraphics[width=\textwidth]{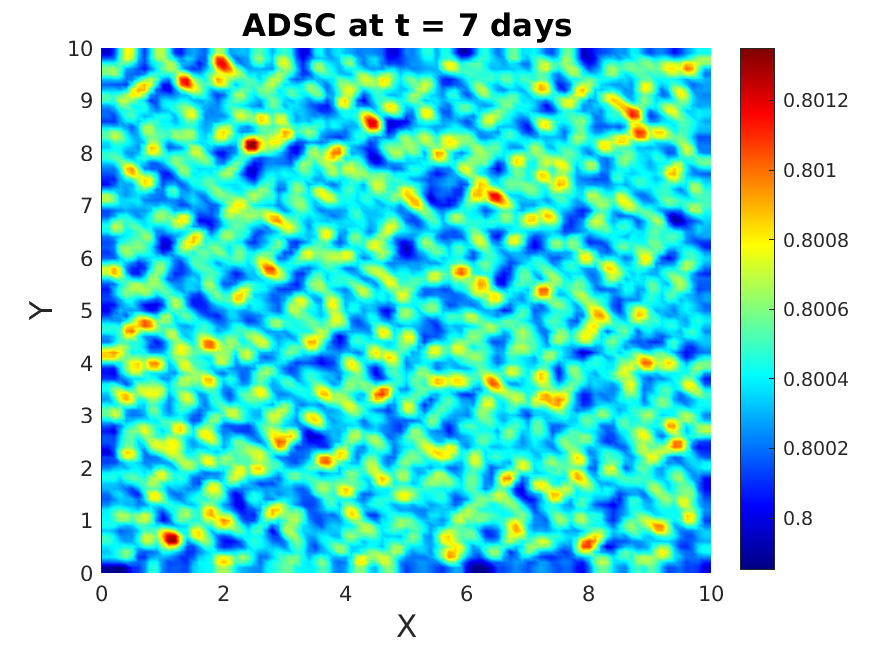}
		\caption{ADSCs at 7 days}
		\label{}
	\end{subfigure}
	\begin{subfigure}{0.32\textwidth}
		\centering
		\includegraphics[width=\textwidth]{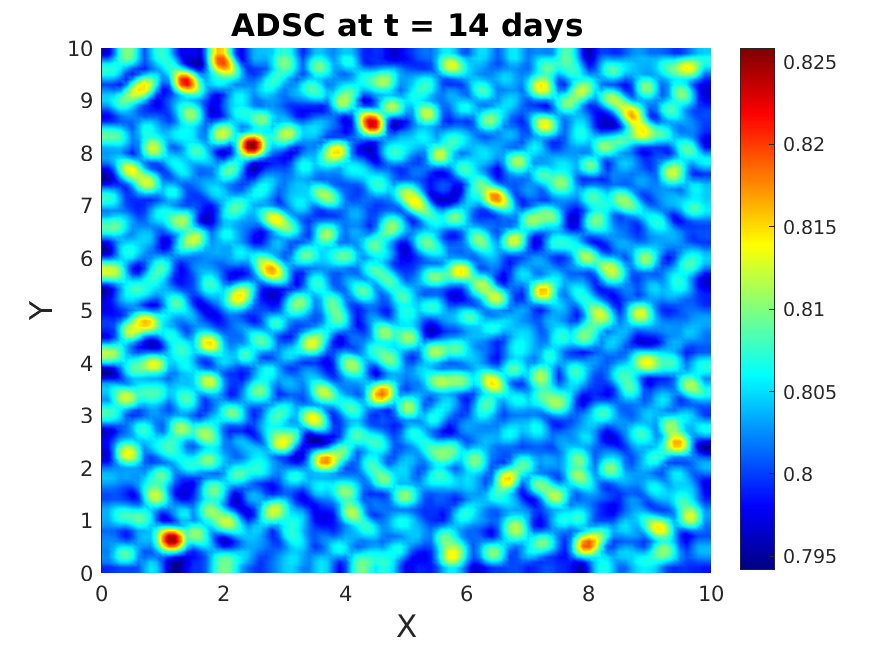}
		\caption{ADSCs at 14 days}
	\end{subfigure}	
	\begin{subfigure}{0.32\textwidth}
		\centering
		\includegraphics[width=\textwidth]{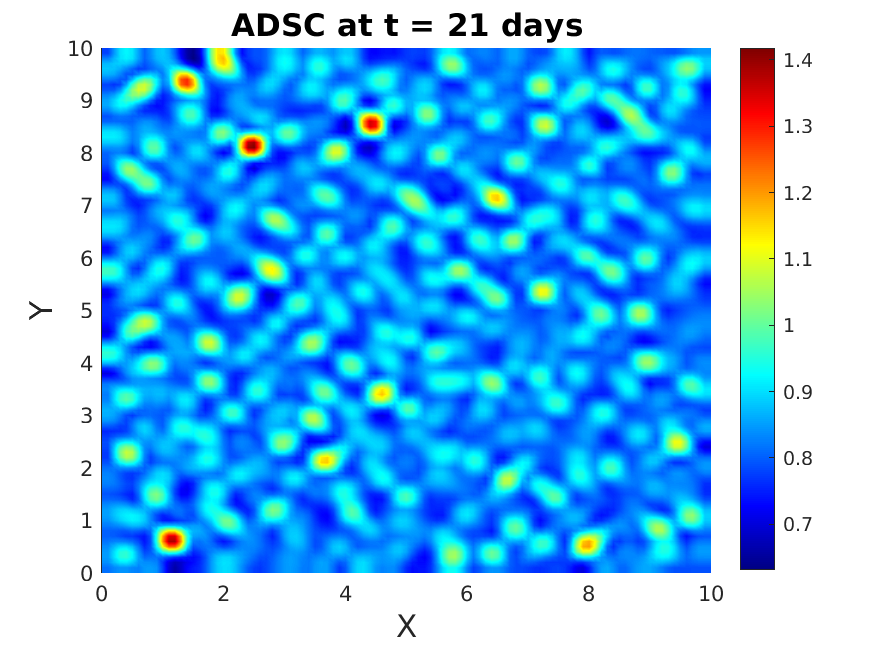}
		\caption{ADSCs at 21 days }
	\end{subfigure}	\\[1ex]
	\begin{subfigure}{0.32\textwidth}
		\centering
		\includegraphics[width=\textwidth]{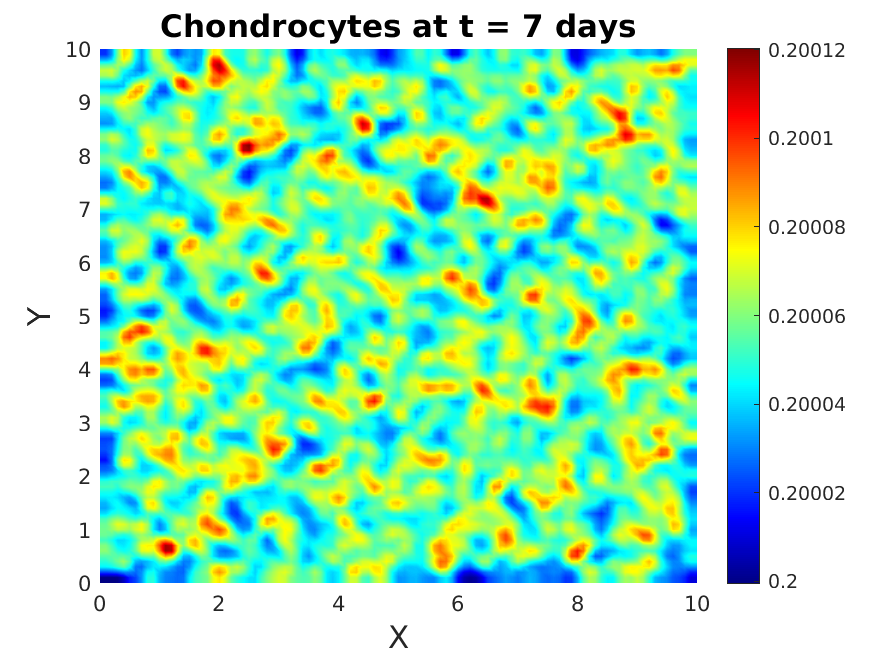}
		\caption{Chondrocytes at 7 days}
		\label{}
	\end{subfigure}
	\begin{subfigure}{0.32\textwidth}
		\centering
		\includegraphics[width=\textwidth]{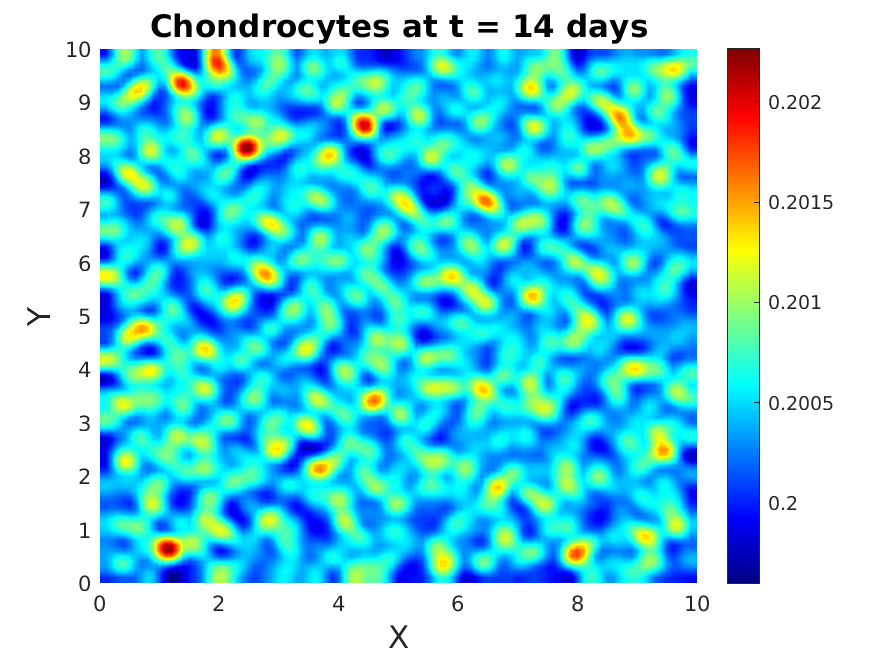}
		\caption{Chondrocytes at 14 days}
	\end{subfigure}	
	\begin{subfigure}{0.32\textwidth}
		\centering
		\includegraphics[width=\textwidth]{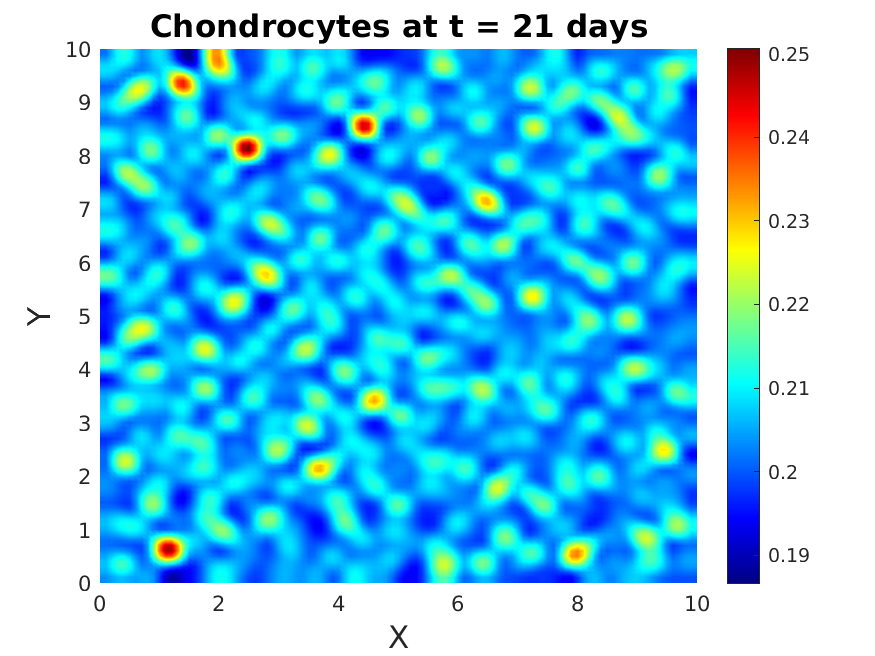}
		\caption{Chondrocytes at 14 days }
	\end{subfigure}\\[1ex]
\begin{subfigure}{0.32\textwidth}
\centering
\includegraphics[width=\textwidth]{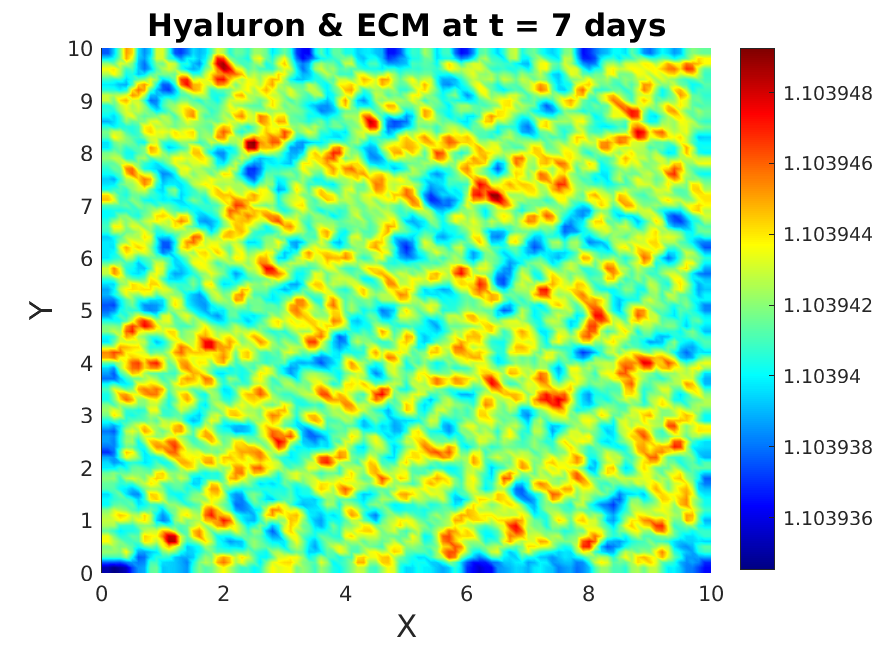}
\caption{Hyaluron \& ECM at 7 days }
\label{}
\end{subfigure}
\begin{subfigure}{0.32\textwidth}
\centering
\includegraphics[width=\textwidth]{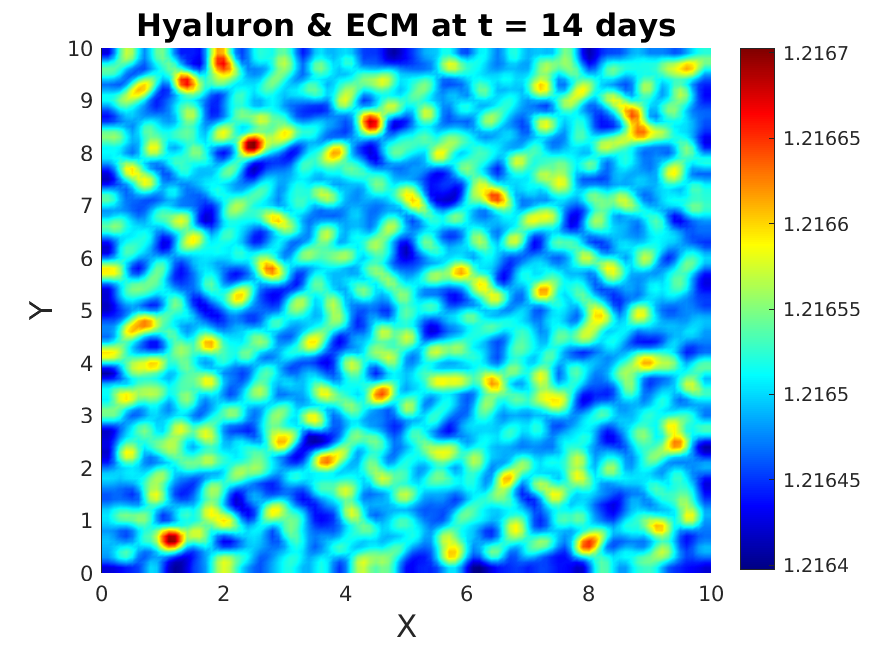}
\caption{Hyaluron \& ECM at 14 days }
\end{subfigure}	
\begin{subfigure}{0.32\textwidth}
\centering
\includegraphics[width=\textwidth]{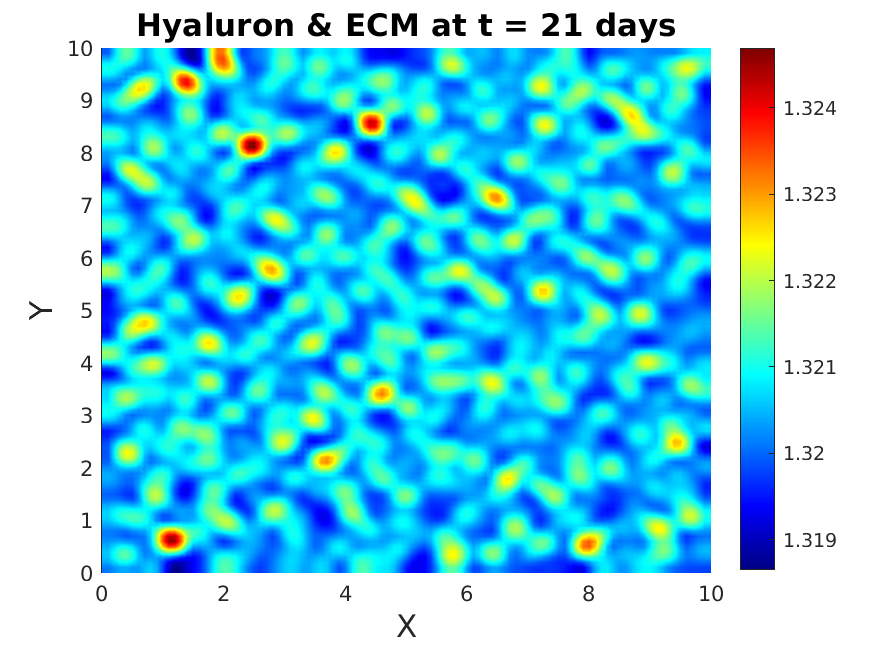}
\caption{Hyaluron \& ECM at 21 days }
\end{subfigure}
	\caption[]{ADSC, chondrocyte, and hyaluron/ECM density at 7, 14, and 21 days, initial conditions \eqref{eq:ICs-2D}, $b>b_c$ (here $b=3.7$)}\label{fig:5}
\end{figure}

\noindent
Figure \ref{fig:5} illustrates the 2D simulation results for system \eqref{model} under conditions \eqref{eq:ICs-2D}. Densities of ADSCs, chondrocytes, and hyaluron/ECM are shown at 7, 14, and 21 days. The tactic sensitivity parameter $b=3.7$ exceeds the threshold obtained in Section \ref{sec:stability} as Hopf bifurcation value. The cells are quickly diffusing from the initial positions. The motility of ADSCs is further enhanced by taxis towards gradients of signal $h$ and the chondrocytes follow, as they are only obtained by ADSC differentiation. This behavior can also be observed in Figure \ref{fig:6}, which shows the same evolution of $c_1,\ c_2,\ h$, but for $b=1.8<b_c$. A higher tactic sensitivity of ADSCs seems to lead as in the 1D case to slightly increased densities of cells and tissue.\\

\begin{figure}[!htbp]
	\begin{subfigure}{0.32\textwidth}
		\centering
		\includegraphics[width=\textwidth]{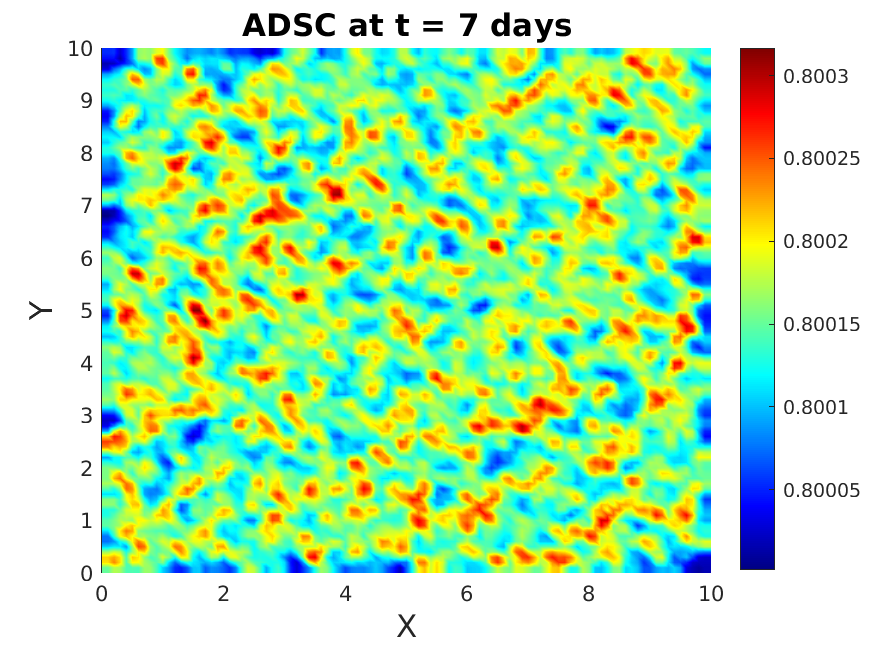}
		\caption{ADSCs at 7 days}
		\label{}
	\end{subfigure}
	\begin{subfigure}{0.32\textwidth}
		\centering
		\includegraphics[width=\textwidth]{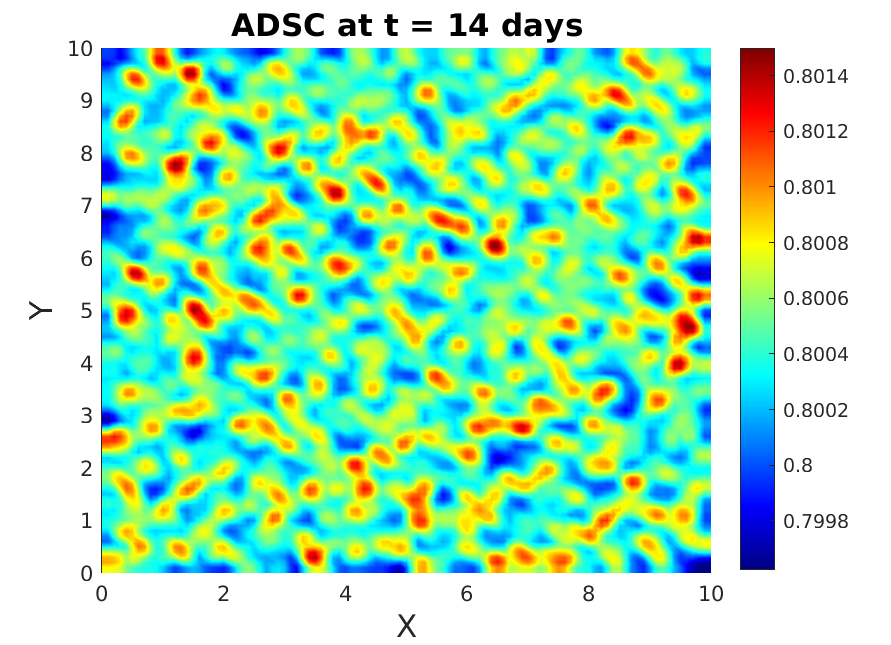}
		\caption{ADSCs at 14 days}
	\end{subfigure}	
	\begin{subfigure}{0.32\textwidth}
		\centering
		\includegraphics[width=\textwidth]{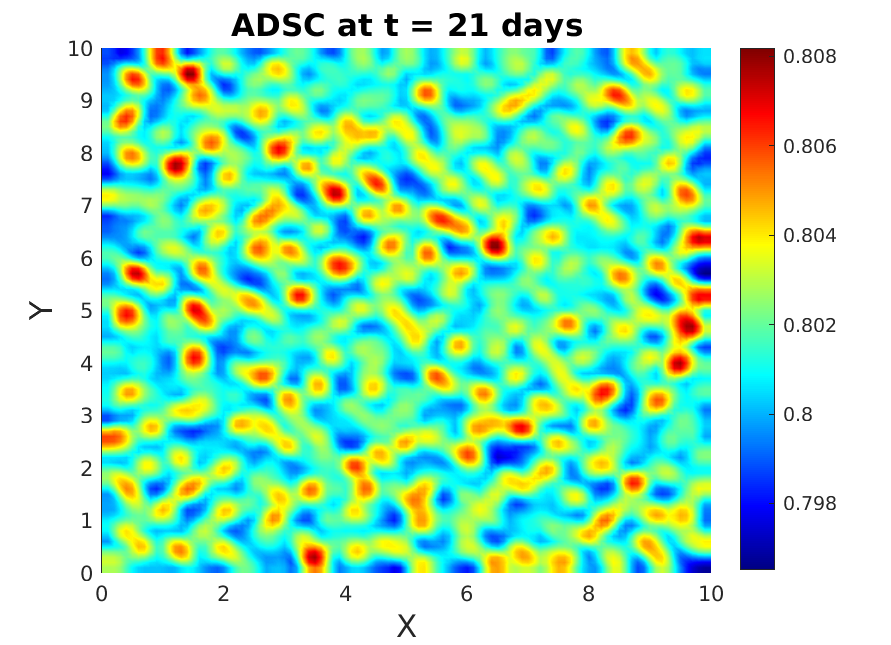}
		\caption{ADSCs at 21 days }
	\end{subfigure}	\\[1ex]
	\begin{subfigure}{0.32\textwidth}
		\centering
		\includegraphics[width=\textwidth]{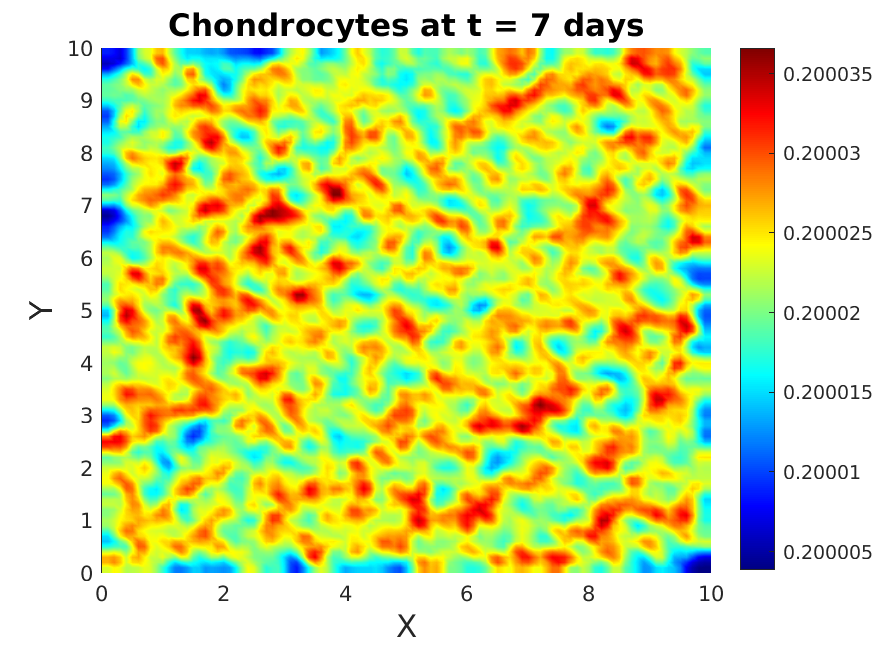}
		\caption{Chondrocytes at 7 days}
		\label{}
	\end{subfigure}
	\begin{subfigure}{0.32\textwidth}
		\centering
		\includegraphics[width=\textwidth]{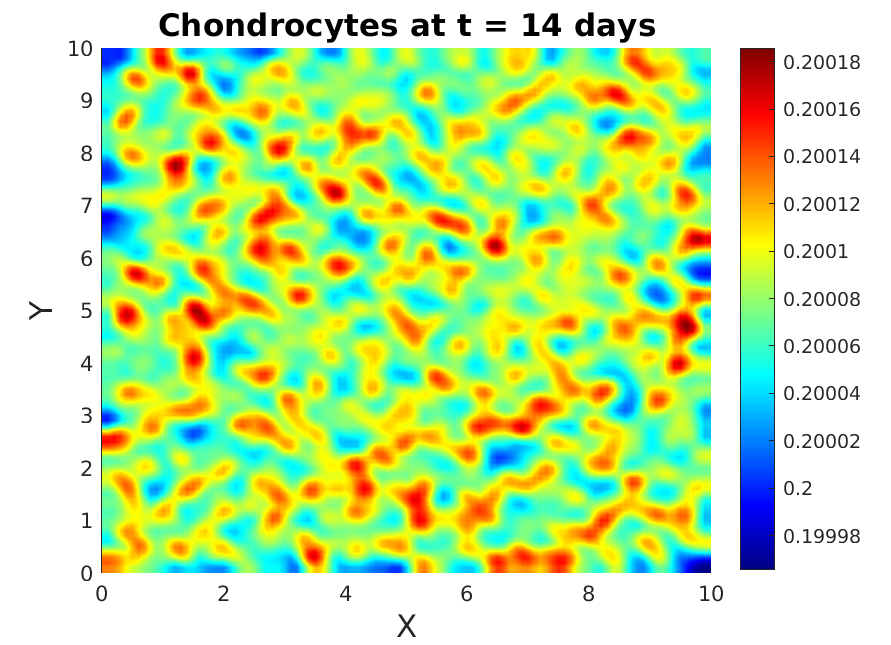}
		\caption{Chondrocytes at 14 days}
	\end{subfigure}	
	\begin{subfigure}{0.32\textwidth}
		\centering
		\includegraphics[width=\textwidth]{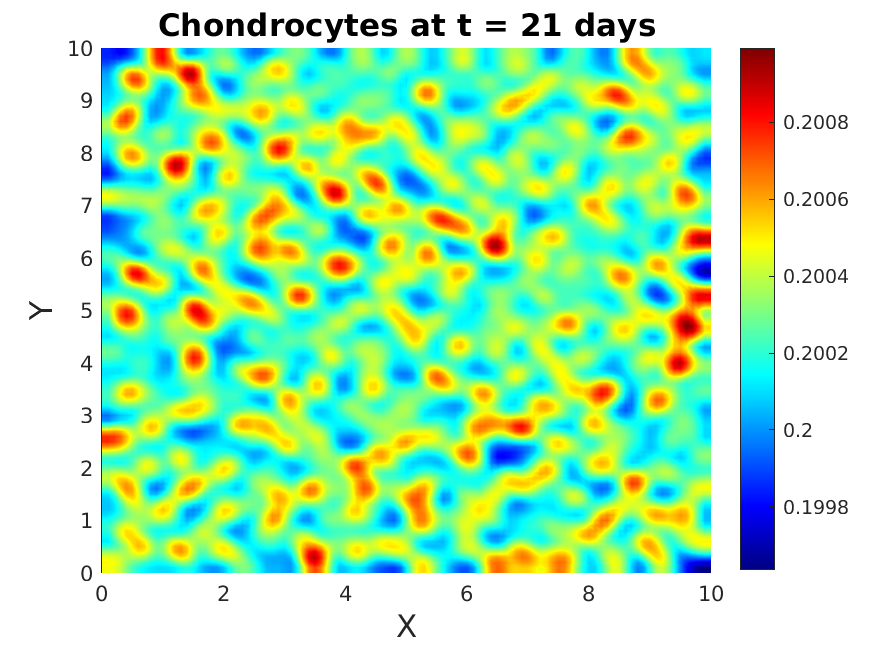}
		\caption{Chondrocytes at 14 days }
	\end{subfigure}\\[1ex]
	\begin{subfigure}{0.32\textwidth}
		\centering
		\includegraphics[width=\textwidth]{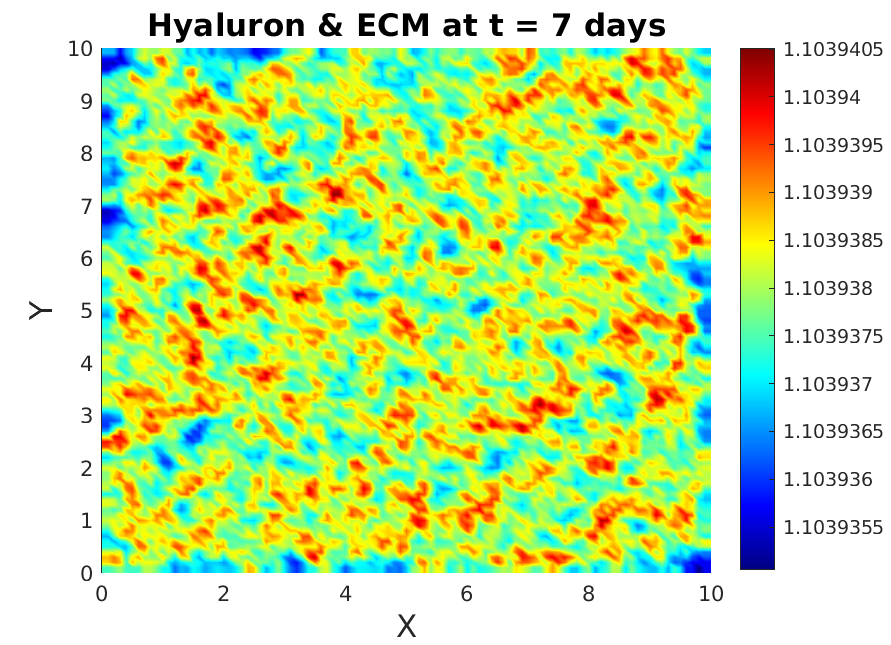}
		\caption{Hyaluron \& ECM at 7 days }
		\label{}
	\end{subfigure}
	\begin{subfigure}{0.32\textwidth}
		\centering
		\includegraphics[width=\textwidth]{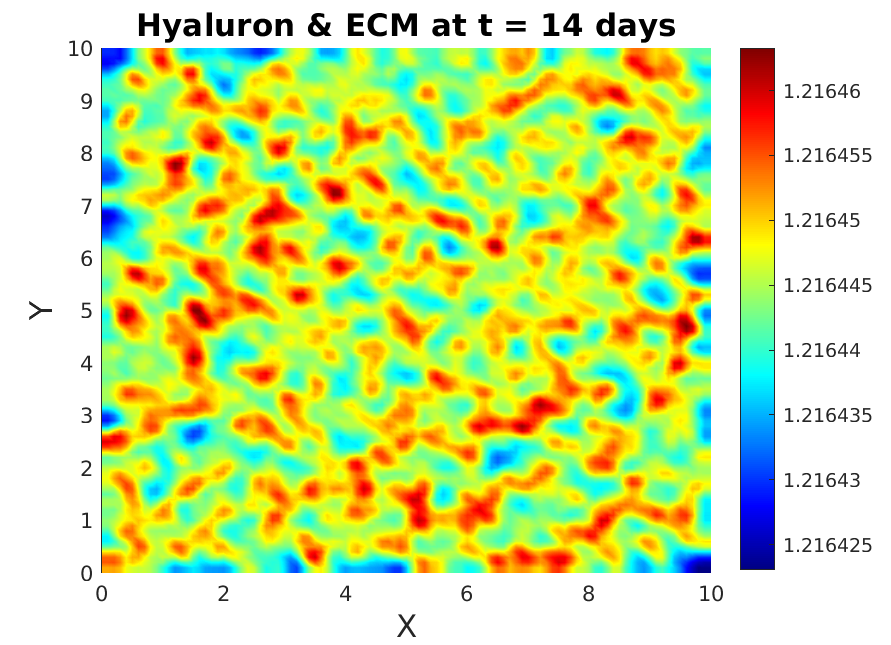}
		\caption{Hyaluron \& ECM at 14 days }
	\end{subfigure}	
	\begin{subfigure}{0.32\textwidth}
		\centering
		\includegraphics[width=\textwidth]{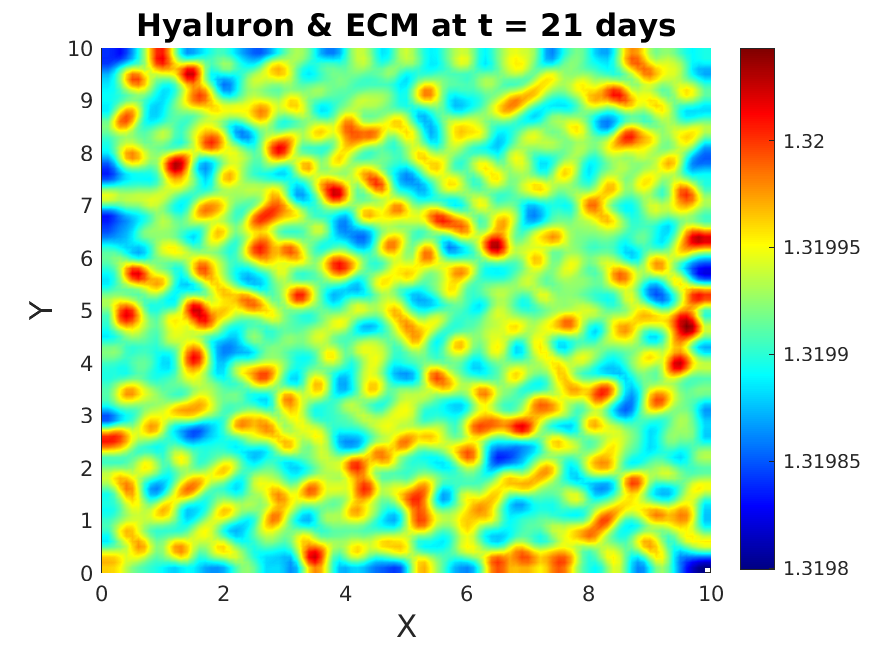}
		\caption{Hyaluron \& ECM at 21 days }
	\end{subfigure}
	\caption[]{ADSC, chondrocyte, and hyaluron/ECM density at 7, 14, and 21 days, initial conditions \eqref{eq:ICs-2D}, $b<b_c$ (here $b=1.8$)}\label{fig:6}
\end{figure}

\noindent
To investigate the effect of initial hyaluron (and hence of tissue) distribution we considered 
\begin{align}\label{eq:ICs-2D-cos}
&c_1 (x,0)= c_1^*+10^{-6}\exp(-((x-5)^2+(y-5)^2)/.2), \\ 
&c_2 (x,0)= c_2^*+10^{-9}\exp(-((x-5)^2+(y-5)^2)/.2), \ h (x,0)= 1+10^{-6} U \cos(4\pi (x-5)/10).\notag 
\end{align}
The corresponding simulation results are shown in Figure \ref{fig:7}. The previously mentioned behavior of cells and tissue is more pregnantly visible: the cells leave the rather concentrated initial blob, the ADSCs follow the tissue signal and make the chondrocytes follow in turn. This allows for higher cell and tissue densities along the 'stripes' generated by the cosine in \eqref{eq:ICs-2D-cos}, suggesting that the initial structure of underlying tissue plays a major role during the regeneration process. Thus, properly designed scaffolds providing support for cell migration can have a relevant influence on the amount and quality of newly produced tissue upon seeding with stem cells and promoting and sustaining differentiation into chondrocytes.\\

\begin{figure}[!htbp]
	\begin{subfigure}{0.32\textwidth}
		\centering
		\includegraphics[width=\textwidth]{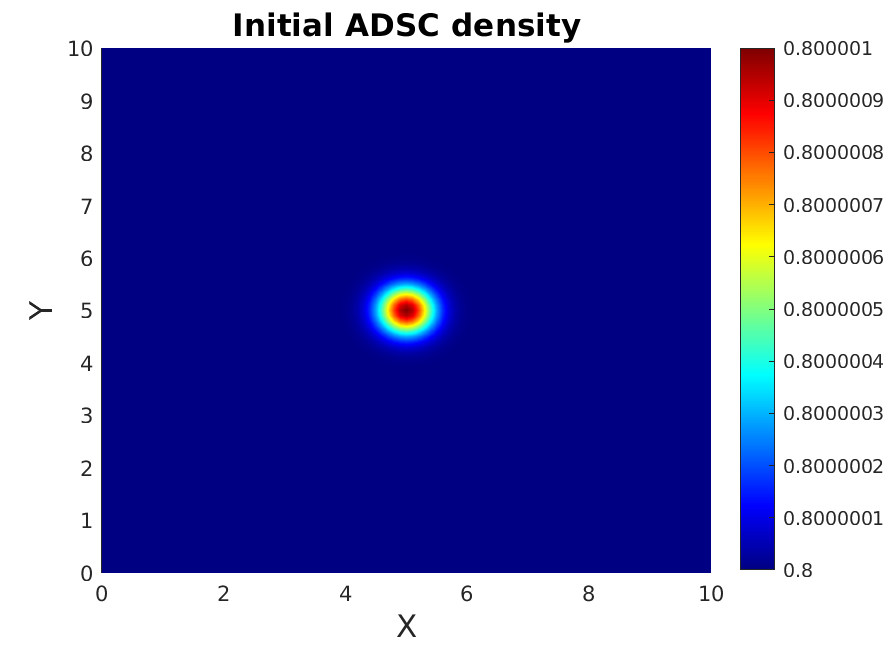}
	\end{subfigure}
	\begin{subfigure}{0.32\textwidth}
		\centering
		\includegraphics[width=\textwidth]{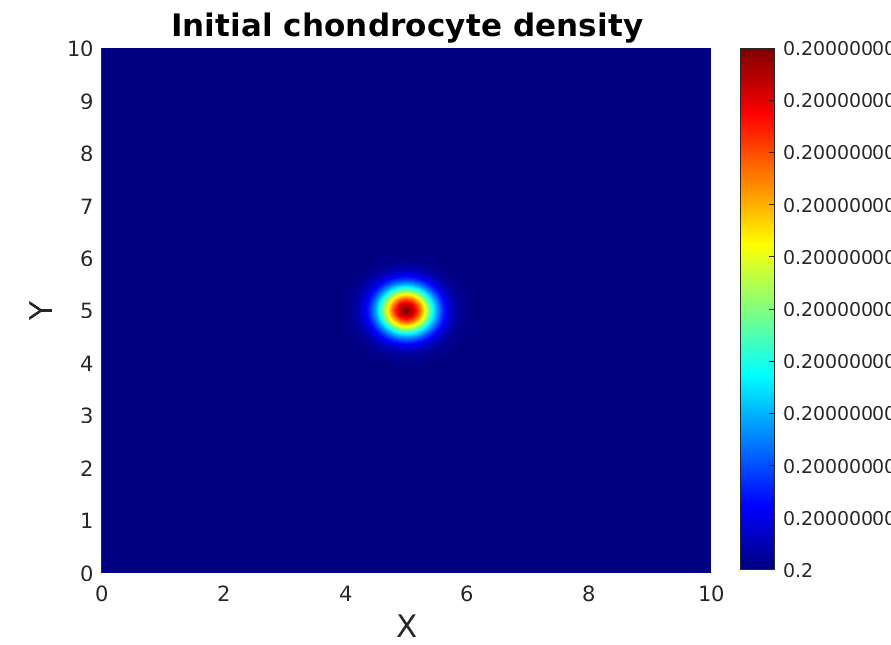}
	\end{subfigure}	
	\begin{subfigure}{0.32\textwidth}
		\centering
		\includegraphics[width=\textwidth]{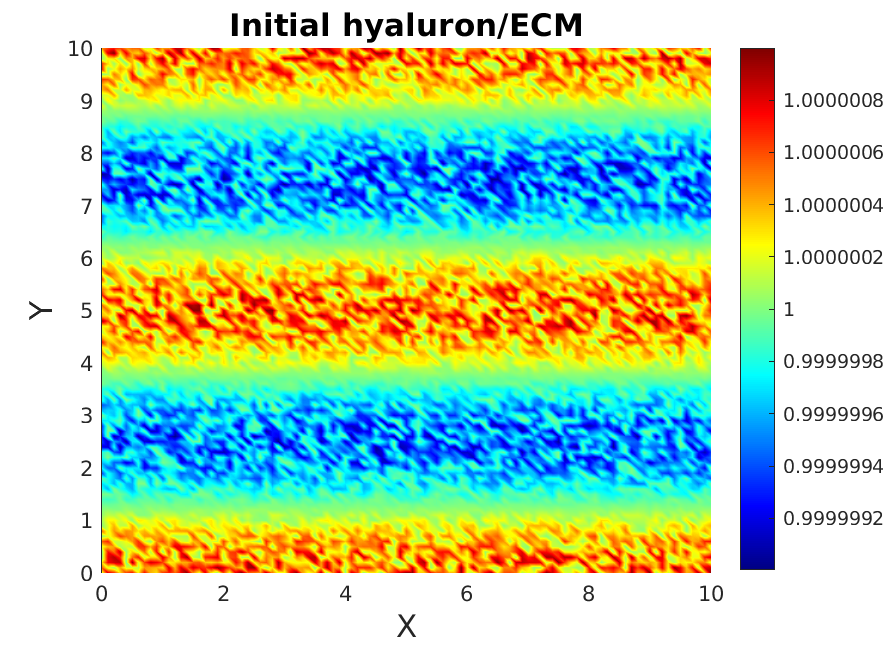}
	\end{subfigure}
	\caption[]{Initial conditions \eqref{eq:ICs-2D-cos} for ADSC, chondrocyte, and hyaluron/ECM density.}\label{fig:IC-2d-cos}
\end{figure}

\begin{figure}[!htbp]
	\begin{subfigure}{0.32\textwidth}
		\centering
		\includegraphics[width=\textwidth]{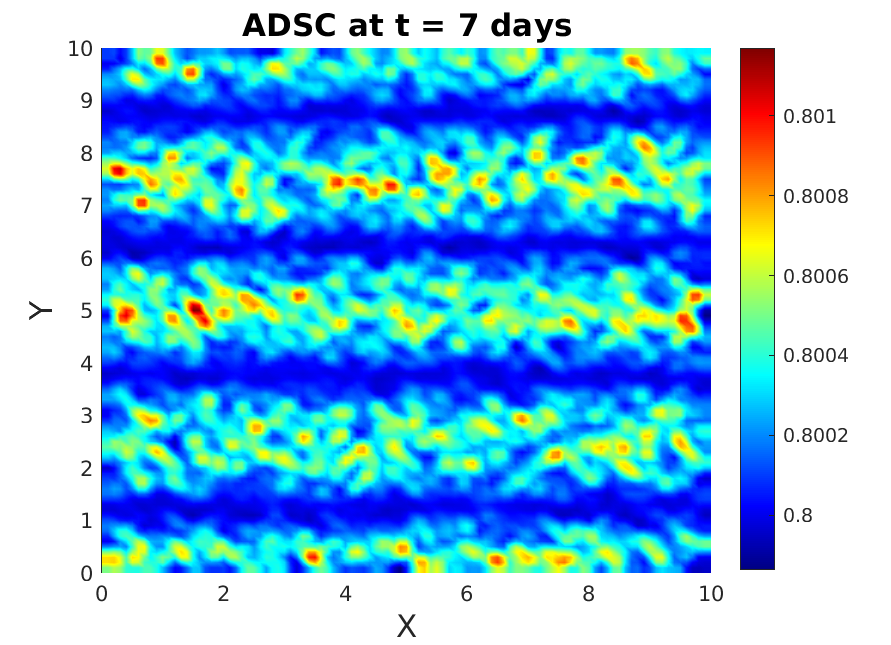}
		\caption{ADSCs at 7 days}
		\label{}
	\end{subfigure}
	\begin{subfigure}{0.32\textwidth}
		\centering
		\includegraphics[width=\textwidth]{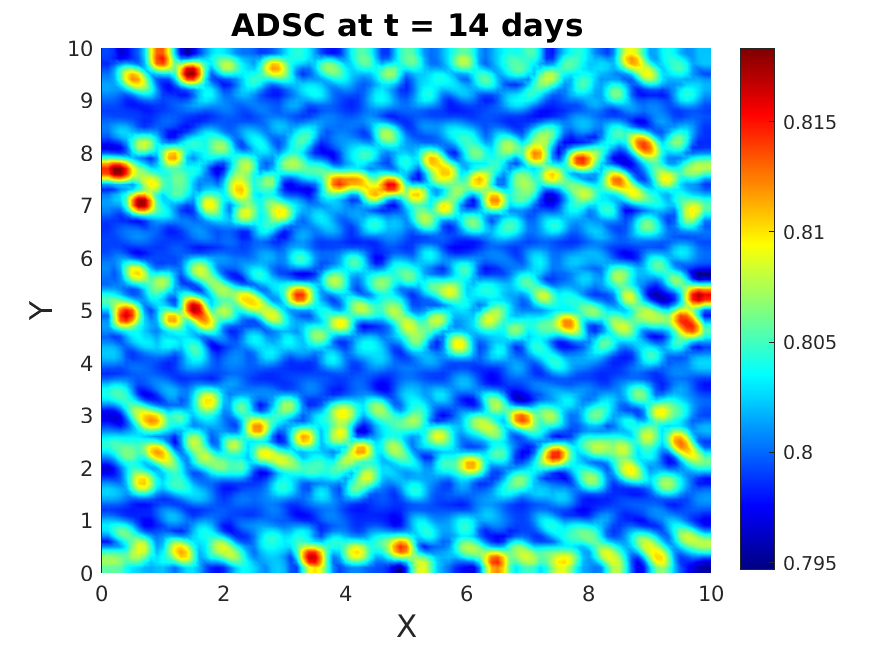}
		\caption{ADSCs at 14 days}
	\end{subfigure}	
	\begin{subfigure}{0.32\textwidth}
		\centering
		\includegraphics[width=\textwidth]{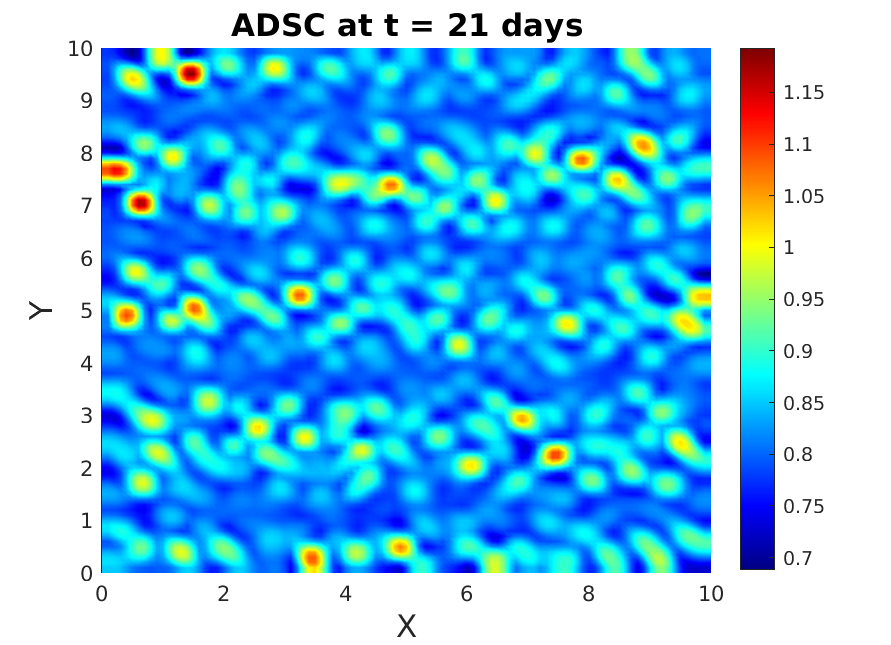}
		\caption{ADSCs at 21 days }
	\end{subfigure}	\\[1ex]
	\begin{subfigure}{0.32\textwidth}
		\centering
		\includegraphics[width=\textwidth]{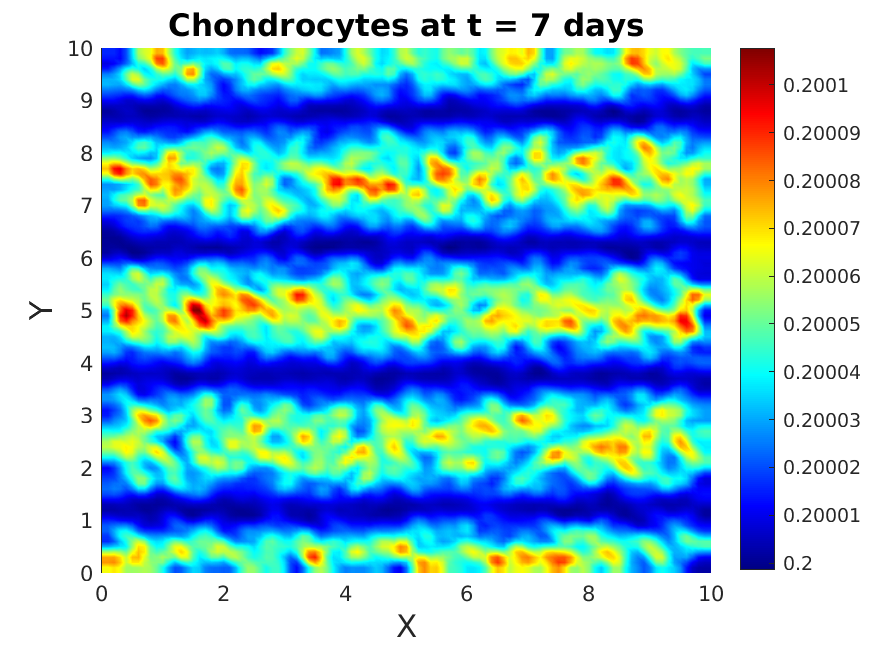}
		\caption{Chondrocytes at 7 days}
		\label{}
	\end{subfigure}
	\begin{subfigure}{0.32\textwidth}
		\centering
		\includegraphics[width=\textwidth]{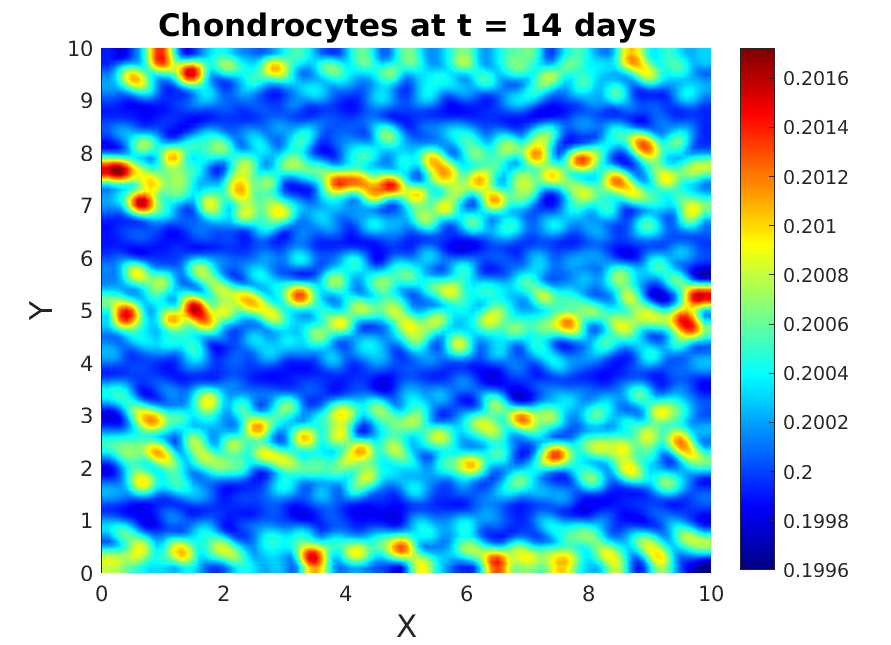}
		\caption{Chondrocytes at 14 days}
	\end{subfigure}	
	\begin{subfigure}{0.32\textwidth}
		\centering
		\includegraphics[width=\textwidth]{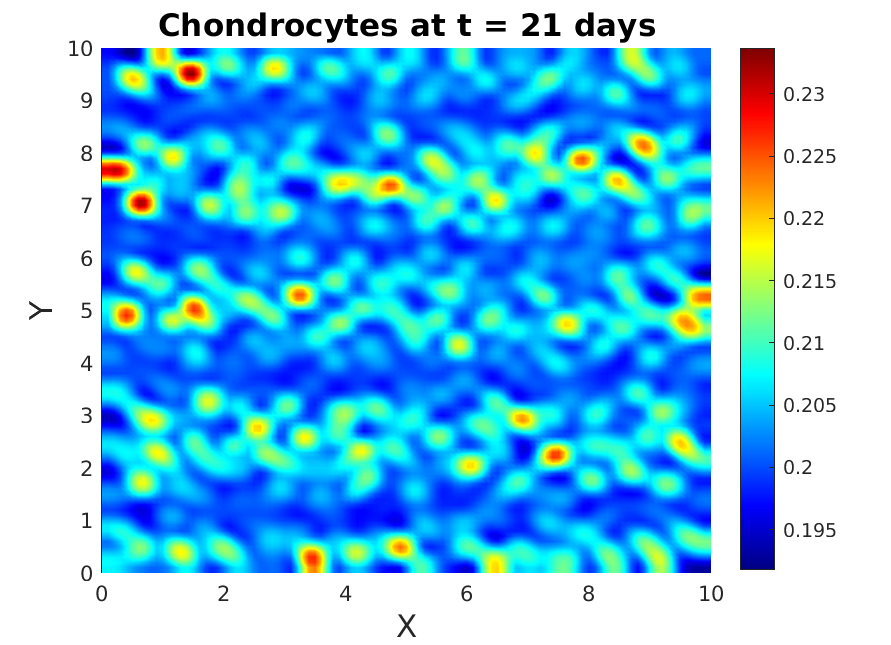}
		\caption{Chondrocytes at 14 days }
	\end{subfigure}\\[1ex]
	\begin{subfigure}{0.32\textwidth}
		\centering
		\includegraphics[width=\textwidth]{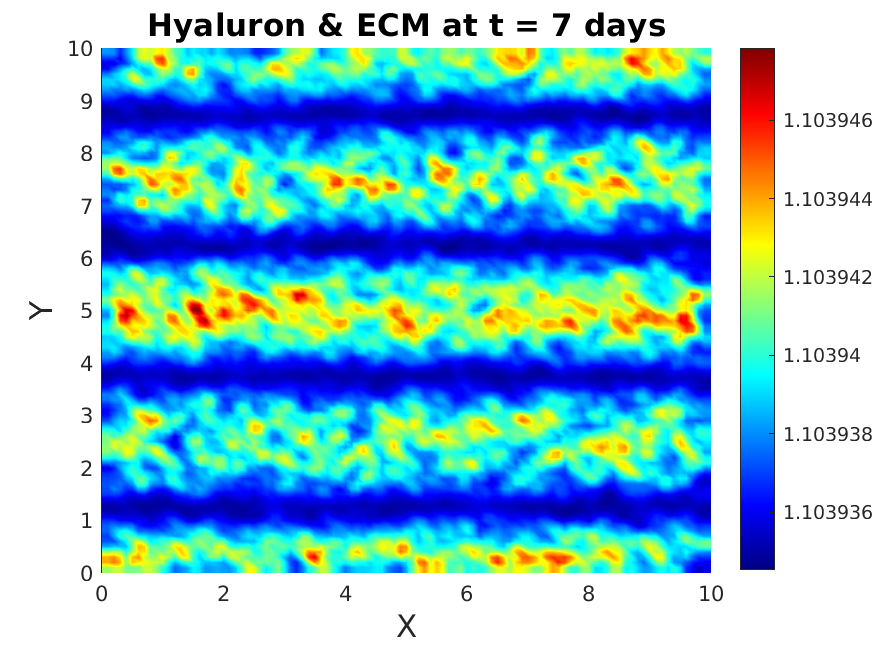}
		\caption{Hyaluron \& ECM at 7 days }
		\label{}
	\end{subfigure}
	\begin{subfigure}{0.32\textwidth}
		\centering
		\includegraphics[width=\textwidth]{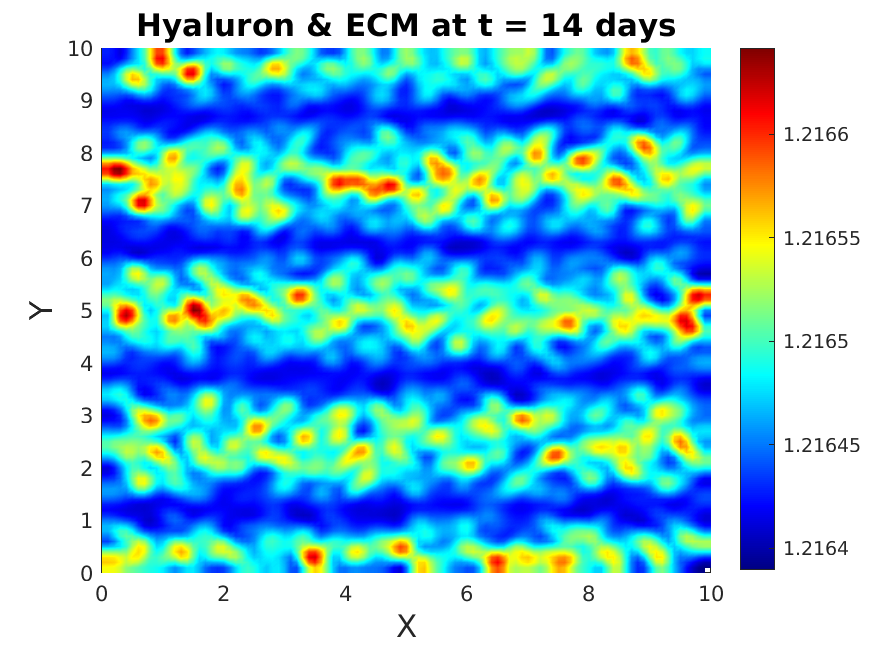}
		\caption{Hyaluron \& ECM at 14 days }
	\end{subfigure}	
	\begin{subfigure}{0.32\textwidth}
		\centering
		\includegraphics[width=\textwidth]{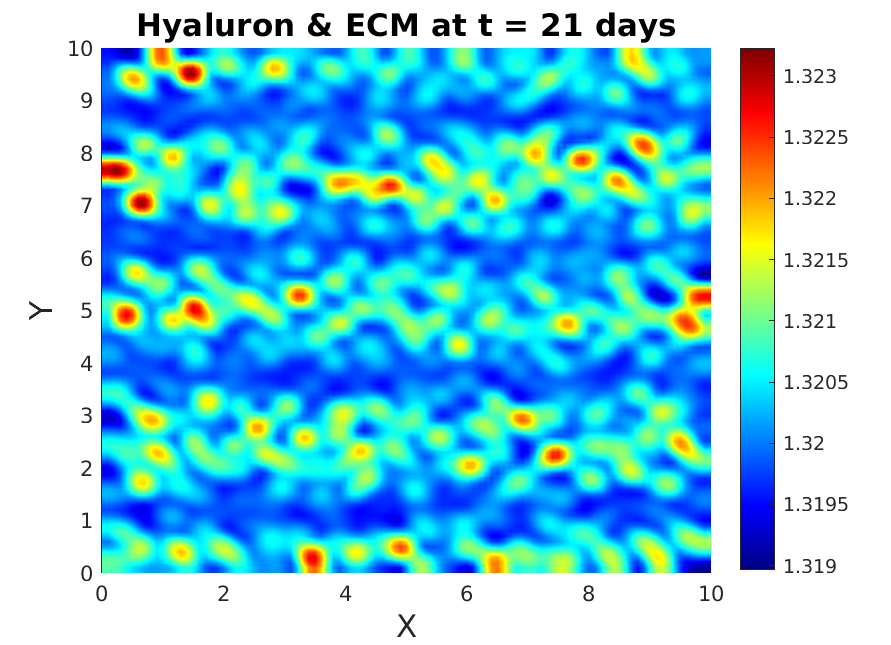}
		\caption{Hyaluron \& ECM at 21 days }
	\end{subfigure}
	\caption[]{ADSC, chondrocyte, and hyaluron/ECM density at 7, 14, and 21 days, initial conditions \eqref{eq:ICs-2D-cos}, $b>b_c$ (here $b=3.7$)}\label{fig:7}
\end{figure}

\appendix

\section{Appendix}

\noindent	
The subsequent lemma closely follows Lemma A in the appendix of \cite{pang2017global}:
\begin{lemma}
	\label{lem_app}
	Assuming $0 < T < 1$, $c_{2j} \in C^{\frac{6}{5}, \frac{3}{5}}(Q_T)$, $c_{2j} \geq 0$, $h_j \in C(0, T; W^{1, 5}(\Omega))$ and $z_j \in C^{1, 0}(\bar{Q}_T) (j = 1, 2)$ and $\|z_j\|_{C^{1, 0}(\bar{Q}_T)} \leq M, \|h_j\|_{ C(0, T; W^{1, 5}(\Omega))} \leq M$ and $\|c_{2j}\|_{C^{\frac{6}{5}, \frac{3}{5}}(Q_T)} \leq M$. Then the solutions $\bar{h}_1$ and $\bar{h}_2$ of the ordinary differential equations
	\begin{equation}
		\label{app1}
		\partial_t \bar{h}_{j} = - \gamma_1 \bar{h}_j c_{2j} + \gamma_2\frac{c_{2j}}{K_{c_2} + c_{2j}}, \quad h_j(x, 0) = h_0(x), \quad j = 1, 2
	\end{equation}
	satisfy
	\begin{equation}
		\label{app3}
		\|\bar{h}_1 - \bar{h}_2\|_{C(0, T; W^{1, 5}(\Omega))} \leq T C(M) (\|(c_{21} - c_{22})\|_{C^{1, 0}(Q_T)} + \|z_1 - z_2\|_{C^{1, 0}(Q_T)} + \|h_1 - h_2\|_{C(0, T; W^{1,5}(\Omega))} ),
	\end{equation}
	where $C(M)$ is a constant depending only on $M$.
\end{lemma}

\begin{proof}
	We can directly have the following ODE from \eqref{app1}
	\begin{equation}
		\label{app2}
		\partial_t (\bar{h}_1 - \bar{h}_2) = w_1(\bar{h}_1 - \bar{h}_2) + w_2, \quad \bar{h}_1(x, 0) - \bar{h}_2(x, 0) = 0
	\end{equation}
	where $w_1 = -\gamma_1 c_{21}, ~ w_2 = (c_{21} - c_{22})\left\{- \gamma_1 \bar{h}_2 + \frac{\gamma_2 K_{c_1}}{(K_{c_1} + c_{21})(K_{c_1} + c_{22})}\right\}$. Solving \eqref{app2}
	\begin{equation*}
		(\bar{h}_1 - \bar{h}_2)(x, t) = \int_0^t e^{\int_s^t w_1(x, \tau) d \tau} w_2(x, s)ds
	\end{equation*}
	and
	\begin{equation*}
		(\nabla \bar{h}_1 - \nabla \bar{h}_2)(x, t) = \int_0^t e^{\int_s^t w_1(x, \tau) d \tau} \nabla w_2(x, s)ds + \int_0^t e^{\int_s^t w_1(x, \tau)d\tau}w_2(x, s) \int_s^t \nabla w_1(x, \tau) d\tau ds.
	\end{equation*}
	We can find a $C(M)$ such that, $\|w_1\|_{C(Q_T)} \leq C(M)$, also
	\begin{align*}
		& \|\nabla w_1\|_{C(0, T; L^5(\Omega))} = \gamma_1 \|\nabla c_{21}\|_{C(0, T; L^5(\Omega))} \leq C(M),                                                                      \\[5pt]
		\|w_2\|_{C(Q_T)} & = \left\|- \gamma_1 \bar{h}_2(c_{21} - c_{22}) + \tfrac{\gamma_2 K_{c_1}(c_{21} - c_{22})}{(K_{c_1} + c_{21})(K_{c_1} + c_{22})}\right\|_{C(Q_T)} \leq C(M)\|(c_{21} - c_{22})\|_{C(Q_T)}       \\[5pt]
		& \leq C(M)(\|(c_{21} - c_{22})\|_{C(Q_T)} + \|z_1 - z_2\|_{C(Q_T)} + \|h_1 - h_2\|_{C(0, T; W^{1, 5}(\Omega))} )
	\end{align*}
	and
	\begin{align*}
		\|\nabla w_2\|_{C(0, T; L^5(\Omega))} = \Big\|& - \gamma_1 (c_{21} - c_{22}) \nabla \bar{h}_2 - \gamma_1 \bar{h}_2 (\nabla c_{21} - \nabla c_{22}) + \tfrac{\gamma_2 K_{c_1}}{(K_{c_1} + c_{21})(K_{c_1} + c_{22})}(\nabla c_{21} - \nabla c_{22}) \\
		& - \tfrac{\gamma_2 K_{c_1}(c_{21}-c_{22})}{(K_{c_1} + c_{21})^2(K_{c_1} + c_{22})}\nabla c_{21} -\tfrac{\gamma_2 K_{c_1}(c_{21}-c_{22})}{(K_{c_1} + c_{21})(K_{c_1} + c_{22})^2}\nabla c_{22}\Big\|_{C(0, T; L^5(\Omega))}\\
		& \leq C(M)(\|(c_{21} - c_{22})\|_{C^{1, 0}(Q_T)} \\
		& \leq C(M)(\|(c_{21} - c_{22})\|_{C^{1, 0}(Q_T)} + \|z_1 - z_2\|{C^{1, 0}(Q_T)} + \|h_1 - h_2\|_{C(0, T; W^{1, 5}(\Omega))} )
	\end{align*}
	Now, as in \cite{pang2017global}, we can also have the following two estimates
	\begin{align*}
		\|\bar{h}_1(\cdot, t) - \bar{h}_2(\cdot, t)\|_{L^5(\Omega)} \leq T C(M) (\|(c_{21} - c_{22})\|_{C(Q_T)} + \|z_1 - z_2\|_{C(Q_T)} + \|h_1 - h_2\|_{C(0, T;  W^{1, 5}(\Omega))} )
	\end{align*}
	and
	\begin{align*}
		\|\nabla \bar{h}_1(\cdot, t) - \nabla \bar{h}_2(\cdot, t)\|_{L^5(\Omega)} \leq T C(M) \big(\|(c_{21} - c_{22})\|_{C^{1, 0}(Q_T)} & + \|z_1 - z_2\|_{C^{1, 0}(Q_T)}\\
		& + \|h_1 - h_2\|_{C(0, T;  W^{1, 5}(\Omega))} \big) .
	\end{align*}
\end{proof}

		\subsection*{Acknowledgement} CS and NM acknowledge funding by the German Research Foundation DFG within the SPP 2311. SCM was funded by DAAD, MathApp, and the Nachwuchsring at RPTU Kaiserslautern-Landau.
%
\phantomsection
\printbibliography	


	\end{document}